\documentclass[reqno,11pt,a4paper]{amsart}
\usepackage[utf8]{inputenc}
\usepackage{enumerate,enumitem}
\usepackage{textcomp}

\usepackage{pgffor}

\usepackage{esint}
\usepackage{pdfsync}
\usepackage[colorlinks=true,linkcolor=blue]{hyperref}%
\usepackage[english]{babel}

\usepackage{ulem}
\usepackage{amsmath,esint,color}
\usepackage{amsthm,latexsym,epsfig,graphicx,marvosym, mathrsfs,bigints,stmaryrd,textgreek,upgreek}
\usepackage{relsize}
\usepackage{amsfonts}
\usepackage{amssymb}
\usepackage[headings]{fullpage}

\usepackage{tikz}
\usetikzlibrary{decorations.pathreplacing}

\allowdisplaybreaks

\newcommand{\TT}{\mathbb{T}}

\newcommand{\RR}{\mathbb{R}}

\newtheorem{theorem}{Theorem}
\newtheorem{proposition}{Proposition}
\newtheorem{lemma}{Lemma}
\newtheorem{remark}{Remark}

\newtheorem{definition}{Definition}

\newcommand{\dd}{\mathrm{d}}
\newcommand{\vertiii}[1]{{\vert\kern-0.25ex\vert\kern-0.25ex\vert #1
\vert\kern-0.25ex\vert\kern-0.25ex\vert}}

\begin{document}

\title[Local regularity and finite-time singularity for generalized SQG patches]
{Local regularity and finite-time singularity for a class of generalized SQG patches on the half-plane}
\author{Qianyun Miao}
\address{School of Mathematics and Statistics, Beijing Institute of Technology, Beijing 100081, P. R. China}
\email{qianyunm@bit.edu.cn}
\author{Changhui Tan}
\address{Department of Mathematics, University of South Carolina, Columbia SC 29208, USA}
\email{tan@math.sc.edu}
\author{Liutang Xue}
\address{School of Mathematical Sciences, Laboratory of Mathematics and Complex Systems (MOE),
Beijing Normal University, Beijing 100875, P.R. China}
\email{xuelt@bnu.edu.cn}
\author{Zhilong Xue}
\address{School of Mathematical Sciences, Laboratory of Mathematics and Complex Systems (MOE),
Beijing Normal University, Beijing 100875, P.R. China}
\email{zhilongxue@mail.bnu.edu.cn}

\date{\today}


\begin{abstract}
In this paper, we investigate a class of inviscid generalized surface quasi-geostrophic (SQG) equations on the half-plane with a rigid boundary. Compared to the Biot-Savart law in the vorticity form of the 2D Euler equation, the velocity formula here includes an additional Fourier multiplier operator $m(\Lambda)$. When $m(\Lambda) = \Lambda^\alpha$, where $\Lambda = (-\Delta)^{1/2}$ and $\alpha\in (0,2)$, the equation reduces to the well-known $\alpha$-SQG equation. Finite-time singularity formation for patch solutions to the $\alpha$-SQG equation was famously discovered by Kiselev, Ryzhik, Yao, and Zlato\v{s} \cite{KRYZ}.

We establish finite-time singularity formation for patch solutions to the generalized SQG equations under the Osgood condition
\begin{align*}
  \int_2^\infty \frac{1}{r (\log r) m(r)} \dd r < \infty
\end{align*}
along with some additional mild conditions. 
Notably, our result fills the gap between the globally well-posed 2D Euler equation ($\alpha = 0$) and the $\alpha$-SQG equation ($\alpha > 0$). Furthermore, in line with Elgindi's global regularity results for 2D Loglog-Euler type equations \cite{Elg14}, 
our findings suggest that the Osgood condition serves as a sharp threshold that distinguishes global regularity and finite-time singularity in these models.

In addition, we generalize the local regularity and finite-time singularity results for patch solutions to the $\alpha$-SQG equation, as established by Gancedo and Patel \cite{GanP21}, 
extending them to cases where $m(r)$ behaves like $r^\alpha$ near infinity but does not have an explicit formulation.
\end{abstract}

\maketitle

\tableofcontents

\section{Introduction}
In this paper, we study a family of inviscid generalized surface quasi-geostrophic (SQG) equations
\begin{align}
  \partial_t\theta + u\cdot \nabla\theta &= 0, \;\; \quad \qquad \qquad\qquad\quad (t,x)\in \mathbb{R}_+\times \mathbf{D},
  \label{eq:geSQG}\\
  u & = \nabla^\perp (-\Delta)^{-1} m(\Lambda) \theta, \;\;\;\; (t,x)\in \mathbb{R}_+\times \mathbf{D}, \label{eq:geSQG-u} \\
  \theta|_{t=0}  & = \theta_0, \;\;\qquad \qquad \qquad\quad \quad x\in \mathbf{D}, \label{eq:gesQG-data}
\end{align}
where $\mathbf{D}$ is either the whole space $\mathbb{R}^2$ or the half-plane
$\mathbb{R}^2_+ \triangleq \mathbb{R}\times \mathbb{R}_+$,
$\nabla^\perp \triangleq (-\partial_{x_2}, \partial_{x_1})$, $\Lambda \triangleq (-\Delta)^{\frac{1}{2}}$.
When $\mathbf{D} =\mathbb{R}^2_+$, we assume the no-penetration boundary condition
\begin{align}\label{eq:bdr}
  u_2 = 0, \quad\qquad (t,x)\in \mathbb{R}_+\times \partial \mathbb{R}^2_+.
\end{align}
The vector $u=(u_1,u_2)$ is the velocity field and the scalar field $\theta$ can be interpreted as the vorticity or density or temperature of the fluid.
The operator $m(\Lambda)$ is a Fourier multiplier with the symbol
$m(\xi)= m(|\xi|)$, which is a radial function of $\RR^2$ satisfying the following hypotheses:
\begin{enumerate}[label=$(\mathbf{H}\theenumi)$,ref=$\mathbf{H}\theenumi$]
\item\label{H1}
$m(r)\in C^5(\RR_+)$ and
\begin{align}\label{cond:m0}
  \forall\; r> 0,\;\;m(r)> 0,\;\; m'(r)\geq0,\;\;
  \textrm{and}\;\;\lim\limits_{r\to 0^{+}}m(r)<\infty,\;\;\lim\limits_{r\to 0^{+}}r\,m'(r) <\infty,
\end{align}
and $m'(r)$ satisfies the Mikhlin-H\"ormander condition, that is, there exists a constant $C>0$ such that
\begin{align}\label{eq:MH}
  \big|\tfrac{\dd^k}{\dd r^k} m'(r)\big| \leq C \frac{m'(r)}{r^k}\,,\quad \forall\; k=1,2,3,4,\,\forall\; r>0.
\end{align}
\end{enumerate}
The velocity formula \eqref{eq:geSQG-u}, known as the Biot-Savart law, can be expressed as:
\begin{equation}\label{eq:u-exp}
  u(x,t) =
  \begin{cases}
  \int_{\mathbb{R}^2} \frac{(x-y)^\perp}{|x-y|^2}G(|x-y|) \theta(y,t) \dd y,
  \quad & \textrm{if}\;\; \mathbf{D} = \mathbb{R}^2, \\[2mm]
  \int_{\mathbb{R}^2_+} \Big(\frac{(x-y)^\perp}{|x-y|^2}G(|x-y|)
  -\frac{(x-\overline{y})^\perp}{|x-\overline{y}|^2}G(|x-\overline{y}|)\Big)\theta(y,t) \dd y,
  \quad & \textrm{if}\;\; \mathbf{D} = \mathbb{R}^2_+,
  \end{cases}
\end{equation}
where $x^\perp \triangleq (x_2,-x_1)$, $\overline{x} \triangleq (x_1,-x_2)$, and the kernel $G(\cdot)$ is a continuously differentiable function on $(0,+\infty)$ given by \eqref{eq:G-exp1} below.


Considering different forms of $m$ satisfying \eqref{H1}, the equation \eqref{eq:geSQG}-\eqref{eq:gesQG-data} can include several important hydrodynamic models as special cases:
\begin{itemize}
\item $m(r)\equiv1$. In this case, $m(\Lambda) \equiv \mathrm{Id}$ and $G \equiv \frac{1}{2\pi}$. The equation \eqref{eq:geSQG}-\eqref{eq:gesQG-data} becomes the 2D Euler equation in the vorticity form,
which describes the motion of 2D inviscid incompressible fluid and is a fundamental model in fluid dynamics.
\item $m(r)=r^\alpha$, with $\alpha\in(0,2)$. We have
\begin{equation}\label{eq:alphaSQG}
  m(\Lambda) = (-\Delta)^{\frac{\alpha}{2}} = \Lambda^\alpha,\;\; \textrm{and}\;\;
  G(\rho)= \alpha \mathrm{c}_{\alpha}\rho^{-\alpha},\;\; \textrm{with}\;\; 
  \mathrm{c}_{\alpha}=\tfrac{\Gamma(\alpha/2)}{\pi 2^{2-\alpha}\Gamma(1- \alpha/2)}.
\end{equation}
In this case, the equation \eqref{eq:geSQG}-\eqref{eq:gesQG-data} reduces to the inviscid $\alpha$-SQG equation.
For $\alpha=1$, the $\alpha$-SQG equation is the well-known SQG equation, which is a simplified model for tracking atmospheric circulation near the tropopause \cite{HPGS} and ocean dynamics in the upper layers \cite{LK06}. The $\alpha$-SQG equation with $0<\alpha <1$ was introduced by C\'ordoba, Fontelos, Mancho, and Rodrigo \cite{CFMR} as a class of models interpolating between the 2D Euler equation and the SQG equation.
\item  $m(r)=\frac{r^2}{r^2+\varepsilon^2}$, with $\varepsilon>0$. In this case,
\begin{align*} 
  m(\Lambda) = \tfrac{-\Delta}{-\Delta + \varepsilon^2},\;\;\textrm{and}\;\; 
  G(\rho) = \tfrac{1}{2\pi} \varepsilon r\mathbf{K}_0'(\varepsilon \rho),
\end{align*}
where $\mathbf{K}_0$ is the modified Bessel function (see \cite{DHR19}). The equation  \eqref{eq:geSQG}-\eqref{eq:gesQG-data} corresponds to the quasi-geostrophic shallow-water (QGSW) equation. This model is derived asymptotically from the rotating shallow water equations in the limit of fast rotation and small variation of free surface \cite{Val08}.
\item $m(r)=\log^\beta(1+\log(1+r^2))$, with $\beta\in[0,1]$. We have
\begin{equation}\label{eq:loglogEuler}
    m(\Lambda) = \log^\beta(1+\log(1-\Delta)).
\end{equation}
In this case, the equations \eqref{eq:geSQG}-\eqref{eq:gesQG-data} correspond to what is typically referred to as the 2D Loglog-Euler equation. This model was introduced by Chae, Constantin, and Wu \cite{CCW11} as a more general framework connecting the 2D Euler and $\alpha$-SQG equations.
\end{itemize}

\subsection{Global regularity versus finite-time singularity}
The global well-posedness of classical solutions for the 2D Euler equation in the whole space $\RR^2$, the half-plane $\RR^2_+$, or bounded smooth planer domains is well-known,
see for instance \cite{MA02,MP94,KL84,JLZ23}.
However, for the $\alpha$-SQG equation in $\RR^2$ with $\alpha\in (0,2)$, the global well-posedness of smooth solutions remains an open problem.

Local well-posedness for the $\alpha$-SQG equation in Sobolev/H\"older spaces was established in \cite{Wu05,CCCG12} for the whole space, and in \cite{ConN18a} for bounded smooth domains. Some ill-posedness results in Sobolev/H\"older spaces can be found in \cite{Zlat15,Cor-Zo21,Cor-Zo22,JeoK24,CJK24}.
In a recent work \cite{Zlat23}, Zlato\v{s} proved that the $\alpha$-SQG equation on the half-plane is locally well-posed for $\alpha \in (0, \frac{1}{2}]$ and ill-posed for $\alpha > \frac{1}{2}$ in anisotropic weighted spaces (see also \cite{JKY24} for similar ill/well-posedness results). Moreover, he demonstrated that smooth initial data can lead to finite-time singularity in the regime $\alpha \in (0, \frac{1}{2}]$.

Regarding weak solutions to the $\alpha$-SQG equation, global existence was established in works such as \cite{Resnick,Mar08,CCCG12,ConN18b,NHQ18,LX19}, while non-uniqueness was shown in \cite{BSV19,CKL21,Isett21}. 

Chae, Constantin and Wu \cite{CCW11} introduced the Loglog-Euler equation \eqref{eq:geSQG}-\eqref{eq:gesQG-data}, 
where the Fourier multiplier $m$ is slightly more singular than the 2D Euler equation, 
but less singular than the $\alpha$-SQG equation. An intriguing question arising from their work is whether the Loglog-Euler equation 
exhibits global regularity, like the Euler equation, or develops a finite-time singularity, similar to the $\alpha$-SQG equation.  
They showed global regularity when $\mathbf{D}=\RR^2$, assuming the parameter $\gamma\leq1$ in \eqref{eq:loglogEuler}. 
Later, Elgindi \cite{Elg14} and Dabkowski et al. \cite{DKSV14} independently extended the global regularity result of \cite{CCW11} 
to more general multipliers $m(\Lambda)$ in \eqref{eq:geSQG-u}, where $m$ satisfies an Osgood-type condition
\begin{align}\label{eq:m-elgindi}
  \int^{+\infty}_2 \frac{1}{r(\log r)m(r)}\dd r = +\infty,
\end{align}
along with other mild assumptions. It remains unknown whether this Osgood-type condition is \textit{sharp}, 
namely whether the violation of condition \eqref{eq:m-elgindi} leads to finite-time singularity, 
and this is one of the main focuses of this paper.

We would like to highlight a related result for a 1D Burgers-type equation where the question of global regularity versus finite-time singularity has been resolved. The equation is given by:
\begin{align}\label{eq:BurgTEq}
  \partial_t \theta - \theta\, \partial_x \theta + \frac{\Lambda}{m(\Lambda)} \theta = 0,
  \quad \theta|_{t=0} = \theta_0,
\end{align}
where $x\in \mathbb{T}$, and $\theta$ is a scalar. Dabkowski, Kiselev, Silvestre and Vicol \cite{DKSV14} proved that under an Osgood-type condition (different than \eqref{eq:m-elgindi})
\begin{align}\label{eq:cd-DKSV}
 \int_2^{+\infty} \frac{1}{r\, m(r)} \dd r = +\infty,
\end{align}
and some mild assumptions, the equation \eqref{eq:BurgTEq} with smooth initial data $\theta_0$ 
generates a unique global-in-time smooth solution. Conversely, if \eqref{eq:cd-DKSV} is violated, 
there exists a smooth initial datum $\theta_0$ such that the solution to 
\eqref{eq:BurgTEq} develops a finite-time singularity, in the sense that 
$\lim\limits_{t\rightarrow T} \|\partial_x \theta\|_{L^\infty} = +\infty$ at some finite time $T$.
Our goal is to establish a similar argument regarding the sharpness of condition \eqref{eq:m-elgindi} for inviscid generalized SQG equations.

\subsection{Patch solutions}
In recent decades, there has been significant interest and intense study on a class of non-smooth solutions called \textit{patch solutions}. These are weak solutions of the transport equation \eqref{eq:geSQG} associated with patch-like initial data, where the initial condition $\theta|_{t=0} = \theta_0$ is given by either a single patch ($N=1$) or multiple patches ($N>1$): 
\begin{equation}\label{eq:patch-data}
  \theta_0(x)=\sum_{j=1}^N a_j \mathbf{1}_{D_j}(x),
  \quad a_j\in \mathbb{R},\quad \mathbf{1}_{D_j}(x)
  \triangleq \begin{cases}
     1,\quad x\in D_j, \\
     0,\quad x\in \mathbf{D} \setminus D_j,
  \end{cases}
\end{equation}
where $D_j$ ($j=1,\cdots,N$) are disjoint simply connected bounded domains with regular boundaries $\partial D_j$ on $\mathbf{D}$. 

For the 2D Euler equation, Yudovich \cite{Yud63} proved the global existence and uniqueness of weak solutions associated with initial data $\theta_0\in L^1\cap L^\infty(\mathbf{D})$, which implies that the corresponding patch solutions are globally well-defined and keeping the patch structure during evolution. However, Yudovich's result does not provide sufficient regularity information on the evolved patch boundaries.

The \textit{vorticity patch problem} for the 2D Euler equation concerning whether the smoothness of patch boundaries $\partial D_j$ can be globally persisted was raised in the 1980s. Chemin \cite{chemin1993} resolved this question by proving the global persistence of $C^{k,\gamma}$  patch boundaries in whole space with $k\in \mathbb{N}^{\star}$ and $0<\gamma<1$.
Alternative proofs of the same result were provided by Bertozzi and Constantin \cite{bertozzi1993}, and Serfati \cite{Serf94}.

Kiselev and Luo \cite{KL23a} showed the ill-posedness of $C^2$ vortex patches, while proving the global 
$W^{2,p}$-regularity persistence of patch boundaries with $p\in (1,+\infty)$. 
In the half-plane $\mathbf{D}=\RR^2_{+}$ and in smooth bounded domain $\mathbf{D}$,  
Kiselev et al. \cite{KRYZ,Kiselev19} established the similar global $C^{1,\gamma}$-regularity persistence of patch boundaries, 
allowing that the patch boundaries touch the rigid boundary $\partial \mathbf{D}$ (see \cite{Depa99} for previous work in the half-plane).

For the Loglog-Euler equation with $m(\cdot)$ satisfying \eqref{eq:loglogEuler} with $\beta\in[0,1]$, or in general, under the Osgood-type condition \eqref{eq:m-elgindi}, the velocity field for the vorticity patch problem is no longer Lipschitz. As a result, there will be an $\varepsilon$-regularity loss in the evolution. In the whole space $\mathbf{D}=\RR^2$, Elgindi \cite{Elg14} proved the global $C^{1,\gamma-\varepsilon}$ regularity (for arbitrarily small $\varepsilon>0$) of the evolved patch boundaries associated with the initial $C^{1,\gamma}$ patches.

For the patch solutions of the $\alpha$-SQG equation (also called \textit{sharp fronts} in the literature \cite{Gill82}), the situation becomes more intricate due to the velocity field $u$ being only H\"older continuous for $\alpha\in(0,1)$ and not even continuous for $\alpha \in [1,2)$.
However, the normal direction of the velocity field is well-defined. Using the contour dynamics equation, Rodrigo et al. \cite{Rodr05,CFMR} showed the local existence and uniqueness of $C^\infty$ patches for the $\alpha$-SQG equation with $\alpha\in (0,1]$, applying the Nash-Moser iteration.
Through using the cancellations of the curve structures, the local well-posedness of patch solutions for the whole-space  $\alpha$-SQG equation in the $L^2$-based Sobolev spaces $X$ were established in a series of works: 
Gancedo \cite{Gan08} for $\alpha\in (0,1]$, $X= H^n$, $n\geq 3$ (with uniqueness later proved by Kiselev et al. \cite{KYZ17} and C\'ordoba et al. \cite{CCG18});
Chae et al. \cite{CCCG12} for $\alpha\in (1,2)$, $X= H^n$, $n\geq 4$; 
Gancedo and Patel \cite{GanP21} for $X = H^2$ if $\alpha\in(0,1)$ and for $X = H^3$ if $\alpha \in[1,2)$;
and Gancedo et al. \cite{GNP22} for $\alpha=1$ and $X=H^{2+s}$, $s\in (0,1)$. 
Additionally, Kiselev and Luo \cite{KL23} proved the strong ill-posedness of patch solutions for $\alpha$-SQG equation with $\alpha\in(0,1)$ in H\"older space $C^{2,\gamma}$ with $\gamma\in (0,1)$ and Sobolev space $W^{2,p}$ with $p\neq 2$.
The exclusion of splash-like singularities in $\alpha$-SQG patches with $\alpha \in (0,1]$ was addressed in works such as \cite{GanS14, KL23b, JZ24}.

The $\alpha$-SQG equation on the half-plane $\RR^2_+$ with rigid boundary is also of significant interest. Kiselev, Ryzhik, Yao and Zlato\v{s} \cite{KRYZ,KYZ17} showed the local well-posedness of patch solutions for $\alpha\in (0,\frac{1}{12})$, associated with patch-like data \eqref{eq:patch-data}. Moreover, they constructed symmetric patch-like data that leads to finite-time singularity formation in the same range of $\alpha$. 
Subsequently, Gancedo and Patel \cite{GanP21} extended these results to the range $\alpha\in(0,\tfrac{1}{3})$ (see also \cite{Zlat23} for the recent improvement). 
These results, together with the global well-posedness of the 2D half-plane Euler equation in \cite{KRYZ}, indicate that $\alpha=0$ is a critical index, marking a phase transition in the behavior of solutions.

We also mention that, although the global well-posedness problem for smooth or patch solutions of the $\alpha$-SQG equation remains open in general, various non-trivial global-in-time (patch) solutions have been established for both the Euler and $\alpha$-SQG equations. See for instance \cite{ADDMW,BHM23,Burb82,CCG16,CCG16b,CCG20,CkJS24,CGI19,DHH16,DHHM,DHMV16,GHS20,GS19,GS-IP,HH15,HHM23,HHM23b,HMW20,HW22,HM17,HMV13,HXX23a} and reference therein.

\subsection{Main results}
The main goal of this paper is to develop a local well-posedness theory and demonstrate finite-time singularity formation for patch solutions to the generalized SQG equation \eqref{eq:geSQG}-\eqref{eq:bdr} on the half-plane $\mathbf{D} = \mathbb{R}^2_+$ with patch-like initial data \eqref{eq:patch-data}. We focus on two types of multipliers $m$ that satisfy the hypotheses \eqref{H1}.

The first regime of consideration is the critical case when $\alpha=0$, where we assume the multiplier $m(r)$ has a sub-algebraic growth:
\begin{enumerate}[label=$(\mathbf{H}2a)$,ref=$\mathbf{H}2a$]
\item\label{H2a} there exists a constant $\gamma\geq 0$ such that $m$ satisfies that
\begin{align}\label{cond:m1}
  \lim\limits_{r\rightarrow +\infty} m(r)= +\infty,\quad
  \lim_{r\rightarrow +\infty} \frac{r\, (\log r) m'(r)}{m(r)} = \gamma,\quad
  \lim_{r \rightarrow +\infty} \frac{r\, m''(r)}{m'(r)} = -1.
\end{align}
\end{enumerate}

Let us list several examples of the multiplier $m$ that satisfy conditions \eqref{H1}-\eqref{H2a}:
\begin{equation}\label{eq:mex}
  m_1(r)=\log^{\beta_1}(1+r),\quad
  m_2(r) = \log^{\beta_2}\log(e+r),\quad
  m_3(r)= \log\log (e+r)\, \log^{\beta_3} \log\log\,(e^2 + r),    
\end{equation}
where $\beta_1, \beta_2, \beta_3>0$, and the constant $\gamma=\beta_1, 0, 0$, respectively. In particular, the Loglog-Euler equation \eqref{eq:loglogEuler} falls into this category.

In light of Elgindi’s work \cite{Elg14} on the global regularity of patch solutions for the 2D Loglog-Euler type equation in the whole space, the Osgood-type condition \eqref{eq:m-elgindi} is known to be a sufficient condition for ensuring global regularity. Our aim is to demonstrate that this condition is also necessary. Specifically, we show that if \eqref{eq:m-elgindi} is violated, finite-time singularity formation will occur for the patch solutions.

\begin{theorem}\label{thm:blow-upa}
  Consider the inviscid generalized SQG equation \eqref{eq:geSQG}-\eqref{eq:bdr} in  $\mathbf{D} = \mathbb{R}^2_+$.
  Assume that $m(\xi) = m(|\xi|)$ is a radial function satisfying \eqref{H1}-\eqref{H2a}, and the following Osgood condition:
\begin{align}\label{cond:m1c}
  \int_2^{+\infty} \frac{1}{r (\log r) m(r)} \dd r < +\infty.
\end{align}
Then there exist non self-intersecting initial patch data $\theta_0$ given by \eqref{eq:patch-data} such that the half-plane inviscid generalized SQG equation \eqref{eq:geSQG}-\eqref{eq:bdr} generates a unique local-in-time $H^2$ patch solution $\theta$ that develops a singularity in finite time.
\end{theorem}

\begin{remark}[Criticality of the Osgood condition]\label{rmk:crit-cond}
The Osgood condition \eqref{cond:m1c} is precisely the opposite of \eqref{eq:m-elgindi}. Therefore, Theorem \ref{thm:blow-upa} demonstrates that the Osgood condition serves as the critical threshold marking a sharp phase transition between global regularity and finite-time blowup for the patch solutions to the generalized SQG equation \eqref{eq:geSQG}-\eqref{eq:bdr}. 

For the examples in \eqref{eq:mex}, the Osgood condition \eqref{cond:m1c} holds when $\beta_1 > 0$, $\beta_2 > 1$, and $\beta_3 > 1$, respectively. Combined with the global well-posedness results for the 2D Euler and Loglog-Euler equations, we obtain critical exponents $\beta_1 = 0$ and $\beta_2 = 1$ that distinguish the global behavior of the solutions.

\begin{table*}[h]
\renewcommand{\arraystretch}{1.2}
\begin{tabular}{ll}
\hline
\multicolumn{2}{c}{$m_1(r)=\log^{\beta_1}(1+r)$}\\ \hline
$\beta_1=0$ (Euler)& Global regularity \cite{KRYZ}\\ 
$\beta_1>0$ & Finite-time blowup\\ \hline
\end{tabular}
\quad
\begin{tabular}{ll}
\hline
\multicolumn{2}{c}{$m_2(r) = \log^{\beta_2}\log(e+r)$}\\ \hline
$\beta_2\leq 1$ (Loglog-Euler) & Global regularity \cite{Elg14}\\ 
$\beta_2>1$ & Finite-time blowup \\ \hline
\end{tabular}    
\end{table*}

It is worth noting that the well-posedness result in \cite{Elg14} for the Loglog-Euler equation was established in the whole space
$\RR^2$, and it is reasonable to conjecture that a similar result holds in the . 
We leave the validation of this conjecture for future work.
\end{remark}

The second regime of consideration is for $\alpha$-SQG-like equations, where we assume the multiplier $m(r)$ behaves similar as 
$r^\alpha$ near infinity, satisfying:

\begin{enumerate}[label=$(\mathbf{H}2b)$,ref=$\mathbf{H}2b$]
\item\label{H2b} there exists a constant $\alpha \in (0,1/3)$ so that
\begin{align}\label{cond:m2}
  \lim_{r \rightarrow +\infty} \frac{r\, m'(r)}{m(r)} = \alpha, \quad
  \lim_{r \rightarrow +\infty} \frac{(1-\alpha)m'(r) + r\, m''(r)}{m'(r)} = 0,
\end{align}
\begin{align}\label{cond:m2a}
  \lim_{r\rightarrow +\infty} \frac{(2-\alpha) r\, m''(r) + r^2 m'''(r)}{ m'(r)}
  = \lim_{r\rightarrow +\infty} \frac{(3-\alpha) r^2 m'''(r) + r^3 m^{(4)}(r) }{m'(r)} = 0.
\end{align}
\end{enumerate}
Note that \eqref{H2a} corresponds to a reduced case of \eqref{H2b} with $\alpha=0$.
The assumption $\alpha<1/3$ is required to ensure local well-posedness (see \cite{GanP21}).

Clearly, the $\alpha$-SQG equation with $m(r)=r^\alpha$ satisfies \eqref{H2b}. Other examples include
\begin{align*}
  m_4(r) = r^\alpha \log^\beta(C_{\alpha,\beta} + r), \quad\textrm{with}\;\;
  \begin{cases}
  C_{\alpha,\beta}\geq 1, \;\; & \textrm{for}\;\beta\geq 0, \\
  C_{\alpha,\beta}\;\textrm{large enough}, \;\; & \textrm{for}\; \beta<0,
  \end{cases}
\end{align*}
and
\begin{align*}
  m_5(r) = \frac{r^2}{r^2 + \varepsilon_1^2} (r^2 + \varepsilon_2^2)^{\frac{\alpha}{2}},\quad \textrm{with}\;\; 
 \varepsilon_1,\varepsilon_2\geq 0.
\end{align*}

Now we state our second result.
\begin{theorem}\label{thm:blow-upb}
  Consider the inviscid generalized SQG equation \eqref{eq:geSQG}-\eqref{eq:bdr} in  $\mathbf{D} = \mathbb{R}^2_+$.
  Assume that $m(\xi) = m(|\xi|)$ is a radial function satisfying \eqref{H1}-\eqref{H2b}.
  
  Then there exist non self-intersecting initial patch data $\theta_0$ given by \eqref{eq:patch-data} such that the half-plane inviscid generalized SQG equation \eqref{eq:geSQG}-\eqref{eq:bdr} generates a unique local-in-time $H^2$ patch solution $\theta$ that develops a singularity in finite time.
\end{theorem}

\begin{remark}
The local well-posedness and finite-time singularity formation for the $\alpha$-SQG equation with $\alpha \in (0, 1/3)$ were studied in \cite{KRYZ,KYZ17,GanP21}. Their analysis relied heavily on the explicit form of the kernel  $G$  in \eqref{eq:alphaSQG}. In Theorem \ref{thm:blow-upb}, we make non-trivial extensions of these results to the generalized SQG equation, where the kernel  $G$  may not have an explicit expression. This requires a careful study of the behavior of  $G$  near the origin and adapting it to the existing theories.
\end{remark}

\subsection{Organization of the paper}
Now, we outline our approach to proving Theorems \ref{thm:blow-upa} and \ref{thm:blow-upb}.

In Section \ref{sec:multiplier}, we derive the integral expression \eqref{eq:u-exp} for the velocity field $u$ from the relation 
\eqref{eq:geSQG-u}. This allows us to express the generalized SQG equation \eqref{eq:geSQG}-\eqref{eq:bdr} 
equivalently as equation \eqref{eq:geSQG}$\&$\eqref{eq:u-exp}, which will be the form used in our analysis.

For general choice of the multiplier $m$, the corresponding kernel function $G$ does not have an explicit expression. In Lemma \ref{lem:mD-cond}, we derive several key properties of $G(\cdot)$ near the origin, under the hypotheses either \eqref{H1}-\eqref{H2a}-\eqref{cond:m1c} or \eqref{H1}-\eqref{H2b} on $m$. Roughly speaking, we obtain
\[G(\rho)\sim m(\rho^{-1})\]
when $\rho$ is close to zero. These properties are essential in establishing both local well-posedness and finite-time singularity formation. Notably, the Osgood condition \eqref{cond:m1c} translates into the following condition on $G$:
\begin{equation}\label{eq:OsgoodG}
    \int_0^{1/2} \frac{1}{\rho(\log \rho^{-1})G(\rho)}\dd r<+\infty.
\end{equation}
The properties of $G$ are summarized in \eqref{A1}-\eqref{A2a}, and \eqref{A1}-\eqref{A2b} for the two regimes, respectively.

Section \ref{sec:loc-reg} is devoted to proving the local regularity of $H^2$ patch solutions for the equation 
\eqref{eq:geSQG}$\&$\eqref{eq:u-exp}. Our strategy builds on the work in \cite{KYZ17,GanP21} for the $\alpha$-SQG equation. 
We first derive the contour dynamics equation (see \eqref{eq:contour} or \eqref{eq:main-eq-GP}-\eqref{eq:lambda-def}) 
for the patch solution (see Definition \ref{def:patch-sol}) of the form
\begin{align*}
  \theta(x,t) = \sum\limits_{j=1}^N a_j \mathbf{1}_{D_j(t)}(x).
\end{align*}
A major challenge, compared to the explicit form in the 2D Euler or $\alpha$-SQG patches, is effectively handling the implicit terms, such as $R(|z_k(\zeta,t)-\overline{z}_j(\zeta-\eta,t)|)$ in \eqref{eq:contour}, where $R(\rho)$ is the primitive function of $-\frac{G(\rho)}{\rho}$. Under the general assumption \eqref{A1} on $G$, we manage to bound these implicit terms in Lemma \ref{lem:G-lemma}. This then leads to the local well-posedness of $H^2$ patch solutions for the equation \eqref{eq:geSQG}$\&$\eqref{eq:u-exp}, as presented in Theorem \ref{thm:loc-reg}. 

Additionally, we explore the relationship between the $C^{1,\sigma}$ patch solution and the flow map $\Phi_t$, 
which satisfies \eqref{eq:flow_map}, in Proposition \ref{prop:flow-map}. 
This relation plays a pivotal role in proving finite-time singularity. 
For a detailed proof of local regularity result, refer to Section \ref{subsec:loc-reg} and Appendix \ref{sec:appendix}.

In Section \ref{sec:singu}, we construct patch-like initial data in the form of \eqref{eq:patch-data} and show that the patch solution to the equation \eqref{eq:geSQG}$\&$\eqref{eq:u-exp} develops a finite-time singularity. Our construction follows the scheme introduced in \cite{KRYZ}. 
The core idea is to show that a locally well-posed $H^2$ patch solution, consisting of two distinct patches with odd symmetry in $x_1$-variable and whose boundaries touch $\partial \mathbb{R}^2_+$, will develop a finite-time singularity when the two patches touch at the origin. A key element in this approach is constructing a moving trapezoidal region $\mathbb{K}(t)\subset (\mathbb{R}_+)^2\triangleq \mathbb{R}_+\times \mathbb{R}_+$ inside the patch that drives it towards the origin. 
Unlike the explicit kernels \eqref{eq:alphaSQG} of the $\alpha$-SQG equation studied in \cite{KRYZ,GanP21}, the implicit form of $G(\cdot)$ in our considered equation \eqref{eq:geSQG}$\&$\eqref{eq:u-exp} presents a significant challenge.

Since the key properties \eqref{A1}-\eqref{A2} of $G$ are concentrated near the origin, our initial data is carefully designed to support the patch in a small neighborhood around the origin.
We then proceed to estimate the horizontal velocity $u_1(x,t)$ and vertical velocity $u_2(x,t)$ of the “front” in Proposition \ref{prop:es-u}. 
Finally, we argue that the Osgood condition \eqref{eq:OsgoodG} ensures that the two patches will collide at the origin in finite time, concluding with the main singularity result stated in Theorem \ref{thm:blow-up2}.

\section{The kernel function for a class of Fourier multiplier operators}\label{sec:multiplier}

In this section, 
we shall derive the expression formula of vector field
$u = \nabla^\perp(-\Delta)^{-1} m(\Lambda) \theta$ 
in $\mathbf{D} =\RR^2$ or $ \mathbf{D} = \mathbb{R}^2_+$,
which provides a foundation for relating the generalized SQG equation 
\eqref{eq:geSQG}-\eqref{eq:bdr} to equation \eqref{eq:geSQG}$\&$\eqref{eq:u-exp}. 
Then, we shall carefully study some useful properties of $G(\cdot)$ 
given by \eqref{eq:G-exp1} under the hypotheses either \eqref{H1}-\eqref{H2a} 
or \eqref{H1}-\eqref{H2b}, which exactly fulfill the needs in Sections \ref{sec:loc-reg} 
and \ref{sec:singu}.

The first result is about the expression formula of velocity $u$.
\begin{lemma}\label{lem:int-exp}
Let $u(x) = \nabla^{\perp}(-\Delta)^{-1}m(\Lambda) \theta(x)$,
where $\nabla^\perp = (-\partial_{x_2}, \partial_{x_1})$, 
$m(\Lambda)$ is a Fourier multiplier operator with the symbol
$m(\xi)=m(|\xi|)$ a radial function satisfying that
$m(\xi)\in C^2(\RR^2\setminus \{0\})$ and $\lim\limits_{r\to 0^{+}}m(r)$, $\lim\limits_{r\to 0^{+}}rm'(r)$ exist, and also
\begin{align}\label{cond:m-add}
  \lim\limits_{r\to +\infty}r^{-\frac{1}{2}}m(r)=0,\quad \lim\limits_{r\rightarrow +\infty} r^{\frac{1}{2}} m'(r) =0.
\end{align}
Then the following statements hold true.
\begin{enumerate}
\item
For every $x\in \RR^2$ we have
\begin{align}\label{eq:u-exp1}
  u(x) = K\ast \theta(x) = \int_{\RR^2} K(x-y) \theta(y) \dd y,\quad
  K(x)  \triangleq  \frac{x^{\perp}}{|x|^2} G(|x|),
\end{align}
where  $x^\perp = (x_2,-x_1)$,
\begin{equation}\label{eq:G-exp1}
\begin{split}
  G(|x|) & \triangleq  \frac{1}{2\pi}m(0^+) + \frac{1}{2\pi} \int_0^\infty J_0(|x| r)\, m'(r) \dd r  \\
  & =  \frac{1}{2\pi} m(0^+) +  \frac{1}{(2\pi)^2} \int_{\RR^2} e^{i x\cdot \xi} \frac{m'(|\xi|)}{|\xi|} \dd \xi ,
\end{split}
\end{equation}
with $m(0^+)\triangleq \lim\limits_{r\rightarrow 0^+}m(r)$ and $J_0(\cdot)$ the zero-order Bessel function given by \eqref{def:BessJ0}.
\item Let $u=(u_1,u_2)$ and $\theta$ be functions defined on $\mathbb{R}^2_+$
satisfying the rigid boundary condition $u_2|_{\partial \mathbb{R}^2_+} =0$.
For every $x\in \mathbb{R}^2_+$, we have
\begin{align}\label{eq:u-exp2}
  u(x) = \int_{\mathbb{R}^2_+} \big(K(x-y) - K(x -\overline{y}) \big) \theta(y) \dd y,
\end{align}
where $y=(y_1, y_2)$, $\overline{y} = (y_1, -y_2)$, and $K(\cdot)$ is given by \eqref{eq:u-exp1}.
\end{enumerate}
\end{lemma}

\begin{remark}\label{rmk:phi-exp}
  If the condition $\lim\limits_{r\to +\infty}r^{-\frac{1}{2}}m(r)=0$ in \eqref{cond:m-add} does not hold, 
we rewrite the formula  \eqref{eq:phi-exp1} of $\phi(\rho)$ as
\begin{align*}
  \phi(\rho) = \frac{1}{2\pi} \int_0^\rho \big(J_0(r) - 1\big) \frac{m(\frac{r}{\rho})}{r} \dd r 
  + \int_\rho^\infty J_0(r) \frac{m(\frac{r}{\rho})}{r} \dd r,
\end{align*}
which also yields
\begin{align*}
  \phi'(\rho) &= - \frac{1}{2\pi \rho^2} \int_0^\infty J_0(r) m'\big(\tfrac{r}{\rho}\big) \dd r 
  +\frac{1}{2\pi \rho^2}\int_{0}^{\rho}m'(\tfrac{r}{\rho})\dd r-\frac{1}{2\pi}\frac{m(1)}{\rho} \\ 
  & =- \frac{1}{2\pi \rho} \int_0^\infty J_0(\rho r) m'(r) \dd r-\frac{1}{2\pi \rho}m(0^{+}),
\end{align*}
which is exactly the same with \eqref{eq:phi'-exp}. In this way, the formula \eqref{eq:phi'-exp} 
will hold under the conditions \eqref{H1}-\eqref{H2b} with more general scope $\alpha\in (0,1)$ or even $\alpha\in(0,2)$.
\end{remark}

\begin{proof}[Proof of Lemma \ref{lem:int-exp}]
\textbf{(i)} Denote by
\begin{align*}
  f(|\xi|) \triangleq \frac{m(|\xi|)}{|\xi|^2},
\end{align*}
then our aim is to compute $\mathcal{F}^{-1} (f) (x) \triangleq \phi(|x|) = \phi(\rho)$, which also solves that
\begin{align}\label{eq:ellpEq-fund}
  \frac{- \Delta}{m(\Lambda)} \phi(|x|) = \delta_0 \quad \textrm{in}\;\;\mathbb{R}^2,
\end{align}
with $\delta_0$ the Dirac measure on $\mathbb{R}^2$ centered at the point $0$.
Note that due to that $f$ is radial, $\mathcal{F}^{-1}(f)$ will also be radial.
It follows from the definition of inverse Fourier transform that
\begin{align}\label{eq:phi-exp0}
  \phi(\rho) & = \mathcal{F}^{-1}(f)(x)
= \frac{1}{(2\pi)^2} \int_{\RR^2} e^{ix\cdot \xi} \frac{m(|\xi|)}{|\xi|^2} \dd \xi   = \frac{1}{(2\pi)^2} \int_{\RR^2} e^{i \rho  \xi_2} \frac{m(|\xi|)}{|\xi|^2} \dd \xi \nonumber\\
 & = \frac{1}{(2\pi)^2} \int_0^\infty \Big(\int_{-\pi}^\pi e^{i \rho\, r \sin \eta}\dd \eta\Big) \frac{m(r)}{r} \dd r = \frac{1}{2\pi} \int_0^\infty J_0(\rho r) \frac{m(r)}{r} \dd r,
\end{align}
where $J_0$ is the zero-order Bessel function of the first kind defined as follows (e.g. see \cite[Sec. 4.9]{Asmar05}):  	
\begin{align}\label{def:BessJ0}
  J_0(x)  = \frac{1}{2\pi}\int_{-\pi}^{\pi}e^{-ix\sin \eta}\dd\eta.
\end{align}
In the above, we may encounter a problem with the above integral at $r=0$.
To circumvent this problem,
we define the distribution $\frac{1}{|\cdot|^2}$ more properly as:
\begin{align*}
  \Big\langle \frac{m(\cdot)}{|\cdot|^2}, g \Big\rangle \triangleq
  \int_{B_1} \frac{(g(\xi)-g(0))m(|\xi|)}{|\xi|^2}\dd\xi
  +\int_{B_1^c} \frac{g(\xi)m(|\xi|)}{|\xi|^2}\dd\xi.
\end{align*}
From this we get 
\begin{align}\label{eq:phi-exp1}
  \phi(\rho)
  & = \frac{1}{(2\pi)^2} \int_{B_1} \frac{(e^{ix\cdot \xi } -1) m(|\xi|)}{|\xi|^2} \dd \xi
  + \frac{1}{(2\pi)^2} \int_{B_1^c} \frac{e^{ix\cdot \xi} m(|\xi|)}{|\xi|^2} \dd \xi \nonumber \\
  & = \frac{1}{2\pi} \bigg( \int_0^1 \big(J_0(\rho r)-1\big) \frac{m(r)}{r} \dd r
  + \int_1^\infty J_0(\rho r) \frac{m(r)}{r} \dd r \bigg).
\end{align}
Noting that 
\begin{align*}
  \frac{\phi(\rho_1)-\phi(\rho_2)}{\rho_1-\rho_2} 
  = \frac{1}{2\pi}\int_0^\infty \int_0^1 J'_0\big((s\rho_1+(1-s)\rho_2) r\big) m(r)\dd s \dd r,
\end{align*}
we now differentiate $\phi(\rho)$ in $\rho$:
\begin{align*}
  \phi'(\rho)= \frac{1}{2\pi} \int_0^\infty J_0'(\rho r) m(r) \dd r.
\end{align*}
Using the asymptotics of Bessel function 
$J_0(x) \sim \sqrt{\frac{2}{\pi x}} \cos(x -\frac{\pi}{4})$
and under the assumption that $\lim\limits_{r\rightarrow\infty} r^{-\frac{1}{2}} m(r) =0$, we integrate by parts to deduce that
\begin{equation}\label{eq:phi'-exp}
\begin{split}
  \phi'(\rho)
  & = - \frac{1}{2\pi \rho} \bigg( \lim_{r\rightarrow 0^+} m(r) + \int_0^\infty J_0(\rho r) m'(r) \dd r \bigg) \\
  & = - \frac{1}{2\pi \rho} \bigg( m(0^+)
  + \frac{1}{2\pi} \int_{\RR^2} e^{i x\cdot \xi} \frac{m'(|\xi|)}{|\xi|} \dd \xi \bigg) .
\end{split}
\end{equation}
Notice that if $m \equiv 1$, the above expression formulas \eqref{eq:phi-exp0}, \eqref{eq:phi-exp1}-\eqref{eq:phi'-exp}
imply that $\phi(\rho) = -\frac{1}{2\pi} \textrm{log} \,\rho$,
and $\phi'(\rho) = - \frac{1}{2\pi \rho}$,
which coincide with the 2D Euler case; while if $m(r) = r^\alpha$, $\alpha\in (0,\frac{1}{2})$,
we have $\phi(\rho) = \mathrm{c}_\alpha \rho^{-\alpha}$ and $\phi'(\rho) = -\alpha \mathrm{c}_\alpha \rho^{-\alpha -1}$
with $\mathrm{c}_\alpha =\frac{1}{2\pi} \int_0^\infty J_0(r) r^{\alpha-1} \dd r
= \frac{\Gamma(\frac{\alpha}{2})}{\pi 2^{2-\alpha}\Gamma(1-\frac{\alpha}{2})}$ (see \cite[sec. 6.561]{GR15}),
which coincide with the $\alpha$-SQG case.
	
Moreover, we have
\begin{align}\label{eq:u-exp1.2}
  u(x) & = \nabla^\perp  \big(\phi(|\cdot|)\ast \theta\big) (x)
   = \Big( - \phi'(|x|)\frac{x^{\perp}}{|x|} \Big)\ast \theta(x) \nonumber\\
  & = \int_{\RR^2} \frac{(x-y)^\perp}{|x-y|^2} G(|x-y|) \theta(y) \dd y  = K \ast \theta(x),
\end{align}
where
\begin{align}\label{def:G}
  G(\rho)\triangleq  - \rho\, \phi'(\rho).
\end{align}
This gives the formula \eqref{eq:u-exp1}-\eqref{eq:G-exp1}.
By using \eqref{eq:G-exp1} and the integration by parts,
we infer that for every $\rho>0$,
\begin{equation}\label{eq:G'-exp}
\begin{split}
  G'(\rho) = G'(|x|) & = \frac{1}{2\pi} \int_0^\infty J_0'(\rho r) r m'(r)\dd r  \\
  & = - \frac{1}{2\pi\rho}\lim_{r\rightarrow 0^+} r m'(r) - \frac{1}{2\pi\rho} \int_0^\infty J_0(\rho r)
  \big(m'(r) + r m''(r)\big) \dd r \\
  & = - \frac{1}{2\pi\rho}\lim_{r\rightarrow 0^+} r m'(r) - \frac{1}{(2\pi)^2\rho} \int_{\RR^2}  e^{ix\cdot\xi}
  \frac{m'(|\xi|) + |\xi| m''(|\xi|)}{|\xi|} \dd \xi ,
\end{split}
\end{equation}
where we have used the fact that
$\lim\limits_{r\rightarrow +\infty} r^{\frac{1}{2}} m'(r) =0$ and the decaying property of $J_0(\cdot)$.

\textbf{(ii)} Under the rigid boundary condition $u_2|_{\partial \mathbb{R}^2_+} =0$, we see that $u(x)= \nabla^\perp \psi(x)$
with $\psi$ the stream function solving the following equation
\begin{align*}
  \frac{-\Delta}{m(\Lambda)} \psi = \theta,\quad
  \textrm{in}\,\, \mathbb{R}^2_+,\quad\quad \psi|_{\partial \mathbb{R}^2_+} = 0.
\end{align*}
For a function $g$ defined on $\mathbb{R}^2_+$, denote by $e_o[g]$ the following extension operator
\begin{align}\label{def:ext-op}
  e_o[g](x) \triangleq
  \begin{cases}
    g(x),\quad & \textrm{for}\;\, x_2\geq 0, \\
    - g(x_1,-x_2),\quad &\textrm{for}\;\, x_2<0.
  \end{cases}
\end{align}
Note that
\begin{align*}
  \frac{-\Delta}{m(\Lambda)} e_o[\psi] = e_o[\theta],\quad \textrm{in}\;\, \mathbb{R}^2.
\end{align*}
In view of \eqref{eq:ellpEq-fund}, we find that for every $x\in \mathbb{R}^2_+$,
\begin{align*}
  \psi(x) = e_o[\psi](x)
  & = \int_{\mathbb{R}^2} \phi(|x-y|) e_o[\theta](y) \dd y \\
  & = \int_{\mathbb{R}^2_+} \phi(|x-y|) \theta(y) \dd y - \int_{\mathbb{R}\times (-\infty,0)} \phi(|x-y|)
  \theta(y_1,-y_2) \dd y \\
  & = \int_{\mathbb{R}^2_+} \Big(\phi(|x-y|) - \phi(|x-\overline{y}|)\Big) \theta(y) \dd y.
\end{align*}
Hence, applying the differential operator $\nabla^\perp$ to the above formula and estimating as \eqref{eq:u-exp1.2}
lead to \eqref{eq:u-exp2}, as desired.
\end{proof}


Next, we show some crucial properties of $G(\rho)$ given by \eqref{eq:G-exp1} under suitable assumptions on $m$.
\begin{lemma}\label{lem:mD-cond}
Assume that $m(\xi)= m(|\xi|)$ is a radial function of $\RR^2$ satisfying either
\eqref{H1}-\eqref{H2a} or \eqref{H1}-\eqref{H2b}
(see \eqref{cond:m0}-\eqref{eq:MH} and \eqref{cond:m1}-\eqref{cond:m2a} in introduction).
Then there exist constants $\bar{c}>0$,  $\bar{c}_0 >0$, and $C>0$ such that $G(\rho)=G(|x|)$ defined by \eqref{eq:G-exp1} 
verifies the following statements:
\begin{align}\label{eq:G-prop1}
  \bar{c}\, m(\rho^{-1}) \leq G(\rho) \leq C m(\rho^{-1}), \quad \forall \rho \in(0, \bar{c}_0],
\end{align}
and
\begin{align}\label{eq:G-prop1b}
  \textrm{on}\;(0,\bar{c}_0),\quad \quad |G'(\rho)| \leq \frac{C m(\rho^{-1})}{\rho},
  \quad |G''(\rho)| \leq \frac{C m(\rho^{-1})}{\rho^2},
\end{align}
and
\begin{equation}\label{eq:G-prop1c}
\begin{cases}
  \textrm{the function $\tfrac{G(\rho)}{\rho}$ is decreasing on $(0,\bar{c}_0]$},\quad &\textrm{if\; \eqref{H2a} is assumed}, \\
  \textrm{the function $G(\rho)$ is decreasing on $(0,\bar{c}_0]$},\quad &\textrm{if\; \eqref{H2b} is assumed}.
\end{cases}
\end{equation}	
and
\begin{align}\label{eq:G-prop2}
  |G(\rho)| + |G'(\rho)| +  |G''(\rho)| \leq C,\quad \forall \rho\geq \bar{c}_0.
\end{align}
Besides, we also have
\begin{align}\label{eq:G-prop3}
  \lim_{\rho\rightarrow 0^+} \frac{\rho\,G'(\rho) + \alpha\, G(\rho)}{G(\rho)} = 0,\quad \textrm{if \eqref{H2b} is assumed},
\end{align}
and for every $l >0$,
\begin{equation}\label{eq:G-prop4}
\begin{cases}
  \lim\limits_{\rho\rightarrow 0^+} \frac{m(l \rho^{-1})}{m(\rho^{-1})} = 1, \quad & \textrm{if\; \eqref{H2a} is assumed}, \\
  \lim\limits_{\rho \rightarrow 0^+} \frac{l^\alpha G(l \rho)}{ G(\rho)} = 1, \quad & \textrm{if\; \eqref{H2b} is assumed}.
\end{cases}
\end{equation}
\end{lemma}

\begin{remark}\label{rmk:m-grow}
\begin{enumerate}
\item Note that the conditions \eqref{cond:m1} and \eqref{cond:m2} respectively imply that for $r>0$ large enough,
\begin{align}\label{cond:m2-deduc}
  0<\frac{r\,m'(r)}{m(r)}\leq \epsilon,\;\;\forall \epsilon>0,\quad \textrm{and}\quad
  \alpha -\epsilon \leq \frac{r\, m'(r)}{m(r)} \leq \alpha+\epsilon,\;\;\forall \epsilon\in(0,\alpha).
\end{align}
By some direct calculation (e.g. see \cite[Lemma 2.2]{MX19}), the following holds that under the condition \eqref{cond:m1},
\begin{align}\label{eq:m-grow0}
  C^{-1} \leq m(r) \leq C r^\epsilon,\;\; \forall \epsilon\in(0,\tfrac{1}{2}),\qquad
  \textrm{for \;$r>0$ large enough};
\end{align}
and under the condition \eqref{cond:m2},
\begin{align}\label{eq:m-grow-alp}
  C^{-1} r^{\alpha -\epsilon} \leq m(r) \leq C r^{\alpha+\epsilon},\;\;
  \forall \epsilon \in (0,\alpha), \qquad \textrm{for \;$r > 0$ large enough}.
\end{align}
As a consequence of the above inequalities, the condition \eqref{cond:m-add} in both cases and \eqref{cond:m1c}
in the case of \eqref{H1}-\eqref{H2b} will naturally hold.  
\item 
From \eqref{eq:MH}, we see that 
$- C \frac{m'(r)}{r} \leq\frac{\dd m'(r)}{\dd r} \leq C\frac{m'(r)}{r}$,
which implies that $r^{C} m'(r)$ is non-decreasing for every $r>0$ and $r^{-C}m'(r)$ is non-increasing for every $r>0$, 
thus consequently,  
\begin{align}\label{ineq:bound-1-MH}
  l^{-C} m'(r)  \leq m'(lr) \leq l^C m'(r) ,\qquad \forall r>0,\,l\geq 1.
\end{align}
Moreover, the non-decreasing property of $m$ (by \eqref{cond:m0}) and \eqref{ineq:bound-1-MH} 
guarantee that for every $C_0 \geq 1$ and $r>0$,
\begin{equation}\label{ineq:m-bound}
\begin{split}
  m(r) \leq m(C_0r) & \leq  m(r) + C_0 \int_{\frac{r}{C_0}}^r m'(C_0 \tau) \dd \tau \\
  & \leq m(r) + C_0^{C+1} \int_{\frac{r}{C_0}}^r m'(\tau) \dd \tau \leq (C_0^{C+1} +1) m(r).
\end{split}
\end{equation}
\end{enumerate}
\end{remark}

\begin{proof}[Proof of Lemma \ref{lem:mD-cond}]
We first consider the case under the conditions \eqref{H1}-\eqref{H2a}.
Let $\chi(\xi)= \chi(|\xi|)\in C^\infty_c(\RR^2)$ be a smooth radial function such that
\begin{align}\label{eq:chi-prop}
  \chi\equiv 1,\;\; \textrm{on}\;\{|\xi|\leq 1\},\qquad \chi\equiv 0,\;\; \textrm{on}\;\{|\xi|\geq 2\},\qquad
  0\leq \chi\leq 1.
\end{align}
For $G(\rho)= G(|x|)$ given by \eqref{eq:G-exp1}, we have
\begin{align}\label{eq:G-decom}
  G(\rho) & = \frac{1}{2\pi}m(0^+) + \frac{1}{(2\pi)^2} \int_{\RR^2} e^{ix\cdot\xi} \chi(\rho |\xi|) \frac{m'(|\xi|)}{|\xi|} \dd \xi
  + \frac{1}{(2\pi)^2} \int_{\RR^2} e^{ix\cdot\xi} \big(1-\chi(\rho |\xi|)\big)
  \frac{m'(|\xi|)}{|\xi|}  \dd \xi \nonumber \\
  & = \frac{1}{2\pi}m(0^+) + \frac{1}{2\pi} 
  \int_0^\infty J_0(\rho r) \chi(\rho r) m'(r) \dd r
  + \frac{1}{(2\pi)^2} \int_{\RR^2} e^{ix\cdot\xi} \big(1-\chi(\rho |\xi|)\big)
  \frac{m'(|\xi|)}{|\xi|}  \dd \xi \nonumber \\
  & \triangleq \frac{1}{2\pi}m(0^+) + I_1  + I_2.
\end{align}
Noting that $\frac{1}{2} \leq J_0(r) \leq 1$ for $0\leq r\leq 1$ (e.g. see \cite[Sec. 4.7]{Asmar05}), we have
\begin{align*}
  I_1 \geq \frac{1}{2\pi} \int_0^{\rho^{-1}} J_0(\rho r) m'(r) \dd r 
  \geq \frac{1}{4\pi} \,\int_0^{\rho^{-1}} m'(r) \dd r
  = \frac{1}{4\pi}\, \big(m(\rho^{-1}) - m(0^+) \big),
\end{align*}
and
\begin{align}\label{eq:I1-es}
  I_1  \leq \frac{1}{2\pi} \int_0^{2\rho^{-1}} J_0(\rho r) m'(r) \dd r 
  \leq \frac{1}{2\pi} \int_0^{2\rho^{-1}} m'(r) \dd r
  = \frac{1}{2\pi} \big( m(2\rho^{-1}) - m(0^+) \big).
\end{align}
For $I_2(\rho)$, taking advantage of the integration by parts, \eqref{eq:MH}, \eqref{ineq:bound-1-MH} 
and the fact that (owing to \eqref{cond:m1})
\begin{align}\label{eq:fact-m}
  \textrm{$r\mapsto m'(r)$ is non-increasing for $r>0$ sufficiently large}, 
\end{align}
we find that for $\rho>0$ small enough (i.e. $0<\rho\leq \bar{c}_0$),
\begin{align}\label{eq:I2-es}
  |I_2 | & = \frac{1}{(2\pi)^2} \frac{1}{\rho^2}\Big| \int_{\RR^2} e^{ix\cdot\xi} \Delta_\xi
  \Big( \big(1-\chi(\rho|\xi|)\big)
  \frac{m'(|\xi|)}{|\xi|}\Big) \dd \xi \Big| \nonumber  \\
  & \leq \frac{C}{\rho^2} \bigg(\int_{\rho^{-1}\leq |\xi|\leq 2\rho^{-1}}
  \big(\rho^2 + \rho |\xi|^{-1}\big) \frac{m'(|\xi|)}{|\xi|}  \dd \xi
  + \int_{|\xi|\geq \rho^{-1}} \frac{m'(|\xi|)}{|\xi|^3}\dd \xi \bigg) \nonumber \\
  & \leq \frac{C}{\rho^2} \Big( \rho\, m'(\rho^{-1}) + m'(\rho^{-1}) \int_{\rho^{-1}}^\infty \frac{1}{r^2} \dd r\Big) 
  \leq C \rho^{-1} m'(\rho^{-1}).
\end{align}
Hence we get that for every $\rho>0$ small enough,
\begin{align}\label{eq:G-bdd}
  \frac{1}{4\pi}m(0^+) + \frac{1}{4\pi}\, m(\rho^{-1}) - C \rho^{-1} m'(\rho^{-1})
  \leq G(\rho)\leq \frac{1}{2\pi} m(2\rho^{-1}) + C \rho^{-1} m'(\rho^{-1}) .
\end{align}
Thus under the conditions in \eqref{cond:m1}, and noticing \eqref{ineq:m-bound},
we infer that $G(\rho)\approx m(\rho^{-1})$	for $\rho>0$ sufficiently small, that is, \eqref{eq:G-prop1} holds.
The bound $|G(\rho)|\leq C$ ($\forall\rho\in [\bar{c}_0,\infty)$) can be easily deduced from
\eqref{eq:G-decom}-\eqref{eq:I2-es} and the following estimate that
\begin{equation}\label{eq:fact-m2}
\begin{split}
  \frac{1}{\rho^2}\int_{|\xi|\geq \rho^{-1}} \frac{m'(|\xi|)}{|\xi|^3} \dd \xi 
  & \leq \frac{C}{\rho^2} \int_{\rho^{-1}}^{\bar{c}_0^{\,-1}} 
  \frac{m'(r)}{r^2} \dd r + \frac{C}{\rho^2} m'(\bar{c}_0^{\,-1}) \int_{\bar{c}_0^{\,-1}}^\infty \frac{1}{r^2} \dd r \\
  & \leq \frac{C}{\rho^2} \Big(\bar{c}_0^2 m(\bar{c}_0^{\,-1}) 
  + 2 \int_{\rho^{-1}}^{\bar{c}_0^{\,-1}} \frac{m(r)}{r^3} \dd r + \bar{c}_0 m'(\bar{c}_0^{\,-1})\Big) 
  \leq C m(\bar{c}_0^{\,-1}).
\end{split}
\end{equation}
	
Now we consider the properties of $G'(\rho)$ and $G''(\rho)$ as in \eqref{eq:G-prop1b} and \eqref{eq:G-prop2}.
We start from the expression formula \eqref{eq:G'-exp} and its derivative
\begin{align*}
  & G''(\rho) = -\frac{1}{\rho}G'(\rho)-\frac{1}{2\pi\rho}\int_0^\infty J'_0(\rho r)\big(r\,m'(r)+r^2\,m''(r)\big)\dd r \nonumber \\
  & = -\frac{G'(\rho) }{\rho}+  \frac{1}{2\pi\rho^2}\lim_{r\to  0^+} 
  \big(r \,m'(r) + r^2\, m''(r)\big)
  +\frac{1}{2\pi\rho^2} \int_0^{\infty} J_0(\rho r) 
  \Big(m'(r)+3r\,m''(r)+r^2\,m'''(r)\Big) \dd r  ,
\end{align*}
where in the last line we have used the integration by parts and the fact that 
$\lim\limits_{r\rightarrow \infty} \big(r^{\frac{1}{2}} m'(r) + r^{\frac{3}{2}} |m''(r)|\big) =0$.
Similarly as \eqref{eq:G-decom}, we have
\begin{align}\label{eq:G'-decom}
  G'(\rho) & = - \frac{1}{2\pi\rho}\lim_{r\rightarrow 0^+} r m'(r) 
  -\frac{1}{2\pi\rho} \int_0^\infty J_0(\rho r) \chi(\rho r) \big(m'(r) + r m''(r) \big) \dd r \nonumber \\
  &\quad - \frac{1}{(2\pi)^2\rho} \int_{\RR^2}  e^{ix\cdot\xi} \big(1- \chi(\rho|\xi|) \big)
  \frac{m'(|\xi|) + |\xi| m''(|\xi|)}{|\xi|} \dd \xi \nonumber \\
  & \triangleq - \frac{1}{2\pi \rho} \Big(\lim_{r\rightarrow 0^+} r m'(r)\Big) + \widetilde{I}_1
  + \widetilde{I}_2,
\end{align}
and
\begin{align*}
  G''(\rho) = & -\frac{1}{\rho}G'(\rho) + \frac{1}{2\pi\rho^2}\lim_{r\to  0^+}\big(r \,m'(r)+r^2m''(r)\big) \\
  & + \frac{1}{2\pi\rho^2} \int_0^{\infty}J_0(\rho r)\chi(\rho r) \big(m'(r)+3rm''(r)+r^2m'''(r)\big) \dd r\\
  &+\frac{1}{(2\pi)^2 \rho^2}\int_{\RR^2}e^{ix\cdot \xi}(1-\chi(\rho |\xi|))
  \frac{m'(|\xi|)+3|\xi|m''(|\xi|)+|\xi|^2 m'''(|\xi|)}{|\xi|}\dd \xi \\
  \triangleq & -\frac{1}{\rho}G'(\rho)
  + \frac{1}{2\pi\rho^2} \lim_{r\to  0^+}\big(m'(r)r+r^2m''(r)\big)
  +\mathbf{I}_1+\mathbf{I}_2.
\end{align*}
By using \eqref{eq:MH}, we obtain
\begin{align*}
  |\widetilde{I}_1| & \leq \frac{1}{2\pi \rho} \int_0^\infty J_0(\rho r) \chi(\rho r)
  \big(m'(r) + r |m''(r)| \big) \dd r \\
  & \leq \frac{C}{\rho} \int_0^{2\rho^{-1}} J_0(\rho r) m'(r) \dd r
  \leq \frac{C }{\rho} \Big(m(2\rho^{-1})  - m(0^+) \Big),
\end{align*}
and similarly,
\begin{align*}
  |\mathbf{I}_1|\leq \frac{C}{\rho^2} \big(m(2\rho^{-1})-m(0^+)\big).
\end{align*}
For the terms $\widetilde{I}_2$ and $\mathbf{I}_2$,
arguing as getting \eqref{eq:I2-es} we find that for $\rho>0$ small enough,
\begin{align}\label{eq:tild-I2-es}
  |\widetilde{I}_2| = \frac{1}{(2\pi)^2} \frac{1}{\rho^3}\Big| \int_{\RR^2} e^{ix\cdot\xi} \Delta_\xi
  \Big( \big(1-\chi(\rho|\xi|)\big)
  \frac{m'(|\xi|) + |\xi| m''(|\xi|)}{|\xi|}\Big) \dd \xi \Big| \leq C \frac{m'(\rho^{-1})}{\rho^2},
\end{align}
and
\begin{align*}
  |\mathbf{I}_2|\leq \frac{1}{(2\pi)^2 \rho^4}\Big|\int_{\RR^2}e^{ix\cdot \xi}\Delta_{\xi}
  \Big((1-\chi(\rho |\xi|))\frac{m'(|\xi|)+3|\xi|m''(|\xi|+|\xi|^2m'''(|\xi|))}{|\xi|}\Big)\dd \xi \Big|
  \leq C\frac{m'(\rho^{-1})}{\rho^3}.
\end{align*}
Collecting the above estimates leads to that for every $\rho >0$ small enough,
\begin{align}\label{eq:G'-bdd}
  |G'(\rho)| \leq \frac{1}{2\pi\rho} \Big(\lim\limits_{r\rightarrow 0^+} r m'(r) \Big)
  + \frac{C}{\rho}\big(m(2\rho^{-1}) - m(0^+)\big) + C \frac{ m'(\rho^{-1})}{\rho^2},
\end{align}
and
\begin{align}\label{eq:G''-extra}
  |G''(\rho)|\leq \frac{1}{2\pi\rho^2} \Big( 2\lim\limits_{r\rightarrow 0^+} r m'(r) + \lim_{r\to  0^+} r^2|m''(r)|\Big)
  +\frac{C}{\rho^2} \big(m(2\rho^{-1})-m(0^+)\big) + C \frac{m'(\rho^{-1})}{\rho^3}.
\end{align}
Hence, arguing as the above treating of $G(\rho)$, 
the corresponding properties of $|G'(\rho)|$ and  $|G''(\rho)|$ in \eqref{eq:G-prop1b} and \eqref{eq:G-prop2} 
can be easily deduced.
	
Next, as a consequence of the condition $\lim\limits_{r\rightarrow +\infty} \frac{r\,m'(r)}{m(r)} =0$ (which can be deduced from \eqref{cond:m1}), we can show that $\lim\limits_{\rho\rightarrow 0^+} \frac{m(l \rho^{-1})}{m(\rho^{-1})} = 1$ for every $l>0$.
Indeed, without loss of generality we assume $l> 1$, then we see that $m(l\rho^{-1}) \geq m(\rho^{-1})$ and 
(using \eqref{ineq:bound-1-MH})
\begin{align*}
  m(l \rho^{-1}) - m(\rho^{-1}) = \int_{\rho^{-1}}^{l \rho^{-1}} m'(r) \dd r \leq  l^C(l-1) \rho^{-1} m'(\rho^{-1});
\end{align*}
which immediately leads to
\begin{align*}
  0 \leq \lim_{\rho\rightarrow 0^+} \Big(\frac{m(l \rho^{-1})}{m(\rho^{-1})} - 1\Big)
  \leq l^C (l-1) \lim_{\rho\rightarrow 0^+} \frac{\rho^{-1} m'(\rho^{-1})}{m(\rho^{-1})} =0,
\end{align*}
as desired.

It remains to show the decreasing property of $\rho\mapsto\frac{G(\rho)}{\rho}$ for $\rho>0$ small enough.
To this end, it suffices to prove that $G'(\rho) < \frac{1}{\rho}G(\rho)$ for $\rho>0$ small enough.
We also use the same splitting of $G'(\rho)$ as in \eqref{eq:G'-decom}.
From the limit $\lim\limits_{r \rightarrow +\infty} \frac{m'(r) + r\, m''(r)}{m'(r)} = 0$ (see \eqref{cond:m1}), 
for every $\varepsilon>0$, there exists a constant $R_1=R_1(\varepsilon)>0$ such that
\begin{align*}
  |m'(r) + r m''(r)| \leq \varepsilon\, m'(r),\quad \forall r\geq R_1;
\end{align*}
then by using the property of Bessel function $J_0$ and \eqref{ineq:m-bound},
we find that for every $0<\rho \leq (2R_1)^{-1}$ (small enough if necessary),
\begin{align*}
  |\widetilde{I}_1| & \leq \frac{1}{2\pi\rho} \int_0^\infty J_0(\rho r) \chi(\rho r) \big| m'(r) + r m''(r)\big| \dd r \\
  & \leq \frac{1}{2\pi\rho} \int_0^{R_1} J_0(\rho r) \big(m'(r) + r |m''(r)|\big) \dd r
  + \frac{\varepsilon}{2\pi\rho} \int_{R_1}^{2 \rho^{-1}} J_0(\rho r) m'(r) \dd r \\
  & \leq \frac{C}{\rho} \big(m(R_1) - m(0^+)\big) + \frac{\varepsilon}{\rho} m(2\rho^{-1}) \\
  & \leq \frac{C}{\rho} m(R_1) + \frac{\varepsilon}{\rho} 2^{C+2} m(\rho^{-1}).
\end{align*}
In view of \eqref{eq:tild-I2-es} and \eqref{cond:m1}, we obtain that for $\varepsilon>0$,
there exists a constant $R_2= R_2(\varepsilon)>0$ so that $C r m'(r) \leq \varepsilon m(r) $ for every $r\geq R_2$,
and also for every $0<\rho\leq R_2^{-1}$,
\begin{align*}
  |\widetilde{I}_2| \leq C \rho^{-2} m'(\rho^{-1}) \leq \frac{\varepsilon}{\rho} m(\rho^{-1}).
\end{align*}
Gathering the above estimates yields that for every $0<\rho \leq \min\{(2R_1)^{-1}, R_2^{-1}\}$,
\begin{align*}
  |G'(\rho)|\leq \frac{1}{\rho} \Big(\lim_{r\rightarrow 0^+} r m'(r) + C m(R_1) \Big) 
  + \frac{\varepsilon}{\rho} (2^{C+2} + 1) m(\rho^{-1}).
\end{align*}
Since $\lim\limits_{r\rightarrow +\infty} m(r)=+\infty$, there exists a constant $R_3=R_3(R_1,\varepsilon)>0$ such that
$\lim\limits_{r\rightarrow 0^+} r m'(r) + C m(R_1) \leq \varepsilon m(r)$ for every $r\geq R_3$, so that for every
$0<\rho\leq \rho_0\triangleq \min\{(2R_1)^{-1},R_2^{-1},R_3^{-1}\} $,
\begin{align*}
  |G'(\rho)| \leq \frac{ \varepsilon}{\rho} 2^{C+3} m(\rho^{-1}).
\end{align*}
Recalling that $G(\rho)\approx m(\rho^{-1})$ for every $\rho >0$ small enough, we can choose $\varepsilon>0$
to be a fixed small constant so that the desired result $|G'(\rho)| < \frac{1}{\rho} G(\rho)$ holds for every $0<\rho\leq \rho_0$
($\rho_0$ is now fixed). Hence we verify the statement \eqref{eq:G-prop1c}
and complete the proof under the conditions \eqref{H1}-\eqref{H2a}.
\vskip2mm

Now we turn to the proof of \eqref{eq:G-prop1}-\eqref{eq:G-prop4} 
under hypotheses \eqref{H1}-\eqref{H2b}.
The upper bounds of $G(\rho)$, $|G'(\rho)|$ and $|G''(\rho)|$ in \eqref{eq:G-prop1}-\eqref{eq:G-prop2} can be easily deduced: indeed,
in the same way as above we obtain the upper bound of \eqref{eq:G-bdd}, \eqref{eq:G'-bdd} and \eqref{eq:G''-extra}  
for every $\rho >0$ small enough (the fact \eqref{eq:fact-m} can be ensured by \eqref{cond:m2}), then the desired result follows by combining these inequalities with \eqref{cond:m0}, \eqref{cond:m2} 
and the facts \eqref{ineq:m-bound}, \eqref{eq:fact-m2}.
	
The lower bound of $G(\rho)$ in the considered case is more delicate.
We borrow some idea from \cite[Lemma 5.2]{DKSV14}.
First, we claim that there exists a constant $c>0$ depending only on $\alpha$ such that
\begin{align}\label{cond:m2b}
  \lim_{|\xi|\rightarrow +\infty} \Big(\frac{m'(|\xi|)}{|\xi|^3}\Big)^{-1}\,
  \Delta_\xi \Big( \frac{m'(|\xi|)}{|\xi|}\Big) \geq c,
\end{align}
\begin{align}\label{cond:m2c}
  \lim_{|\xi|\rightarrow +\infty} \Big(\frac{m'(|\xi|)}{|\xi|^3}\Big)^{-1}\,
  \Delta_\xi \Big(\frac{m'(|\xi|) + |\xi| m''(|\xi|)}{|\xi|} \Big) \geq c.
\end{align}
Indeed, due to that $\Delta_\xi \big(g(|\xi|)\big) = g''(|\xi|) + \frac{1}{|\xi|} g'(|\xi|)$,
direct computation implies that the estimates \eqref{cond:m2b}-\eqref{cond:m2c}
are respectively equivalent with the following inequalities
\begin{align}
  \lim_{r\rightarrow +\infty} \frac{r^2\, m'''(r) - r\, m''(r) + m'(r)}{m'(r)} & \geq c, \quad \label{cond:m2b-rep}\\
  \lim_{r\rightarrow +\infty} \frac{r^3\, m^{(4)}(r) + 2r^2\, m'''(r) - r\,m''(r)+ m'(r)}{m'(r)} & \geq c;
  \label{cond:m2c-rep}
\end{align}
then according to the hypotheses \eqref{cond:m2}-\eqref{cond:m2a}, the desired result
\eqref{cond:m2b-rep}-\eqref{cond:m2c-rep} can be justified as follows
\begin{align*}
  & \lim_{r\rightarrow +\infty} \frac{r^2\, m'''(r) - r\, m''(r) + m'(r)}{m'(r)} \\
  & = \lim_{r\rightarrow +\infty} \frac{r^2 m'''(r) + (2-\alpha) r\, m''(r)}{m'(r)}
  + (\alpha-3)\lim_{r\rightarrow +\infty} \frac{r\, m''(r) + (1-\alpha) m'(r)}{m'(r)}
  + (3-\alpha)(1-\alpha) +1 \\
  & = (3-\alpha)(1-\alpha) + 1 = (\alpha-2)^2,
\end{align*}
and
\begin{align*}
  & \lim_{r\rightarrow +\infty} \frac{r^3\, m^{(4)}(r) + 2r^2\, m'''(r) - r\,m''(r)+ m'(r)}{m'(r)} \\
  & = \lim_{r\rightarrow + \infty} \frac{r^3 m^{(4)}(r) + (3-\alpha) r^2 m'''(r)}{m'(r)}
  + (\alpha -1) \lim_{r\rightarrow +\infty} \frac{r^2 m'''(r) + (2-\alpha) r\,m''(r)}{m'(r)} \\
  & \quad - \big((\alpha-1)(2-\alpha) + 1\big) \lim_{r\rightarrow +\infty}
  \frac{r m''(r) + (1-\alpha) m'(r)}{m'(r)} + \big((\alpha-1)(2-\alpha) +1\big)(1-\alpha) +1 \\
  & = \alpha (\alpha-2)^2 .
\end{align*}

Let $\chi(\xi)=\chi(|\xi|)\in C^\infty_c(\RR^2)$ be a radial function satisfying \eqref{eq:chi-prop}, and let $\Xi_0>0$ be a large constant such that for every $|\xi|\geq \Xi_0$,
\begin{align}\label{cond:m'-b}
  \Delta_\xi \Big(\frac{m'(|\xi|)}{|\xi|}\Big) \geq \frac{c}{2} \frac{m'(|\xi|)}{|\xi|^3},
  \quad \Delta_\xi \Big(\frac{m'(|\xi|) + |\xi| m''(|\xi|)}{|\xi|} \Big) \geq \frac{c}{2} \frac{m'(|\xi|)}{|\xi|^3}.
\end{align}
Denote by $\chi_\mathrm{R}(\xi) \triangleq \chi(\frac{|\xi|}{\mathrm{R}})$ with $0<\mathrm{R}< \frac{\Xi_0}{2}$ a fixed constant.
From \eqref{eq:G-exp1}, we get
\begin{align*}
  G(\rho) & = \frac{1}{2\pi} m(0^+) + \frac{1}{(2\pi)^2} \int_{\RR^2} e^{ix\cdot\xi} \chi_\mathrm{R}(|\xi|) 
  \frac{m'(|\xi|)}{|\xi|} \dd \xi
  + \frac{1}{(2\pi)^2} \int_{\RR^2} e^{ix\cdot\xi} \big( 1- \chi_\mathrm{R}(|\xi|)\big) \frac{m'(|\xi|)}{|\xi|} \dd \xi \\
  & \triangleq \frac{1}{2\pi} m(0^+) + L_1 + L_2. 
\end{align*}
For every $\rho =|x|>0$ small enough so that $\rho\leq \Xi_0^{-1}$, we find
\begin{align*}
  L_1 = \frac{1}{2\pi} \int_0^\infty J_0(\rho r) \chi_\mathrm{R}(r) m'(r) \dd r 
  \geq \frac{1}{4\pi}\,\int_0^\mathrm{R} m'(r) \dd r = \frac{1}{4\pi}\,\big(m(\mathrm{R}) - m(0^+)\big).
\end{align*}
For $L_2$, the integration by parts gives
\begin{align}\label{def:Pxi}
  L_2 = - \frac{1}{(2\pi)^2} \frac{1}{\rho^2} \int_{\RR^2} e^{ix\cdot\xi} P(\xi) \dd \xi,\quad
  \textrm{with}\quad P(\xi) \triangleq \Delta_\xi \Big( \big(1-\chi_\mathrm{R}(|\xi|)\big) \frac{m'(|\xi|)}{|\xi|}\Big) .
\end{align}
Note that from \eqref{cond:m0}-\eqref{eq:MH} and the fact that $r\mapsto m'(r)$ is decreasing for $r>0$ sufficiently large,
\begin{align*}
  \int_{\RR^2} |P(\xi)| \dd \xi
  & \leq C \bigg( \int_{\mathrm{R}\leq |\xi|\leq 2\mathrm{R}} \big(\mathrm{R}^{-1} |\xi|^{-1} + \mathrm{R}^{-2} \big)
  \frac{m'(|\xi|)}{|\xi|} \dd \xi
  + \int_{\mathrm{R}\leq |\xi|} \frac{m'(|\xi|)}{|\xi|^3} \dd \xi \bigg) \\
  & \leq C \mathrm{R}^{-1} m'(\mathrm{R}) < \infty.
\end{align*}
Since $P(\xi)$ is a radial function belonging to $L^1$ and it is mean-free $\int_{\RR^2} P(\xi) \dd \xi =0$, we infer that
\begin{align*}
  L_2 & = -\frac{1}{(2\pi)^2\rho^2}  \int_{\RR^2} \cos (x\cdot \xi) P(\xi) \dd \xi \\
  & = \frac{1}{(2\pi)^2\rho^2}  \int_{\RR^2} \big(1- \cos(x\cdot\xi)\big) P(\xi) \dd \xi
  = \frac{1}{(2\pi)^2\rho^2} \int_{\RR^2} \big(1- \cos(\rho \xi_1)\big) P(\xi) \dd \xi .
\end{align*}
In order to use \eqref{cond:m'-b}, we decompose it as
\begin{align*}
  L_2 = \frac{1}{(2\pi)^2 \rho^2}\int_{|\xi|\leq \Xi_0} \big(1- \cos(\rho \xi_1)\big) P(\xi) \dd \xi
  + \frac{1}{(2\pi)^2 \rho^2} \int_{|\xi|\geq \Xi_0} \big(1-\cos(\rho \xi_1)\big) P(\xi) \dd \xi.
\end{align*}
For the first integral, direct computation gives that for every $0<\rho \leq \Xi_0^{-1}$,
\begin{align*}
  & \frac{1}{(2\pi)^2 \rho^2} \Big|\int_{|\xi| \leq\Xi_0} \big(1-\cos(\rho \xi_1)\big) P(\xi)\dd\xi\Big| \\
  &\leq \frac{C}{\rho^2} \int_{|\xi|\leq \Xi_0} |\rho \xi_1|^2 |P(\xi)| \dd \xi \\
  & \leq C \bigg( \int_{\mathrm{R}\leq |\xi|\leq 2\mathrm{R}} \big(\mathrm{R}^{-1} |\xi|^{-1}
  + \mathrm{R}^{-2}\big) |\xi|\, m'(|\xi|) \dd \xi
  + \int_{\mathrm{R}\leq |\xi|\leq \Xi_0} \frac{m'(|\xi|)}{|\xi|} \dd \xi \bigg) \\
  & \leq C \mathrm{R} m'(\mathrm{R}) + C m(\Xi_0).
\end{align*}
By virtue of \eqref{cond:m'-b} and the support property of $\chi_\mathrm{R}$, we obtain that for every $0<\rho \leq \Xi_0^{-1}$ small enough,
\begin{align*}
  L_2 & \geq - C \big(\mathrm{R} m'(\mathrm{R}) + m(\Xi_0)\big) + \frac{1}{4\pi^2 \rho^2}
  \int_{|\xi|\geq \Xi_0} \big(1-\cos(\rho \xi_1) \big) P(\xi) \dd \xi \\
  & \geq -C  \big(\mathrm{R} m'(\mathrm{R}) + m(\Xi_0)\big) + \frac{1}{4\pi^2 \rho^2} \int_{|\xi|\geq \Xi_0} \big(1- \cos(\rho \xi_1)\big)
  \Delta_\xi \Big(\frac{m'(|\xi|)}{|\xi|}\Big) \dd \xi \\
  & \geq - C  \big(\mathrm{R} m'(\mathrm{R}) + m(\Xi_0)\big) + \frac{c}{8\pi^2 \rho^2} \int_{|\xi|\geq \Xi_0} \big(1- \cos(\rho \xi_1) \big)
  \frac{m'(|\xi|)}{|\xi|^3} \dd \xi \\
  & \geq - C \big(\mathrm{R} m'(\mathrm{R}) + m(\Xi_0)\big) + \frac{c}{8\pi^2 \rho } \int_{|\eta|\geq \rho \Xi_0} \big(1- \cos (\eta_1)\big)
  \frac{m'(\rho^{-1} |\eta|)}{|\eta|^3} \dd \eta \\
  & \geq - C  \big(\mathrm{R} m'(\mathrm{R}) + m(\Xi_0)\big) 
  + \frac{c}{8\pi^2 \rho} \int_{1\leq |\eta|\leq 2} 
  \big(1-\cos(\eta_1)\big) \frac{m'(\rho^{-1} |\eta|)}{|\eta|^3} \dd \eta \\
  & \geq - C \big(\mathrm{R} m'(\mathrm{R}) + m(\Xi_0)\big) + \frac{C_0 c}{\rho}  m'(\rho^{-1}) \\
  & \geq - C  \big(\mathrm{R} m'(\mathrm{R}) + m(\Xi_0)\big) + \frac{C_0 c\,\alpha}{2} m(\rho^{-1}),
\end{align*}
where in the last line we have used that $m'(\rho^{-1}) \geq \frac{\alpha}{2} \rho\, m(\rho^{-1})$ 
for $\rho>0$ small enough (from \eqref{cond:m2-deduc}).
In view of \eqref{eq:m-grow-alp} and the above estimates, we conclude the desired lower bound that
\begin{align*}
  G(\rho)\geq L_2 \geq \frac{C_0 c\, \alpha}{4} m(\rho^{-1}),\qquad \textrm{for $0<\rho \leq \Xi_0^{-1}$ small enough}.
\end{align*}
	
In an analogous way as above, we can prove that for $0<\rho \leq \Xi_0^{-1}$ small enough,
\begin{align}\label{eq:G'-uppbdd}
  G'(\rho) \leq -  C \frac{m(\rho^{-1})}{\rho}.
\end{align}
Indeed, notice that from \eqref{eq:G'-exp},
\begin{align*}
  G'(\rho) & = - \frac{1}{2\pi\rho} \Big(\lim_{r\rightarrow 0^+} r\, m'(r)\Big) - \frac{1}{(2\pi)^2\rho} \int_{\RR^2}
  e^{ix\cdot\xi} \chi_{\mathrm{R}}(|\xi|)  \frac{m'(|\xi|) + |\xi| m''(|\xi|)}{|\xi|} \dd \xi \\
  &\quad - \frac{1}{(2\pi)^2\rho} \int_{\RR^2}  e^{ix\cdot\xi} \big(1- \chi_{\mathrm{R}}(|\xi|) \big)
  \frac{m'(|\xi|) + |\xi| m''(|\xi|)}{|\xi|} \dd \xi \\
  & = - \frac{1}{2\pi \rho} \Big(\lim_{r\rightarrow 0^+} r m'(r)\Big)
  - \frac{1}{2\pi \rho} \int_0^\infty J_0(\rho r) \chi_{\mathrm{R}}(r) \big(r m'(r) \big)' \dd r \\
  & \quad + \frac{1}{(2\pi)^2 \rho^3} \int_{\RR^2}  e^{ix\cdot\xi} Q(\xi) \dd \xi,\quad \textrm{with}\;\;
  Q(\xi)\triangleq \Delta_\xi \bigg(\big(1- \chi_{\mathrm{R}}(|\xi|) \big)
  \frac{m'(|\xi|) + |\xi| m''(|\xi|)}{|\xi|} \bigg),
\end{align*}
and
\begin{align*}
  & \frac{1}{(2\pi)^2\rho^3} \int_{\RR^2} e^{ix\cdot\xi} Q(\xi) \dd \xi = - \frac{1}{(2\pi)^2 \rho^3}
  \int_{\RR^2} \big(1- \cos(\rho \xi_1) \big) Q(\xi) \dd \xi \\
  & = - \frac{1}{(2\pi)^2 \rho^3}\int_{|\xi|\leq \Xi_0} \big(1- \cos(\rho \xi_1)\big) Q(\xi) \dd \xi
  - \frac{1}{(2\pi)^2 \rho^3} \int_{|\xi|\geq \Xi_0} \big(1-\cos(\rho \xi_1)\big) Q(\xi) \dd \xi,
\end{align*}
due to that \eqref{cond:m'-b} is satisfied and 
\begin{align*}
\Big|\frac{1}{2\pi \rho} \int_0^\infty J_0(\rho r) \chi_{\mathrm{R}}(r) \big(r m'(r) \big)' \dd r\Big|\le \frac{C}{2\pi \rho}\int_0^{\mathrm{R}}m'(r)\dd r\le C\frac{m(R)}{\rho},
\end{align*}
the desired result \eqref{eq:G'-uppbdd}
can be obtained via the similar deduction as above.
As a direct consequence of \eqref{eq:G'-uppbdd}, the function $\rho\mapsto G(\rho)$ is decreasing
for $\rho>0$ small enough.

Next, in view of \eqref{eq:G-exp1} and \eqref{eq:G'-exp}, we see that
\begin{align*}
  \rho\, G'(\rho) + \alpha\, G(\rho) & =
   - \frac{1}{2\pi}\lim_{r\rightarrow 0^+} r m'(r) + \frac{\alpha}{2\pi}m(0^+)
   - \Psi(\rho),
\end{align*}
with
\begin{align*}
  \Psi(\rho) \triangleq \frac{1}{2\pi} \int_0^\infty J_0(\rho r) \big( (1-\alpha) m'(r) + r m''(r)\big) \dd r 
  = \frac{1}{(2\pi)^2} \int_{\mathbb{R}^2} e^{i x \cdot \xi}
  \frac{(1-\alpha) m'(|\xi|) + |\xi| m''(|\xi|)}{|\xi|} \dd\xi.
\end{align*}
Below let us consider the upper bound of $|\Psi(\rho)|$.
Recalling that $\chi\in C_c^\infty (\mathbb{R}^2)$ is a cut-off function satisfying \eqref{eq:chi-prop}, we have
\begin{align*}
  \Psi(\rho) & =  \frac{1}{2\pi} \int_0^\infty J_0(\rho r) \chi(\rho r) \big((1-\alpha) m'(r) + r\, m''(r)\big) \dd r \\
  & \quad + \frac{1}{(2\pi)^2} \int_{\mathbb{R}^2} e^{i x\cdot\xi} \big(1-\chi(\rho |\xi|) \big)
  \frac{(1-\alpha)m'(|\xi|) + |\xi| m''(|\xi|)}{|\xi|} \dd \xi
  \triangleq \Psi_1(\rho) + \Psi_2(\rho).
\end{align*}
Denote by $M(r) \triangleq (1-\alpha) m'(r) + r\, m''(r)$. For every $\epsilon >0$, according to \eqref{cond:m2}-\eqref{cond:m2a}, there exists a constant
$\mathrm{R}_0 = \mathrm{R}_0(\epsilon)>0$ such that
\begin{align}\label{eq:Mr-fact}
  \frac{|M(r)|}{m'(r)} + \frac{|r\,M'(r)|}{m'(r)}  + \frac{|r^2\, M''(r)|}{m'(r)} \leq \epsilon,\quad \forall r> \mathrm{R}_0.
\end{align}
Let $\rho>0$ be small enough such that $0< \rho \leq \mathrm{R}_0^{-1}$. For $\Psi_1(\rho)$, by \eqref{eq:MH} and \eqref{eq:Mr-fact}, we have
\begin{align*}
  |\Psi_1(\rho)| & = \frac{1}{2\pi} \Big|\int_0^\infty J_0(\rho r) \chi(\rho r)
  M(r) \dd r \Big| \\
  & \leq C \Big( \int_0^{\mathrm{R}_0} |M(r)| \dd r + \int_{\mathrm{R}_0}^{2 \rho^{-1}} |M(r)| \dd r \Big) \\
  & \leq C \Big( \int_0^{\mathrm{R}_0} m'(r) \dd r + \epsilon \int_{\mathrm{R}_0}^{2 \rho^{-1}} m'(r) \dd r \Big) \\
  & \leq C m(\mathrm{R}_0)  + C\, \epsilon\, m(2 \rho^{-1}) .
\end{align*}
For $\Psi_2(\rho)$, using \eqref{eq:Mr-fact} and the integration by parts, we find
\begin{align*}
  |\Psi_2(\rho)| & = \frac{1}{(2\pi)^2} \Big|\int_{\mathbb{R}^2} e^{i x\cdot\xi} \big(1- \chi(\rho|\xi|)\big)
  \frac{M(|\xi|)}{|\xi|} \dd \xi \Big| \\
  & = \frac{1}{(2\pi)^2} \frac{1}{\rho^2} \Big| \int_{\mathbb{R}^2} e^{i x\cdot \xi} \Delta_\xi \Big( \big(1- \chi(\rho |\xi|) \big)
  \frac{M(|\xi|)}{|\xi|} \Big) \dd \xi \Big| \\
  & \leq \frac{1}{2\pi} \frac{1}{\rho^2} \int_0^\infty \Big| \Big(\partial_r^2 + \frac{\partial_r}{r} \Big)
  \Big(\big(1-\chi(\rho r) \big) \frac{M(r)}{r} \Big) \Big| r\, \dd r \\
  & \leq \frac{C}{\rho^2} \int_{\rho^{-1}}^{2 \rho^{-1}}\Big( \rho^2 \frac{|M(r)|}{r}
  + \rho \frac{|r \, M'(r) - M(r)|}{r^2} \Big) r \dd r \\
  & \quad + \frac{C}{\rho^2} \int_{\rho^{-1}}^\infty \frac{| r^2 M''(r) + r \,M'(r) + M(r) |}{r^2} \dd r \\
  & \leq C \, \epsilon \Big( \int_{\rho^{-1}}^{2 \rho^{-1}} m'(r) \dd r
  + \frac{1}{\rho^2} \int_{\rho^{-1}}^\infty \frac{m'(r)}{r^2} \dd r  \Big) \\
  & \leq C\, \epsilon\, \Big( m(2 \rho^{-1}) 
  + \frac{2}{\rho^2} \int_{\rho^{-1}}^\infty \frac{m(r)}{r^3} \dd r \Big) \\
  & \leq C \, \epsilon \Big( m(\rho^{-1}) + \frac{m(\rho^{-1}) }{\rho} 
  \int_{\rho^{-1}}^\infty \frac{1}{r^2} \dd r \Big)\leq C \,\epsilon\, m(\rho^{-1}),
\end{align*}
where in the last line we have used \eqref{ineq:m-bound} and the fact that $r\mapsto r^{-1}m(r) $ 
is decreasing for $r>0$ sufficiently large.
Gathering the above estimates yields that for every $0<\rho \leq \mathrm{R}_0^{-1}$ and for any $\epsilon>0$,
\begin{align*}
  |\rho G'(\rho) + \alpha G(\rho)| \leq C + C m(\mathrm{R}_0) + C\,\epsilon\, m(2 \rho^{-1}).
\end{align*}
In combination with \eqref{eq:G-prop1} and \eqref{eq:m-grow-alp}, we get
\begin{align*}
  \lim_{\rho\rightarrow 0^+} \frac{|\rho G'(\rho) + \alpha G(\rho)|}{G(\rho)} \leq C\, \epsilon,
\end{align*}
which implies the estimate \eqref{eq:G-prop3}, as desired.

Finally, by using the property \eqref{eq:G-prop3}, we prove that $\lim\limits_{\rho\rightarrow 0^+} \frac{l^\alpha G(l\rho)}{G(\rho)} = 1$
for every $l>0$. In fact, we also assume $l>1$ without loss of generality, and from \eqref{eq:G-prop3} we deduce that for every
$\epsilon>0$ there exists a small constant $\rho_0>0$ such that
\begin{align*}
  |\rho G'(\rho) + \alpha G(\rho)| \leq \epsilon \, G(\rho),\quad \forall 0<\rho \leq \rho_0,
\end{align*}
then noting that
\begin{align*}
  (l\rho)^\alpha G(l\rho) - \rho^\alpha G(\rho) = \int_\rho^{l\rho} \big(\alpha s^{\alpha-1} G(s) + s^\alpha G'(s) \big) \dd s,
\end{align*}
we find that for every $0<\rho < \min\{\frac{\rho_0}{l}, \frac{\bar{c}_0}{l} \}$,
\begin{align*}
  \Big|\frac{l^\alpha G(l\rho)}{G(\rho)} -1 \Big|
  \leq \frac{1}{\rho^\alpha G(\rho)} \int_\rho^{l\rho} s^{\alpha-1} | s G'(s) + \alpha G(s)| \dd s 
  \leq \frac{\epsilon}{\rho^\alpha} \int_\rho^{l\rho} s^{\alpha-1} \dd s = \frac{l^\alpha -1}{\alpha} \epsilon,
\end{align*}
where in the second inequality we have used the property \eqref{eq:G-prop1c}. 
Hence, the desired equality in \eqref{eq:G-prop4} follows immediately.

Therefore, the wanted properties of $G(\rho)$ in \eqref{eq:G-prop1}-\eqref{eq:G-prop4}
have been verified in both cases and we complete the proof of Lemma \ref{lem:mD-cond}.
\end{proof}

\section{Local regularity for the generalized SQG patches}\label{sec:loc-reg}
In this section, we establish a local-in-time well-posedness theory for the generalized SQG equation \eqref{eq:geSQG}$\&$\eqref{eq:u-exp}. This theory has been previously developed in the works \cite{KYZ17,GanP21} for $\alpha$-SQG equation with $\alpha\in(0,\frac13)$. Our goal is to extend this theory to general kernels $G$.

Let us proceed by stating the following assumptions on $G$.

\begin{enumerate}[label=$(\mathbf{A}1)$,ref=$\mathbf{A}1$]
\item\label{A1} Assume that $G: \mathbb{R}_+\rightarrow \mathbb{R}$ is a continuously differentiable function, 
and there exists $\alpha\in (0,\frac{1}{3})$ when $\mathbf{D} = \mathbb{R}^2_+$ or $\alpha\in (0,1)$ when $\mathbf{D} =\mathbb{R}^2$, such that 
\begin{align}\label{conds:G-s}
  \textrm{on}\;\,(0, c_0),\qquad |G(\rho)|\leq C_0 \rho^{-\alpha},\quad |G'(\rho)|\leq C_0 \rho^{-1-\alpha},
  \quad
\end{align}
and
\begin{align}\label{conds:G-l}
	\textrm{on}\;\,[c_0,+\infty),\qquad |G(\rho)|\leq C_0,\quad |G'(\rho)|\leq C_0,
\end{align}
with some constants $c_0,C_0>0$. 
\end{enumerate}

In Section \ref{subsec:patch-sol}, we define the patch solution for the equation \eqref{eq:geSQG}$\&$\eqref{eq:u-exp}, 
and establish a useful relationship between the patch solution and the flow map $\Phi_t$. 
In Section \ref{subsec:contour}, we derive a contour dynamics equation associated with the patch solution for the equation 
\eqref{eq:geSQG}$\&$\eqref{eq:u-exp}. Then, in Section \ref{subsec:loc-reg}, by examining the $H^2$-solvability of the contour equation, 
we prove the local-in-time existence and uniqueness of the patch solution in $H^2$ for the equation \eqref{eq:geSQG}$\&$\eqref{eq:u-exp}, 
under the assumptions \eqref{A1} on $G$.

\subsection{Patch solution}\label{subsec:patch-sol}

We first introduce the definition of patch solution the equation \eqref{eq:geSQG}$\&$\eqref{eq:u-exp}.
\begin{definition}\label{def:patch-sol}
  Let $\mathbf{D} = \mathbb{R}^2$ or $\mathbb{R}^2_+$. Denote by $d_\mathcal{H}(\Gamma,\widetilde{\Gamma})$
the Hausdorff distance between two sets $\Gamma,\widetilde{\Gamma}\subseteq \mathbb{R}^2$,
and for a set $\Gamma\subseteq \mathbb{R}^2$, a vector field $v:\Gamma\rightarrow \mathbb{R}^2$, and $h\in\mathbb{R}$,
denote $X_v^h[\Gamma] \triangleq \{x+ h v(x) \,:\, x\in \Gamma \}$.
Let $a_1, \cdots, a_N\in \mathbb{R}\setminus \{0\}$, and let $D_1(t)$, $\cdots, D_N(t) \subseteq \mathbf{D}$ 
for every $t\in [0,T]$ be pairwise-disjoint bounded open sets where each boundary 
$\partial D_j(t)$ is a simple closed curve
and is also continuous in $t\in [0,T]$ with respect to $d_\mathcal{H}$. Denote $D(t) \triangleq \cup_{j=1}^N D_j(t)$.
Let 
\begin{align}\label{eq:the-patch-sol}
  \theta(x,t) = \sum_{j=1}^N a_j \mathbf{1}_{D_j(t)}(x) . 
\end{align}
and $u$ be defined from \eqref{eq:u-exp}.
If for each $t\in (0,T]$ we have 
\begin{align}\label{eq:def-patch}
  \lim_{h\rightarrow 0}  \frac{d_\mathcal{H}\big(\partial D(t+h), X_{u(\cdot,t)}^h[\partial D(t)] \big)}{h}
  =0,
\end{align}
then $\theta$ is called a patch solution to the equation \eqref{eq:geSQG}$\&$\eqref{eq:u-exp} on $\mathbf{D}\times [0,T]$.
If $\partial D_j(t)$ also belongs to $C^{n,\sigma}$ (resp. $H^n$) for each $j\in\{1,\cdots,N\}$ and $t\in [0,T]$,
then $\theta$ is a $C^{n,\sigma}$ (resp. $H^n$) patch solution to the equation 
\eqref{eq:geSQG}$\&$\eqref{eq:u-exp} on $\mathbf{D}\times [0,T]$. 
\end{definition}

\begin{remark}
Let $\Phi_t: \mathbf{D} \rightarrow \mathbb{R}^2$ be the flow map generated by the velocity $u$ which solves
\begin{align}\label{eq:flow_map}
  \frac{\dd }{\dd t} \Phi_t(x) = u(\Phi_t(x),t),\quad \Phi_t(x)|_{t=0} = x.
\end{align}
If $\theta$ given by \eqref{eq:the-patch-sol} satisfies that each patch
$D_j(t)$ has pairwise disjoint closure with its boundary 
$\partial D_j(t)$ has simple closed curve
and also $\partial D_j(t) = \Phi_t(\partial D_j(0))$ 
for each $j\in \{1,\cdots,N\}$ 
and $t\in [0,T]$, then by virtue of the H\"older continuity of $u$ 
(from Lemma \ref{lem:u-point-es} below) and compactness of 
$\partial D_j(t)$, we have that $\theta$ is a patch solution to the equation 
\eqref{eq:geSQG}$\&$\eqref{eq:u-exp} on $\mathbf{D}\times [0,T]$.
Moreover, if $\partial D(t)$ belongs to $C^1$ and $n(x,t)$ is the outer unit normal vector at $x\in \partial D(t)$,
then \eqref{eq:def-patch} is equivalent to the motion of $\partial D(t)$ with the normal velocity $u(x,t)\cdot n(x,t)$
at each $x\in \partial D(t)$.
\end{remark}

Next, motivated by \cite{KYZ17}, we build some important relationship of patch
solution to the flow map $\Phi_t$ for the equation 
\eqref{eq:geSQG}$\&$\eqref{eq:u-exp}, which will play a key role 
in the finite-time singularity part.
\begin{proposition}\label{prop:flow-map}
Let $\theta$ given by \eqref{eq:the-patch-sol} be the patch solution on $[0,T]$ 
satisfying the assumptions in Definition \ref{def:patch-sol}.
Let $x\in \overline{\mathbf{D}}\setminus \partial D(0)$ and $t_{\theta,x}\in [0,T)$ 
be the maximal time such that the solution of  \eqref{eq:flow_map} 
with $u$ defined by \eqref{eq:u-exp} 
satisfies $\Phi_t(x)\in \overline{\mathbf{D}}\setminus \partial D(t)$ for each $t\in [0,t_{\theta,x})$. 
\begin{enumerate}
\item[(i)] 
If the assumptions \eqref{A1} with $\alpha\in (0,\tfrac{1}{2})$ are assumed, $\sigma\in (\frac{\alpha}{1-\alpha},1]$, 
and $\theta$ is a $C^{1,\sigma}$ patch solution to the equation
\eqref{eq:geSQG}$\&$\eqref{eq:u-exp} on $[0,T)$, 
then $t_{\theta,x}=T$ for each $x\in \overline{\mathbf{D}}\setminus \partial D(0)$ and 
\begin{align*}
  \quad\Phi_t: \overline{\mathbf{D}} \setminus \partial D(0) \to \overline{\mathbf{D}}\setminus \partial D(t)
  \;\textrm{is a bijection for each $t\in [0,T)$}.
\end{align*}

\item[(ii)] If the assumptions \eqref{A1} with $\alpha\in (0,1)$ are assumed, 
$t_{\theta,x}=T$ for each $x\in \overline{\mathbf{D}}\setminus \partial D(0)$, 
and $\Phi_t:\overline{\mathbf{D}}\setminus \partial D(0) \to \overline{\mathbf{D}}\setminus \partial D(t)$ 
is a bijection for each $t\in [0,T)$, then $\theta$ is a patch solution to the equation \eqref{eq:geSQG}$\&$\eqref{eq:u-exp} on $[0,T)$. 
Moreover, $\Phi_t$ is measure preserving on $\overline{\mathbf{D}}\setminus \partial D(0)$ and it also maps each $D_j(0)$ to $D_j(t)$ 
as well as $\overline{\mathbf{D}}\setminus \overline{D(0)}$ onto $\overline{\mathbf{D}}\setminus \overline{D(t)}$. 
At last, for each $j\in\{1,\cdots,N\}$ and $t\in [0,T)$ we have 
\begin{align*}
  \Phi_t(\partial D_j(0))=\partial D_j(t),
\end{align*}
in the sense that any solution of \eqref{eq:flow_map} with $x\in \partial D_j(0)$ 
satisfies $\Phi_t(x)\in \partial D_j(t)$, and for each $y\in \partial D_j(t)$, 
there is $x\in \partial D_j(0)$ and a solution of \eqref{eq:flow_map} such that $\Phi_t(x)=y$. 
\end{enumerate}
\end{proposition}

In the $\alpha$-SQG case, which corresponds to $G(\rho) = \alpha \mathrm{c}_\alpha \rho^{-\alpha}$, this result was proved by Proposition 1.3 in \cite{KYZ17}.  
Hence, Proposition \ref{prop:flow-map} can be viewed as a suitable generalization to general $G$ satisfying conditions 
\eqref{conds:G-s}-\eqref{conds:G-l}. For the proof of Proposition \ref{prop:flow-map}, one can see the section \ref{subsec:flow-map}.

\subsection{Contour dynamics equation}\label{subsec:contour}

Let $z_k(\zeta,t)$ ($k=1,\cdots,N$) with $\zeta\in \TT$ be a parametrization of the patch boundary $\partial D_k(t) $,
where each $z_k(\cdot,t)$ is running counterclockwise along $\partial D_k(t)$.
We assume that $z_k(\zeta,0)$ with each $k=1,\cdots,N$ belongs to $H^2(\mathbb{T})$ and is non-degenerate.

We mainly consider the case $\mathbf{D} =\mathbb{R}^2_+$. 
For $x\in \partial D_k(t)$, denote by $n(x,t)$ the outer unit normal vector of $D_k(t)$ at $x$.
From \eqref{eq:u-exp}, and using Gauss-Green's theorem and the following simple facts 
$u^\perp \cdot v = - u\cdot v^\perp$, $u\cdot \overline{v} = \overline{u}\cdot v$ and $\overline{n}^\perp = - \overline{n^\perp}$, 
we have
\begin{align}
  u_n(x,t) & = u(x,t) \cdot n(x,t) \notag\\
  & = - \sum_{j=1}^N a_j \int_{D_j(t)} \left(G(|x-y|)\frac{(x-y)\cdot n(x,t)^\perp }{|x-y|^2}
  - G(|x-\overline{y}|) \frac{(x-\overline{y})\cdot n(x,t)^\perp}{|x-\overline{y}|^2}\right) \dd y \notag\\
  & =  - \sum_{j=1}^N a_j \int_{D_j(t)} \bigg(G(|x-y|)\frac{(x-y)\cdot n(x,t)^\perp }{|x-y|^2}
  - G(|\overline{x}- y|) \frac{(\overline{x}- y)\cdot \overline{n(x,t)^\perp}}{|\overline{x}-y|^2}\bigg) \dd y \notag\\
  & = - \sum_{j=1}^N a_j \int_{D_j(t)} \nabla_y \cdot
  \Big( R(|x-y|) n(x,t)^\perp + R(|\overline{x} -y|) \overline{n(x,t)}^\perp \Big) \dd y \notag\\
  & = \sum_{j=1}^N a_j \int_{\partial D_j(t)} \Big( R(|x-y|) n(y,t)^\perp + R(|x - \overline{y}|)
  \overline{n(y,t)^\perp}\Big) \cdot n(x,t)\, \dd \sigma(y) \label{eq:velocity_normal},
\end{align}
where $R(\rho)$ is the primitive function of $-\frac{G(\rho)}{\rho}$ (recalling that $G(\cdot)$ is given by
\eqref{def:G}), i.e.
\begin{equation}\label{def:R}
  R(\rho)  \triangleq  \int_\rho^1 \frac{G(s)}{s} \dd s + C
  = \phi(\rho) - \phi(1) + C,
\end{equation}
with $\phi(\rho)$ given by \eqref{eq:phi-exp1} and $C\in \mathbb{R}$ some constant chosen for convenience.
For the 2D Euler equation, it was chosen as
$R(\rho) = - \frac{1}{2\pi} \log \rho$ in \cite[Eq. (8.62)]{MA02};
while for the $\alpha$-SQG equation, it was selected as $R(\rho)= \mathrm{c}_\alpha \rho^{-\alpha}$ in \cite[Eq. (2.5)]{KRYZ}
(the authors have dropped the positive constant $\mathrm{c}_\alpha$).

By using the parametrization of $\partial D_k(t)$, we find
\begin{align*}
  u_n(x,t) & = - \sum_{j=1}^N a_j \int_\mathbb{T} \Big( R\big(|z_k(\zeta,t) - z_j(\zeta-\eta,t)| \big) \partial_\zeta z_j(\zeta -\eta,t) \\
  & \qquad \qquad \qquad \quad + R\big(|z_k(\zeta,t) - \overline{z}_j(\zeta-\eta,t)|\big) \partial_\zeta \overline{z}_j(\zeta-\eta,t)\Big) \cdot n(x,t) \dd \eta.
\end{align*}
Since the evolution of patches is solely governed by the normal velocity and one can add any multiple of the tangent vector
$\partial_\zeta z_k(\zeta,t)$ to the velocity $u(x,t)$, 
we will write the \textit{contour dynamics equation} for $\partial D_k(t)$ as follows
\begin{equation}\label{eq:contour}
\begin{aligned}
  \partial_t z_k(\zeta,t) = &\sum_{j=1}^N a_j
  \int_{\mathbb{T}}	\big(\partial_\zeta z_k(\zeta,t)
  - \partial_\zeta z_j(\zeta - \eta, t)\big) R\big(| z_k(\zeta,t)  - z_j(\zeta-\eta, t)|\big)  \dd\eta \\
  & + \sum_{j=1}^N a_j
  \int_{\mathbb{T}}	\big(\partial_\zeta z_k(\zeta,t)
  - \partial_\zeta \overline{z}_j(\zeta - \eta, t)\big) R\big(| z_k(\zeta,t)  - \overline{z}_j(\zeta-\eta, t)|\big) \dd\eta.
\end{aligned}
\end{equation}
Since $z_k$ is periodic in $\TT$, the choice of $C$ in \eqref{def:R} does not change the equation \eqref{eq:contour}.

If in the whole space $\mathbf{D}=\mathbb{R}^2$, we analogously get the \textit{contour dynamics equation} as
\begin{equation}\label{eq:contourR2}
\begin{aligned}
  \partial_t z_k(\zeta,t) =  \sum_{j=1}^N a_j
  \int_{\mathbb{T}}	\big(\partial_\zeta z_k(\zeta,t)
  - \partial_\zeta z_j(\zeta - \eta, t)\big) R\big(| z_k(\zeta,t)  - z_j(\zeta-\eta, t)|\big)  \dd\eta .
\end{aligned}
\end{equation}

Moreover, we will use the contour parametrization depending only on time, i.e. $|\partial_\zeta z_j(\zeta,t)|^2 = A_j(t)$
for each $j=1,2,\cdots, N$, and exactly arguing as in \cite[Sec. 2]{GanP21},
we find that the \textit{contour dynamics equation} in $\mathbf{D}=\mathbb{R}^2_+$ now becomes
\begin{align}\label{eq:main-eq-GP}
  \partial_t z_k(\zeta,t) = \mathrm{NL}_k(\zeta,t) + \lambda_k(\zeta,t) \partial_{\zeta} z_k(\zeta,t),
\end{align}
where
\begin{equation}\label{eq:NL-k-GP}
\begin{aligned}
  \mathrm{NL}_k(\zeta,t) \triangleq &\sum_{j=1}^N a_j \int_{\mathbb{T}} \big(\partial_\zeta z_k(\zeta,t)
  - \partial_\zeta z_j(\zeta - \eta, t)\big) R\big(| z_k(\zeta,t)  - z_j(\zeta-\eta, t)|\big)  \dd\eta \\
  & + \sum_{j=1}^N a_j \int_{\mathbb{T}} \big(\partial_\zeta z_k(\zeta,t)
  - \partial_\zeta \overline{z}_j(\zeta - \eta, t)\big) R\big(| z_k(\zeta,t)  - \overline{z}_j(\zeta -\eta, t)|\big)  \dd\eta,
\end{aligned}
\end{equation}
and
\begin{align}\label{eq:lambda-def}
  \lambda_k(\zeta,t) \triangleq \frac{\zeta+\pi}{2\pi} \int_{\TT}\frac{\partial_\eta z_k(\eta,t)\cdot \partial_\eta
  \mathrm{NL}_k(\eta,t)}{A_k(t)}\dd \eta - \int_{-\pi}^\zeta \frac{\partial_\eta z_k(\eta,t)\cdot
  \partial_\eta \mathrm{NL}_k(\eta,t)}{A_k(t)} \dd \eta,
\end{align}
and $A_k(t)= |\partial_\zeta z_k(\zeta,t)|^2$, $k=1,2,\cdots,N$. If $\mathbf{D} = \mathbb{R}^2 $, 
the contour equation is also \eqref{eq:main-eq-GP} with keeping only the first integral in $\mathrm{NL}_k(\zeta,t)$.

Now we introduce some notations used in this whole section. Denote by $\mathbf{z} \triangleq (z_1,z_2,\cdots,z_N)$
and define the arc-chord term that
\begin{equation}\label{def:F[z_k]}
\begin{split}
  F[z_k](\zeta,\eta,t) \triangleq \,& \frac{|\eta|}{|z_k(\zeta,t) - z_k(\zeta-\eta,t)|},\quad \zeta,\eta\in \TT,\\
  F[z_k](\zeta,0,t) \triangleq \, & |\partial_\zeta z_k(\zeta,t)|^{-1};
\end{split}
\end{equation}
and also
\begin{align}\label{def:del-z}
  \delta[\mathbf{z}](t) \triangleq \min_{i\ne j} \min_{\zeta,\eta\in \TT}|z_i(\zeta,t) - z_j(\eta,t)|.
\end{align}
In the sequel, we denote by $\mathbf{W}$ the following set
\begin{align}\label{def:norm-whole}
  \mathbf{W} \triangleq \Big\{ \mathbf{z}=(z_1,\cdots,z_N)\, :\, 
  \lVert \mathbf{z}\rVert_{W} \triangleq \lVert \mathbf{z} \rVert^2_{H^2(\TT)}
  +  \sum_{k=1}^N \Big( \lVert F[z_k]\rVert_{L^\infty(\TT^2)} \Big) + \frac{1}{\delta[\mathbf{z}]} <\infty\Big\} .
\end{align}
In addition, we suppose that the initial data $\mathbf{z}|_{t=0} =\mathbf{z}_0$
satisfy that $\lVert \mathbf{z}_0\rVert_W < +\infty$.

\subsection{Local \texorpdfstring{$H^2$}{H2}-solvability for the contour equation}\label{subsec:loc-reg}

By applying the similar arguments as in \cite{GanP21,KYZ17}, we can prove the following local-in-time regularity result
for the patch solution of the equation \eqref{eq:geSQG}$\&$\eqref{eq:u-exp}.
\begin{theorem}\label{thm:loc-reg}
Let $\mathbf{D}=\mathbb{R}^2_+$. 
Suppose that $G(\rho)$ satisfies the assumptions \eqref{A1} with $\alpha\in (0,\frac{1}{3})$.
Then for each non self-intersect $H^2$ patch-like initial data $\theta_0$ given by \eqref{eq:patch-data},
there exists a unique local $H^2$ patch solution $\theta$ to the inviscid generalized SQG equation
\eqref{eq:geSQG}$\&$\eqref{eq:u-exp} associated with $\theta(\cdot,0)=\theta_0$.
\end{theorem}

\begin{remark}\label{rem:locregR2}
If $\mathbf{D} = \mathbb{R}^2$, since the contour dynamics equation becomes \eqref{eq:contourR2} or  
\eqref{eq:main-eq-GP} with removing the second integral term in $\mathrm{NL}_k(\zeta,t)$, 
Theorem \ref{thm:loc-reg} will hold true under the assumptions \eqref{A1} with $0<\alpha <1$.
If additionally we consider the $H^3$ patch solution, 
analogously as \cite[Theorem 4]{GanP21}, we can show the local regularity result 
for the equation \eqref{eq:geSQG}$\&$\eqref{eq:u-exp} under assumptions
\eqref{A1} with $1\leq \alpha <2$.
\end{remark}


The following simple properties of the function $R(\cdot)$ defined by \eqref{def:R} 
appearing in \eqref{eq:NL-k-GP} will be repeatedly used in this section.
\begin{lemma}\label{lem:G-lemma}
If $G(\rho)$ is a continuously differentiable function satisfying \eqref{conds:G-s} and \eqref{conds:G-l} 
with $\alpha>0$, then we have that for every $r>0$,
\begin{align}\label{es:R-deri}
  |R(\rho)|\leq C \max\{r^{-\alpha}, |\log \rho| \} ,\quad
  R'(\rho) \leq C \max\big\{ \rho^{-1-\alpha}, \rho^{-1} \big\},
\end{align}
and
\begin{align}\label{es:R-deri2}
  |R''(\rho)| \leq C \max\{\rho^{-2-\alpha}, \rho^{-1} \} ,
\end{align}
where $C>0$ is a constant depending only on $G$.
\end{lemma}

\begin{proof}
Recalling the formula of $R(\rho)$ given by \eqref{def:R}, and under the assumptions \eqref{conds:G-s}-\eqref{conds:G-l},
it is straightforward to check that for every $\rho\in (0,c_0)$,
\begin{equation*}
  |R(\rho)|\leq C \rho^{-\alpha}, \quad |R'(\rho)|\leq C \rho^{-1-\alpha},\quad |R''(\rho)|\leq C \rho^{-2-\alpha},
\end{equation*}
and for every $r\in [c_0,+\infty)$,
\begin{equation*}
  |R(\rho)|\leq C |\log \rho|,\quad |R'(\rho)|\leq  \frac{C}{\rho},\quad |R''(\rho)|\leq \frac{C}{\rho}.
\end{equation*}
Combining the above estimates leads to the desired inequalities \eqref{es:R-deri}-\eqref{es:R-deri2}.
\end{proof}

The next proposition, whose proof is placed in the appendix section, 
is concerned with the \textit{a priori} estimates of $\mathbf{z}=(z_1,\cdots,z_N)$
solving the contour dynamics equation \eqref{eq:main-eq-GP}-\eqref{eq:lambda-def}.
\begin{proposition}\label{pro:ap-es}
Under the assumptions of Theorem \ref{thm:loc-reg}, then there exists a polynomial function $\mathcal{P}(\cdot)$ such that
\begin{align}\label{ineq:main-estimate}
  \frac{\dd }{\dd t}\lVert \mathbf{z}\rVert_{W}\leq \mathcal{P}(\lVert \mathbf{z}\rVert_{W}).
\end{align}
\end{proposition}

The following result is about the $H^2$-regularity for the change of variables between two contour parametrizations.
For the detailed proof, one can see the section \ref{subsec:reg-param} below. 
\begin{proposition}\label{prop:reg-param}
Suppose that $\mathbf{z} = (z_1,\cdots,z_N)$ is a solution of contour dynamics equation
$\mathrm{\eqref{eq:main-eq-GP}\text{-}\eqref{eq:lambda-def}}$
with $\mathbf{z}\in C([0,T];\mathbf{W})$. Let $\mathbf{y}=(y_1,\cdots,y_N) \in C([0,T]; H^2)$ be
a contour parametrization satisfying \eqref{eq:contour} and for every $k\in\{1,2,\cdots,N\}$,
\begin{align}\label{eq:zkyk-exp}
  z_k(\zeta,t) = y_k\big(\phi_k(\zeta,t),t\big).
\end{align}
Then the change of parametrization $\phi_k(\zeta,t)-\zeta\in C\big([0,T];H^2(\TT)\big)$.
\end{proposition}

The uniqueness result of the contour dynamics equation \eqref{eq:contour} in $C([0,T];\mathbf{W})$ 
is exhibited as follows.
\begin{proposition}\label{pro:unique}
Let $G(\rho)$ be a continuously differentiable function satisfying \eqref{conds:G-s} and \eqref{conds:G-l}.
Suppose that $\{z_k(\zeta,t)\}_{k=1,2,\cdots,N}$ and $\{y_k(\zeta,t)\}_{k=1,2,\cdots,N}$
are both solutions to the contour dynamics equation \eqref{eq:contour} in $C([0,T];\mathbf{W})$ with initial data
$z_k(\zeta,0) = y_k(\zeta,0)$. Define $w_k(\zeta,t)\triangleq z_k(\zeta,t) - y_k(\zeta,t)$.
Then we have
\begin{align*}
  \frac{\dd }{\dd t}\Big(\sum_{k=1}^N \lVert w_k \rVert^2_{L^2}\Big)
  \leq C \sum_{k=1}^N \lVert w_k \rVert^2_{L^2},
\end{align*}
where the constant $C>0$ depends continuously on $\delta[\mathbf{z}]^{-1}$, $\delta[\mathbf{y}]^{-1}$,
$\lVert (F[z_k],F[y_k])\rVert_{L^\infty}$ and $\lVert (\mathbf{z},\mathbf{y} )\rVert_{H^2}$.
The above inequality together with Gronwall's inequality 
provides the desired uniqueness $z_k\equiv y_k$ on $\mathbb{T} \times [0,T]$ 
for every $k\in \{1,\cdots,N\}$.
\end{proposition}

\begin{proof}[Proof of Proposition \ref{pro:unique}]
For the sake of convenience, we define that for every $k,j\in \{1,2,\cdots,N\}$,
\begin{align}\label{def:Zkj0}
  Z_{k,j}(\zeta,\eta,t) \triangleq z_k(\zeta,t) - \overline{z}_j(\zeta - \eta,t),\quad
  \mathbf{Z}_{k,j}(\zeta,\eta,t) \triangleq z_k(\zeta,t) - z_j(\zeta - \eta, t ),
\end{align}
and
\begin{align*}
  Y_{k,j}(\zeta,\eta,t) \triangleq y_k(\zeta,t) - \overline{y}_j(\zeta-\eta,t),
  \quad \mathbf{Y}_{k,j}(\zeta,\eta,t)  \triangleq y_k(\zeta,t)- y_j(\zeta-\eta,t),
\end{align*}
where $\overline{z}_j(\zeta,t) = \big(z_j^{(1)}, -z_j^{(2)} \big)(\zeta,t)$.
In the sequel, if the variables $(\zeta,\eta,t)$ are clear in the contexts,
we also abbreviate $Z_{k,j}(\zeta,\eta,t)$ and $Y_{k,j}(\zeta,\eta,t)$ as $Z_{k,j}$ and $Y_{k,j}$, respectively.

Using the contour dynamics equation \eqref{eq:contour}, we have
\begin{align*}
  & \frac{1}{2}\frac{\dd }{\dd t}\lVert w_k(t)\rVert^2_{L^2}
  = \int_{\TT} \partial_t w_k(\zeta,t)\cdot w_k(\zeta,t) \dd \zeta \\	
  & = \sum_{j=1}^N a_j \int_{\TT}\int_{\TT} \Big(R(|Z_{k,j}(\zeta,\eta,t)|)
  \partial_\zeta Z_{k,j}(\zeta,\eta,t) - R(|Y_{k,j}(\zeta,\eta,t)|) \partial_\zeta Y_{k,j}(\zeta,\eta,t)
  \Big)\cdot w_k(\zeta) \dd \eta\dd \zeta \\		
  & \quad + \sum_{j=1}^N a_j \int_{\TT}\int_{\TT}
  \Big(R(|\mathbf{Z}_{k,j}(\zeta,\eta,t)|)\partial_\zeta \mathbf{Z}_{k,j}(\zeta,\eta,t)
  - R(|\mathbf{Y}_{k,j}(\zeta,\eta,t)|) \partial_\zeta\mathbf{Y}_{k,j}(\zeta,\eta,t) \Big)
  \cdot w_k(\zeta)\dd \eta\dd \zeta \\
  & \triangleq \mathbf{L}_1 + \mathbf{L}_2 .
\end{align*}
For the term $\mathbf{L}_1$, we further split it as follows
\begin{align*}
  \mathbf{L}_1 
  & = \sum_{j=1}^N a_j \int_{\TT} \int_{\TT}  R(|Z_{k,j}|)(\partial_\zeta z_k(\zeta)
  - \partial_\zeta y_k(\zeta))\cdot w_k(\zeta)\dd \eta\dd \zeta \\
  & \quad  -\sum_{j=1}^N a_j \int_{\TT} \int_{\TT} R(|Z_{k,j}|)(\partial_\zeta \overline{z}_j(\zeta - \eta)
  - \partial_\zeta \overline{y}_j(\zeta - \eta))\cdot w_k(\zeta) \dd \eta\dd \zeta \\
  & \quad + \sum_{j=1}^N a_j \int_{\TT} \int_{\TT} \Big(\big(R(|Z_{k,j}|) - R(|Y_{k,j}|)
  \big)\partial_\zeta Y_{k,j} \Big)\cdot w_k(\zeta)\dd \eta \dd \zeta  \\
  & \triangleq \mathbf{L}_{11} + \mathbf{L}_{12} + \mathbf{L}_{13} .
\end{align*}
For $\mathbf{L}_{11}$, via the integrating by parts, and using Lemma \ref{lem:G-lemma} 
and \eqref{eq:Zkk-fact}, \eqref{eq:diff-pointwise-es}, we obtain that
\begin{align*}
  |\mathbf{L}_{11}| & = \frac{1}{2} \Big|\sum_{j=1}^N a_j \int_{\TT} \int_{\TT}
  \partial_\zeta \big(|w_k(\zeta)|^2\big) R(|Z_{k,j}(\zeta,\eta)|) \dd \eta\dd \zeta \Big| \\
  & = \frac{1}{2} \Big| \sum_{j=1}^N \int_{\TT} \int_{\TT} |w_k(\zeta)|^2 R'(|Z_{k,j}|)\frac{Z_{k,j}\cdot \partial_\zeta Z_{k,j}}{|Z_{k,j}|}
  \dd \eta\dd \zeta \Big| \\
  & \leq C \big(\|z_k\|_{H^2} \|F[z_k]\|_{L^\infty}^{\frac{1}{3}} + \|z_k\|_{H^2}^{\frac{2}{3}} \big)
  \int_{\TT}\int_{\TT} |w_k(\zeta)|^2 \Big(\frac{1}{|Z_{k,k}|^{\alpha+ 2/3}} + \frac{1}{|Z_{k,k}|^{2/3}} \Big) \dd \eta \dd \zeta  \\
  & \quad + C \|\mathbf{z}\|_{L^\infty} \sum_{j\neq k} \int_{\TT} \int_{\TT} |w_k(\zeta)|^2
  \Big(\frac{1}{|Z_{k,j}|^{1+\alpha}} + \frac{1}{|Z_{k,j}|} \Big) \dd \eta \dd \zeta \\
  & \leq C \lVert w_k\rVert_{L^2}^2 \big(\|\mathbf{z}\|_{H^2} + 1 \big) \big(\|F[z_k]\|_{L^\infty}^{1+\alpha}
  + \delta[\mathbf{z}]^{-1-\alpha} + 1 \big) .
\end{align*}
For $\mathbf{L}_{12}$, we separately consider the $j=k$ case and $j\neq k$ case in the summation: by performing the change of variables and using the fact  $x\cdot \overline{y} = \overline{x} \cdot y$, we find
\begin{align*}
  \mathbf{L}_{12}|_{j=k} & = - a_k \int_{\TT} \int_{\TT} R(|Z_{k,k}|) \partial_\zeta \overline{w}_k(\zeta-\eta) \cdot w_k(\zeta) \dd \eta\dd \zeta \\
  & = - a_k \int_{\TT^2} R(|Z_{k,k}|) \partial_\zeta w_k(\zeta) \cdot \overline{w}_k(\zeta - \eta) \dd \eta \dd \zeta \\
  & = - \frac{a_k}{2} \int_{\TT^2} R(|Z_{k,k}|) \partial_\zeta \big(\overline{w}_k(\zeta-\eta) \cdot w_k(\zeta) \big) \dd \eta \dd \zeta \\
  & = \frac{a_k}{2} \int_{\TT} \int_{\TT} \overline{w}_k(\zeta-\eta) \cdot w_k(\zeta) R'(|Z_{k,k}|) \frac{Z_{k,k} \cdot
  \partial_\zeta Z_{k,k}}{|Z_{k,k}|} \dd \eta\dd \zeta \\
  & \leq C  \big(\|z_k\|_{H^2} + 1 \big) \big(\|F[z_k]\|_{L^\infty}^{1+\alpha} + 1  \big)
  \int_{\TT} \int_{\TT} \overline{w}_k(\zeta-\eta) \cdot w_k(\zeta)
  \Big(\frac{1}{|\eta|^{\alpha+2/3}} + \frac{1}{|\eta|^{2/3}} \Big) \dd \eta \dd \zeta \\
  & \leq C \lVert w_k\rVert^2_{L^2}  \big(\|z_k\|_{H^2} + 1 \big) \big(\|F[z_k]\|_{L^\infty}^{1+\alpha} + 1  \big);
\end{align*}
on the other hand,
 noting that $\partial_\zeta w_j(\zeta -\eta) = - \partial_{\eta} w_j(\zeta - \eta)$ and integrating by parts,
we deduce that
\begin{align*}
  \mathbf{L}_{12}|_{j\neq k} & = - \sum_{j\neq k} a_j \int_{\TT} \int_{\TT}
  R(|Z_{k,j}|)\partial_\zeta \overline{w}_j(\zeta-\eta)\cdot w_k(\zeta) \dd \eta \dd \zeta \\
  & = - \sum_{j\neq k} a_j \int_{\TT} \int_{\TT} w_k(\zeta) \cdot w_j(\zeta - \eta) R'(|Z_{k,j}|)
  \frac{Z_{k,j} \cdot \partial_\eta Z_{k,j}}{|Z_{k,j}|} \dd \eta \dd \zeta \\
  & \leq C \Big(\sum_{k=1}^N \lVert w_k\rVert^2_{L^2}\Big) \|\mathbf{z}\|_{H^2} \big(\delta[\mathbf{z}]^{-1-\alpha} + 1 \big) .
\end{align*}
Hence, collecting the above estimates yields
\begin{align*}
  |\mathbf{L}_{12}| \leq C \Big(\sum_{k=1}^N \lVert w_k\rVert^2_{L^2}\Big) \big(\|\mathbf{z}\|_{H^2} + 1 \big)
  \big(\|F[z_k]\|_{L^\infty}^{1+\alpha} +  \delta[\mathbf{z}]^{-1-\alpha} + 1 \big).
\end{align*}
For $\mathbf{L}_{13}$ in the $j=k$ case of the summation, using \eqref{eq:Zkk-fact}, \eqref{eq:diff-pointwise-es}
and the fact that
\begin{align}\label{eq:mean-value-theorem-R}
  R(|Z_{k,j}|) - R(|Y_{k,j}|) = \big(|Z_{k,j}|-|Y_{k,j}|\big) \int_0^1 R'\big(\tau|Z_{k,j}| + (1-\tau)|Y_{k,j}|\big) \dd \tau,
\end{align}
we have
\begin{align*}
  & \Big|\big(R(|Z_{k,k}|) - R(|Y_{k,k}|)\big) \partial_\zeta Y_{k,k}\Big| \\
  & \leq C \big(\|y_k\|_{H^2} \|F[y_k]\|_{L^\infty}^{\frac{1}{3}} + \|y_k\|_{H^2}^{\frac{2}{3}} \big)
  |Z_{k,k} - Y_{k,k}| \int_0^1 |R'(\tau|Z_{k,k}| + (1-\tau) |Y_{k,k}|)| \dd \tau \, |Y_{k,k}|^{\frac{1}{3}} \\
  & \leq C \big(\|y_k\|_{H^2} \|F[y_k]\|_{L^\infty}^{\frac{1}{3}} + \|y_k\|_{H^2}^{\frac{2}{3}} \big)
  |Z_{k,k} - Y_{k,k}| \times \\
  & \quad \times \int_0^1 \frac{1}{(1-\tau)^{\frac{1}{3}}} \Big(\frac{1}{\big(\tau |Z_{k,k}|
  + (1-\tau) |Y_{k,k}|\big)^{2/3 + \alpha}} + \frac{1}{\big(\tau |Z_{k,k}| + (1-\tau)|Y_{k,k}|\big)^{2/3}} \Big) \dd \tau \\
  & \leq C \big(\|y_k\|_{H^2}+1\big) \big(\|(F[y_k], F[z_k]) \|_{L^\infty}^{1+\alpha} +1 \big)
  |Z_{k,k} - Y_{k,k}| \Big(\frac{1}{|\eta|^{\alpha + 2/3}} + \frac{1}{|\eta|^{2/3}} \Big),
\end{align*}
and
\begin{align*}
  \mathbf{L}_{13}|_{j=k} & = a_k \int_{\TT} \int_{\TT} \Big(\big(R(|Z_{k,k}|) - R(|Y_{k,k}|)
  \big)\partial_\zeta Y_{k,k} \Big)\cdot w_k(\zeta)\dd \eta \dd \zeta \\
  & \leq C \int_{\TT}\int_{\TT} |w_k(\zeta)| |w_k(\zeta)-\overline{w}_k(\zeta-\eta)|
  \Big(\frac{1}{|\eta|^{\alpha + 2/3}} + \frac{1}{|\eta|^{2/3}} \Big)\dd \eta \dd \zeta  \\
  & \leq C \lVert w_k\rVert^2_{L^2}.
\end{align*}
For $\mathbf{L}_{13}$ in the $j\neq k$ case of the summation, using \eqref{eq:mean-value-theorem-R} and the fact that for any 
$\tau\in [0,1]$,
\begin{align*}
  \big|R'(\tau|Z_{k,j}| + (1-\tau)|Y_{k,j}|)\big| & \leq C \Big( \frac{1}{(\tau|Z_{k,j}| + (1-\tau)|Y_{k,j}|)^{1+\alpha}}
  + \frac{1}{\tau |Z_{k,j}| + (1-\tau)|Y_{k,j}|} \Big) \\
  & \leq C \big( \min(\delta[\mathbf{z}], \delta[\mathbf{y}])^{-1-\alpha} + 1 \big),
\end{align*}
it follows that
\begin{align*}
  \mathbf{L}_{13}|_{j\neq k} & = \sum_{j\neq k} a_j \int_{\TT} \int_{\TT}
  \Big(\big(R(|Z_{k,j}|) - R(|Y_{k,j}|)
  \big)\partial_\zeta Y_{k,j} \Big)\cdot w_k(\zeta)\dd \eta \dd \zeta \\
  & \leq C \lVert \mathbf{y}\rVert_{C^1(\TT)} \big( \min(\delta[\mathbf{z}],\delta[\mathbf{y}])^{-1-\alpha} + 1 \big)
  \sum_{j\neq k}\int_{\TT} \int_{\TT}  |w_k(\zeta) - \overline{w}_j(\zeta - \eta)| |w_k(\zeta)| \dd \eta\dd \zeta \\
  & \leq C  \sum_{k=1}^N \lVert w_k \rVert^2_{L^2}  .
\end{align*}
Hence, gathering the above estimates leads to
\begin{align*}
  |\mathbf{L}_1| \leq |\mathbf{L}_{11}| + |\mathbf{L}_{12}| + |\mathbf{L}_{13}| \leq C \sum_{k=1}^N \|w_k\|_{L^2}^2 .
\end{align*}
The estimation of $\mathbf{L}_2$ can be done in a similar way as that of $\mathbf{L}_1$ (indeed it is easier since $\mathbf{L}_2$ contains more cancellation), thus we omit the details. Therefore, the proof of Proposition \ref{pro:unique} is completed.
\end{proof}

Now, based on Propositions \ref{pro:ap-es}, \ref{prop:reg-param} and \ref{pro:unique}, 
we can give the proof of Theorem \ref{thm:loc-reg}.
\begin{proof}[Proof of Theorem  \ref{thm:loc-reg}]
  Relied on the \textit{a priori} control of $\lVert \mathbf{z}\rVert_W$ followed from \eqref{ineq:main-estimate} 
in Proposition \ref{pro:ap-es}, the existence of $H^2$-regular solutions to the contour dynamics equation 
\eqref {eq:main-eq-GP}-\eqref{eq:lambda-def} will be obtained by taking the limit of approximate solutions
to an appropriate family of mollified equations. We refer to \cite{KYZ17,GNP22} for the detailed process. 
Besides, exactly arguing as \cite{KYZ17}, one can prove that for $D_k(t)$ ($k=1,\cdots,N$) the interior domain governed by the contour $z_k(\cdot,t)$ 
constructed above, $\theta(\cdot,t) = \sum_{k=1}^N a_k \mathbf{1}_{D_k(t)} $ is an $H^2$ patch solution (in the sense of Definition \ref{def:patch-sol})
to the generalized SQG equation \eqref{eq:geSQG}$\&$\eqref{eq:u-exp}. This finishes the existence part.

Next, we treat the uniqueness issue. Consider any patch solution $\theta(x,t)$ given by \eqref{eq:the-patch-sol} 
to the equation \eqref{eq:geSQG}$\&$\eqref{eq:u-exp}
with $\partial D_j(t) \in C([0,T]; \mathbf{W} )$, $j=1,\cdots,N$ non self-intersecting
and $D_j(t) \cap D_k(t) = \emptyset$ for $j\neq k$. For any parameterization of the boundary of patches,
we can change of variables to deduce that $\partial D_j(t) = \{z_j(\zeta,t) \,:\, \zeta\in \mathbb{T} \}$
with $|\partial_\zeta z_j(\zeta,t)|^2 = A_j(t)$ ($j=1,\cdots,N$) depending only on time
and $\mathbf{z} =(z_1,\cdots,z_N)$ solves the contour dynamics equation \eqref{eq:main-eq-GP}-\eqref{eq:lambda-def} 
(see \cite{CCG18} for more details).
By using the change of variables $\phi_j(\cdot,t)$ in Proposition \ref{prop:reg-param}, 
it recasts $\partial D_j(t) = \{y_j(\mu,t)\,:\,\mu\in \mathbb{T}\}$ with $y_j(\mu,t)$ ($j=1,\cdots,N$)
solving the contour dynamics equation \eqref{eq:contour}. Moreover, Proposition \ref{prop:reg-param} 
guarantees that $\mathbf{y} = (y_1,\cdots, y_N) \in C([0,T];\mathbf{W})$. 
Hence, the desired uniqueness result follows from Proposition \ref{pro:unique}
concerning the uniqueness of solutions to the contour equation \eqref{eq:contour} in $C([0,T];\mathbf{W})$.
\end{proof}

\section{Finite-time singularity for the generalized SQG patches}\label{sec:singu}

In this section, we demonstrate the finite-time singularity formation for the patch solution of the generalized SQG equation 
\eqref{eq:geSQG}$\&$\eqref{eq:u-exp}, associated with patch-like initial data \eqref{eq:patch-data}. Our focus is on the half-plane case, where $\mathbf{D}=\mathbb{R}^2_+$. The kernel $G$ is assumed to satisfy \eqref{A1} to ensure the local well-posedness of the solutions.

Additionally, we introduce a second set of assumptions, \eqref{A2}, concerning the behavior of $G$ near the origin. These assumptions feature two types of Biot-Savart laws that are central to our analysis.



\begin{enumerate}[label=$(\mathbf{A}2)$,ref=$\mathbf{A}2$]
\item \label{A2} There exists a constant $c_0 >0$ (it can be chosen the same constant as in \eqref{A1} without loss of generality) 
such that \textit{either} one of the following conditions is satisfied:
\begin{enumerate}[label=$(\mathbf{A}2\alph*)$,ref=$\mathbf{A}2\alph*$]
\item\label{A2a}
$G(\rho)$ satisfies that
\begin{align}\label{cond:G1}
  \textrm{$\frac{G(\rho)}{\rho}$ is non-increasing on $(0,c_0)$,}
\end{align}
\begin{align}\label{cond:G7}
  \int_0^{c_0} \frac{1}{\rho (\log \rho^{-1}) G(\rho)} \dd \rho < +\infty,
\end{align}
and
\begin{align*}
  \textrm{$\mathcal{G}(\rho) \leq G(\rho) \leq C_1 \mathcal{G}(\rho)$  \;on $r\in (0,c_0)$},
\end{align*}
where $C_1\geq 1$ and $\mathcal{G}(\rho)$ is a positive smooth function on $(0,c_0)$ such that
\begin{align}\label{cond:G5}
  \lim_{\rho\rightarrow 0^+} \mathcal{G}(\rho) = +\infty,\quad
  \lim_{\rho\rightarrow 0^+} \frac{\mathcal{G}(l \rho)}{\mathcal{G}(\rho)} =1, \;\;\forall\; l>0,
\end{align}
\begin{align}\label{cond:G6}
  \lim_{\rho\rightarrow 0^+} \frac{\rho (\log \rho^{-1}) (-\mathcal{G}'(\rho))}{\mathcal{G}(\rho)} = \gamma,
  \quad \textrm{with}\;\;\gamma \geq 0.
\end{align}

\item\label{A2b} $G(\rho)$ satisfies that
\begin{align}\label{eq:G2a}
  \textrm{on $(0,c_0)$,\quad $G(\rho)>0$\; and\; $G(\rho)$ is non-increasing},
\end{align}
and for some $0<\beta <\frac{1}{3}$,
\begin{align}\label{cond:M0}
  \lim_{\rho\rightarrow 0^+} \frac{l^\beta G(l \rho)}{ G(\rho)} = 1, \;\;\forall\; l>0,
\end{align}
\begin{align}\label{cond:M1}
  \lim_{\rho\rightarrow 0^+} \frac{r  G'(\rho) + \beta G(\rho)}{G(\rho)} = 0.
\end{align}
\end{enumerate}
\end{enumerate}

The conditions in \eqref{A2} can be derived from the assumptions \eqref{H1}-\eqref{H2a}-\eqref{H2b}  on $m$ using Lemma \ref{lem:mD-cond}.

In particular, \eqref{A2a} addresses the borderline scenarios between the Loglog-Euler equation and the $\alpha$-SQG equation. Condition \eqref{cond:G7} is equivalent to the Osgood condition \eqref{cond:m1c}, while conditions \eqref{cond:G5} and \eqref{cond:G6} further describe the behavior of $G$ near the origin. A typical example of the function $\mathcal{G}$ is
\[\mathcal{G}(\rho)=(\log \rho^{-1})^{\gamma},\, \gamma>0.\]
There may also be additional log-log terms. From \eqref{cond:G6}, we can deduce that for any $\varepsilon>0$, there exists some $\rho_\varepsilon >0$ such that 
\begin{align}\label{coro:cond-G6}
  \mathcal{G}(\rho)\le C_{\varepsilon}(\log \rho^{-1})^{\gamma+\varepsilon},\quad \forall \rho\in (0,\rho_\varepsilon].
\end{align}
Indeed, this follows from the fact that the function $(\log \rho^{-1})^{-\gamma-\varepsilon}\mathcal{G}(\rho)$ 
is increasing on some interval $(0,\rho_{\varepsilon})$, 
which can be seen by \eqref{cond:G6},
\begin{align*}
  \frac{\dd }{\dd \rho}\Big(\frac{\mathcal{G}(\rho)}{(\log \rho^{-1})^{\gamma+\varepsilon}}\Big) 
  = \frac{(\gamma+\varepsilon)\mathcal{G}(\rho) - \rho\log \rho^{-1}(-\mathcal{G}'(\rho))}{\rho(\log \rho^{-1})^{\gamma+\varepsilon+1}}>0
\end{align*}
for $\rho>0$ small enough. 

The assumptions \eqref{A2b} contain $\alpha$-SQG-like equations. Note that the kernel $G(\rho)= \rho^{-\beta}$ satisfies conditions \eqref{cond:M0} and \eqref{cond:M1}. The singularity formation for the $\alpha$-SQG equation has been established in \cite{KRYZ,GanP21}. Here, we extend the results to general kernels $G$ that behaves like $\rho^{-\beta}$ near the origin. The extension is non-trivial due to the lack of an explicit form for $G$.

We now state the main result of this section concerning the finite-time singularity formation of the patch solution for the equation \eqref{eq:geSQG}$\&$\eqref{eq:u-exp}.
\begin{theorem}\label{thm:blow-up2}
  Let $\mathbf{D} = \mathbb{R}^2_+$. Suppose that $G(\cdot)$ is a continuously differentiable function satisfying assumptions either
\eqref{A1}-\eqref{A2a} or \eqref{A1}-\eqref{A2b}.
Then there exist non self-intersecting $H^2$ patch initial data \eqref{eq:patch-data} for the generalized SQG equation
\eqref{eq:geSQG}$\&$\eqref{eq:u-exp} so that the corresponding unique local-in-time $H^2$ patch solution develops
a singularity in finite time.
\end{theorem}


Based on Lemma \ref{lem:mD-cond}, Theorems \ref{thm:loc-reg} and \ref{thm:blow-up2}, 
one can immediately conclude the proof of Theorem \ref{thm:blow-upa} and Theorem \ref{thm:blow-upb}.
\begin{proof}[Proof of Theorem \ref{thm:blow-upa} and Theorem \ref{thm:blow-upb}]
  Under the hypotheses either \eqref{H1}-\eqref{H2a}-\eqref{cond:m1c}
or \eqref{H1}-\eqref{H2b}, according to Lemma \ref{lem:mD-cond}, 
the function $G(\cdot)$ given by \eqref{eq:G-exp1} in \eqref{eq:u-exp}
satisfies all the assumptions in Theorem \ref{thm:loc-reg} and Theorem \ref{thm:blow-up2}; 
in particular, if hypotheses \eqref{H1}-\eqref{H2a}-\eqref{cond:m1c} are assumed, 
\eqref{cond:G7} follows from \eqref{cond:m1c} and \eqref{eq:G-prop1} that (assuming $c_0\leq \min \{\bar{c}_0,\tfrac{1}{2}\}$ without loss of generality), 
\begin{align*}
  \int_0^{c_0}\frac{1}{\rho\log \rho^{-1}G(\rho)}\dd \rho 
  \leq \frac{1}{\bar{c}}\int_0^{\frac{1}{2}}\frac{1}{\rho\log \rho^{-1} m(\rho^{-1})}\dd \rho < +\infty.
\end{align*}
Hence, Theorem \ref{thm:blow-upa} and Theorem \ref{thm:blow-upb} follow as a direct consequence. 
\end{proof}

The remainder of this section is devoted to proving Theorem \ref{thm:blow-up2}. We construct an  $H^2$  patch initial data and demonstrate that it develops a singularity at the origin in finite time. The construction is primarily based on the approach in \cite{KRYZ}, but we must derive refined and implicit estimates to effectively handle the general kernels $G$.


\subsection{Set-up and mechanism for finite-time singularity}\label{subsec:equa}

In this subsection, we shall demonstrate our setting for the patch-like initial data 
and the mechanism of finite-time singularity, and we also present a more precise statement of Theorem \ref{thm:blow-up2}.


Recall that $c_0>0$ is defined in \eqref{A1} and \eqref{A2}. 
Let 
\[c_*\triangleq\frac{c_0}{4},\]
and $\epsilon$ be a small constant satisfying~$0<\epsilon \ll c_*$ to be determined later.
Denote by
\begin{align*}
  \Omega_1 \triangleq (\epsilon,c_0)\times (0,c_0) = (\epsilon,4c_*)\times (0,4c_*),
  \quad \textrm{and} \quad \Omega_2 \triangleq (2\epsilon,3c_*)\times(0,3c_*).
\end{align*}
Let $\Omega_0 $ be a domain with smooth boundary such that $\Omega_2 \subseteq \Omega_0 \subseteq \Omega_1$ 
(see Figure \ref{fig:K_0}). 
The reason that we choose the size of $\Omega_0$ small is to use the good local property of 
the kernel function $G(\rho)$ around 0. 
We let $\theta_0$ in $\overline{\mathbb{R}^2_+} = \RR \times \overline{\RR}_+$ be defined as follows,
\begin{equation}\label{def:w(0)}
  \theta_0(x) = \mathbf{1}_{\Omega_0}(x) - \mathbf{1}_{\widetilde{\Omega}_0}(x),
\end{equation}
where $\widetilde{\Omega}_0$ is the reflection of $\Omega_0$ with respect to the $x_2$-axis.

The oddness of $\theta_0$ in $x_1$-variable and the local uniqueness result in Theorem \ref{thm:loc-reg}
imply that for every $t\in [0,T_{\theta_0})$ with $T_{\theta_0}>0$ the maximal existence time in Theorem \ref{thm:loc-reg},
\begin{equation}\label{def:w(t)}
  \theta(x,t) = \mathbf{1}_{\Omega(t)}(x) - \mathbf{1}_{\widetilde{\Omega}(t)}(x),
\end{equation}
where $\Omega(t)=\Phi_t(\Omega_0)$ and $\widetilde{\Omega}(t) = \Phi_t(\widetilde{\Omega}_0)$
(which is the reflection of $\Omega(t)$ with respect to the $x_2$-axis), and $\Phi_t(\cdot)$
is the particle trajectory mapping generated by the velocity $u$ 
which satisfies \eqref{eq:flow_map}.

Below we shall show that $T_{\theta_0} < \infty$ provided that $G(\rho)$ satisfies the assumptions \eqref{A1}-\eqref{A2};
that is, the patch solution $\theta$ develops singularity in finite time.

More specifically, let
\begin{align}\label{def:X(t)}
  X'(t) = - \frac{1}{2} \mathbf{F}(X(t)),\quad X(0)=3\epsilon,
\end{align}
where
\begin{align}\label{def:F}
  \mathbf{F}(\rho) \triangleq
  \begin{cases}
    c\, \rho (\log \rho^{-1}) G(\rho),\quad & \textrm{if \eqref{A2a} is assumed}, \\
    c\, \rho\, G(\rho), \quad & \textrm{if \eqref{A2b} is assumed},
  \end{cases}
\end{align}
with some absolute constant $c>0$ (depending only on $G$).
Equivalently,
\begin{align*}
  \int_{X(t)}^{3\epsilon}\frac{2}{\mathbf{F}(\rho)}\dd \rho = t.
\end{align*}
Define
\begin{align}\label{def:T*}
  T_* \triangleq \int_0^{3\epsilon}\frac{2}{\mathbf{F}(\rho)}\dd \rho
  = \begin{cases}
    \frac{2}{c} \,\int_0^{3\epsilon} \frac{1}{\rho (\log \rho^{-1}) G(\rho)} \dd \rho,
    \quad & \textrm{if \eqref{A2a} is assumed}, \\
    \frac{2}{c}\, \int_0^{3\epsilon} \frac{1}{\rho\, G(\rho)} \dd \rho, 
    \quad & \textrm{if \eqref{A2b} is assumed}.
  \end{cases}
\end{align}
It follows from \eqref{cond:G7} and \eqref{cond:M1} that $T_*<\infty$ and $X(T_*)=0$. Indeed, if \eqref{A2b} is assumed,
denote by $ G(\rho) = \widetilde{G}(\frac{1}{\rho})$, $\rho>0$, then \eqref{cond:M1}
is equivalent to $\lim\limits_{r\rightarrow \infty} \frac{r \widetilde{G}'(r)}{\widetilde{G}(r)} 
= \beta$, thus arguing as Remark \ref{rmk:m-grow}-($\mathrm{i}$), 
we get 
\begin{align*}
  G(\rho) = \widetilde{G}(\tfrac{1}{\rho}) \geq C \rho^{-\frac{\beta}{2}}, \quad
  \textrm{for $\rho>0$ small enough}, 
\end{align*}
which directly leads to $\int_0^{3\epsilon} \frac{1}{\rho\, G(\rho)} \dd \rho <\infty$.
\par

Let $\mathrm{k} \in \mathbb{N}^\star$ be fixed later, which is the slope of the following trapezoid. 
For every $t\in [0,T_{\theta_0})$, denote by (see Figure \ref{fig:K_0})
\begin{equation}\label{def:K(t)}
  \mathbb{K}(t) \triangleq \{x\in (\mathbb{R}_+)^2\, :\, x_1\in (X(t),\tfrac{2c_{*}}{\mathrm{k}}) 
  \text{ and } x_2\in (0, \mathrm{k} x_1) \},
\end{equation}
with $(\mathbb{R}_+)^2 \triangleq \mathbb{R}_+ \times \mathbb{R}_+$.
We shall show that
\begin{align*}
  \textrm{if\; $T_{\theta_0}>T_*$,\;\;\; then\;\; $\mathbb{K}(t)\subset \Omega(t)$\; for all $t\in [0, T_*]$}.
\end{align*}
Since $u_1(x_1,x_2)=-u_1(-x_1,x_2)$ and $u_2(x_1,x_2)=u_2(-x_1,x_2)$ for any $x\in \mathbb{R}^2_+$,
this yields a contradiction because then $\Omega(T_*)$ and $\widetilde{\Omega}(T_*)$ touch at the origin,
and so the solution can not remain regular on the whole $[0,T_{\theta_0})$.
\par

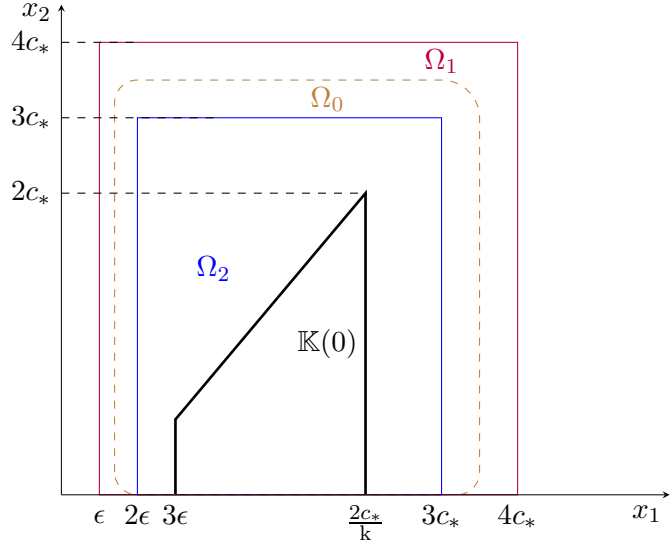
\begin{figure}[htbp]
	\begin{tikzpicture}
		
		\draw[blue]  (1,0) rectangle (5,5);
		\draw[purple]  (0.5,0) rectangle (6,6);
		\draw[line width=1pt,black] (1.5,0)--(1.5,1)--(4,4)--(4,0);

		\draw[dashed,brown] (0.7,0)[rounded corners=.3cm] -- (0.7,5.5)[rounded corners=.5cm] 
		 -- (5.5,5.5)[rounded corners=.3cm]  -- (5.5,0)
		  -- cycle;
		
		\draw[dashed] (0,5)--(2,5);
		\draw[dashed] (0,6)--(1,6);
		\draw[dashed] (0,4)--(4,4);
		
		\draw[-stealth,line width=0.2pt] (0,0) -- (8,0);
		\draw[-stealth,line width=0.2pt] (0,0) -- (0,6.5);
		
		\node[anchor=center] () at (3.5,2) {$\mathbb{K}(0)$}; 
		\node[anchor=center] () at (2,3) {\color{blue}{$\Omega_2$}};
		\node[anchor=center] () at (3.5,5.25) {\color{brown}{$\Omega_0$}};
		\node[anchor=center] () at (5,5.75) {\color{purple}{$\Omega_1$}};
		
		\node[anchor=north,below=2pt] () at (0.5,0) {$\epsilon$}; 
		\node[anchor=north] () at (1,0) {$2\epsilon$}; 
		\node[anchor=north] () at (1.5,0) {$3\epsilon$}; 
		\node[anchor=north] () at (4,0) {$\frac{2c_{\ast}}{\mathrm{k}}$}; 
		\node[anchor=north] () at (5,0) {$3c_{\ast}$}; 
		\node[anchor=north] () at (6,0) {$4c_{\ast}$}; 
		
		\node[anchor=east] () at (0,4) {$2c_{\ast}$}; 
		\node[anchor=east] () at (0,5) {$3c_{\ast}$}; 
		\node[anchor=east] () at (0,6) {$4c_{\ast}$};
		\node[anchor=north] () at (7.7,0) {$x_1$}; 
		\node[anchor=east] () at (0,6.4) {$x_2$}; 
		
	\end{tikzpicture}
	\caption{The domains $\Omega_1$, $\Omega_2$, $\Omega_0$ and $\mathbb{K}(0)$.}\label{fig:K_0}
\end{figure}

As a result, we exhibit the more precise version of Theorem \ref{thm:blow-up2}.
\begin{theorem} \label{thm:main_blowup}
Assume that $G(\rho)$ satisfies the assumptions \eqref{A1}-\eqref{A2}.
Let $\epsilon>0$ be a sufficiently small constant.  
Assume that $\theta_0(x)$ is an odd-in-$x_1$ function given by \eqref{def:w(0)},
with a bounded open domain $\Omega_0\subseteq (\mathbb{R}_+)^2$
such that $(2\epsilon,3c_*)\times(0,3c_*)\subseteq \Omega_0 \subseteq (\epsilon,4c_*)\times(0,4c_{*})$ and $\partial\Omega_0$
is a smooth simple closed curve.
Then there is no $H^2$ patch solution $\theta$ to the generalized SQG equation \eqref{eq:geSQG}$\&$\eqref{eq:u-exp}
on any interval $[0,T)$ with $T> \int_0^{3\epsilon} \frac{2}{\mathbf{F}(\rho)}\dd \rho$ and $\mathbf{F}(\rho)$ given by \eqref{def:F}.
\end{theorem}

We give the detailed proof of Theorem \ref{thm:main_blowup} in the section \ref{subsec:singularity}. 
Before that, in the section \ref{subsec:es-veloc} we show the estimates of the
velocity fields which play a crucial role in the finite-time singularity analysis.

\subsection{Estimates on the velocity fields.}\label{subsec:es-veloc}

In order to show the singularity scenario that the patch $\Omega(t)$ and its reflection across the 
$x_2$-axis touch at the origin in finite time, we need to prove that in an appropriate subset of 
$(\mathbb{R}_+)^2$, 
the horizontal velocity $u_1$ is sufficiently negative and the vertical velocity $u_2$ is sufficiently positive 
(at least for some time). This subsection is devoted to showing this result, which corresponds to Proposition \ref{prop:es-u} below.

We start with some basic pointwise estimates on the velocity field $u$.
\begin{lemma}\label{lem:u-point-es}
Let $\mathbf{D}$ be either $\mathbb{R}^2$ or $\mathbb{R}^2_+$.
Let $\theta(\cdot,t)\in L^1(\mathbf{D})\cap L^\infty(\mathbf{D})$, and $u(\cdot,t)$ be defined by \eqref{eq:u-exp} 
with $G(\rho)$ satisfying \eqref{A1}
with $\alpha\in(0,1)$. Then we have
\begin{equation}\label{es:u-Linf}
  \lVert u(t)\rVert_{L^\infty}  \leq \frac{C}{1-\alpha} \lVert \theta(t)\rVert_{L^{\infty}}
  + C \lVert \theta(t)\rVert_{L^1},
\end{equation}
and
\begin{align}\label{es:u-Hold}
  \|u(t)\|_{C^{1-\alpha}} \leq \frac{C}{\alpha(1-\alpha)} \lVert \theta(t)\rVert_{L^\infty} + C \|\theta(t)\|_{L^1},
\end{align}
where $C>0$ is a universal constant.
Furthermore, if $\theta$ is weak-$*$ continuous as an $L^\infty(\mathbf{D})$-valued function on the time interval~$[t_1,t_2]$,
and is supported on a fixed compact subset of $\overline{\mathbf{D}}$ for every $t\in [t_1,t_2]$,
then $u$ is continuous on $\overline{\mathbf{D}} \times [t_1,t_2]$.
\end{lemma}

\begin{proof}
Since the whole space case is easier, we only treat the case $\mathbf{D} = \mathbb{R}^2_+$.
Recall that for a function $g$ defined on $\mathbb{R}^2_+$, $e_o[g](\cdot)$ given by \eqref{def:ext-op} 
is the odd extension function in
$\mathbb{R}^2$. Then we have
\begin{align}\label{eq:exp-u-extend}
  u(x,t) = \int_{\mathbb{R}^2}\frac{(x-y)^{\perp}}{|x-y|^2}G(|x-y|) e_o[\theta](y,t)\dd y.
\end{align}
Using the conditions \eqref{conds:G-s}-\eqref{conds:G-l}, we infer that
\begin{align*}
  |u(x,t)| &\leq  \int_{|x-y|\leq c_0} \frac{|e_o[\theta](y,t)|}{|x-y|}|G(|x-y|)|\dd y
  +  \int_{|x-y|>c_0} \frac{|e_o[\theta](y,t)|}{|x-y|} |G(|x-y|)| \dd y \\
  &\leq  C\|e_o[\theta](\cdot,t)\|_{L^\infty} \int_{|x-y|\leq c_0}
  \frac{|G(|x-y|)|}{|x-y|} \dd y +  C\|e_o[\theta](\cdot,t)\|_{L^1} \\
  & \leq \frac{C}{1-\alpha} \|\theta(\cdot,t)\|_{L^\infty} + C \|\theta(\cdot,t)\|_{L^1}.
\end{align*}
To prove \eqref{es:u-Hold}, consider any $x,z \in \overline{\mathbb{R}^2_+}$ with $r \triangleq |x-z|$
(with no loss of generality assuming $r<c_0$),
then by virtue of the conditions \eqref{conds:G-s}-\eqref{conds:G-l}, we get
\begin{align*}
  & |u(x,t)-u(z,t)| \\
  & \leq  \int_{B(x,2r)} \frac{|G(|x-y|)|}{|x-y|} |e_o[\theta](y,t)|\,\dd y
  +  \int_{B(x,2r)} \frac{|G(|z-y|)|}{|z-y|} |e_o[\theta](y,t)|\,\dd y \\
  & \quad + \int_{\mathbb{R}^2 \setminus B(x,2r)} \left| \frac{(x-y)^\perp}{|x-y|^2}G(|x-y|)
  - \frac{(z-y)^\perp}{|z-y|^2}G(|z-y|) \right| |e_o[\theta](y,t)|\,\dd y  \\
  &\leq C\lVert e_o[\theta]\rVert_{L^\infty}\int_0^{3r}|G(s)| \dd s\\
  & \quad + C\int_{\mathbb{R}^2\setminus B(x,2r)}
  \left| \frac{(x-y)^\perp}{|x-y|^{2}}G(|x-y|) - \frac{(z-y)^\perp}{|z-y|^{2}}G(|z-y|)
  \right| |e_o[\theta] (y,t)|\,\dd y \\
  & \leq C \lVert \theta(\cdot,t)\rVert_{L^{\infty}}  \left(\int_0^{3r} |G(s)|\dd s
  + r \int_r^{c_0} \big(s^{-1}|G(s)|+|G'(s)| \big)\,\dd s \right) + C r \|\theta(\cdot,t)\|_{L^1} \\
  & \leq \Big(\frac{C}{\alpha (1-\alpha)} \|\theta(\cdot,t)\|_{L^\infty} + C \|\theta(\cdot,t)\|_{L^1}\Big) |x-z|^{1-\alpha},
\end{align*}
where the third inequality follows from the mean value theorem.
Hence, \eqref{es:u-Hold} can be deduced from this inequality and \eqref{es:u-Linf}.
\par

For the last assertion, since the kernel $\frac{x^\perp}{|x|^2} G(|x|)$ belongs to $L^1$ on any compact subset of $\overline{\mathbb{R}^2_+}$,
the assumptions yield that $u(x,t)$ is continuous in $t\in [t_1,t_2]$ for any fixed $x\in \overline{\mathbb{R}^2_+}$.
The wanted result now follows from the uniform continuity of $u$ in $x$ as shown in \eqref{es:u-Hold}.
\end{proof}

Below we drop the $t$-variable in the functions for brevity.
For every $y=(y_1,y_2)\in (\mathbb{R}_+)^2$, denote by $\widetilde{y} \triangleq (-y_1,y_2)$ and
$\overline{y} = (y_1, - y_2)$.
If $\theta(\cdot)\in L^\infty (\mathbb{R}^2_+)$ is odd in $x_1$, then from \eqref{eq:u-exp} and \eqref{def:w(0)},
we infer that $u(x) = \big(u_1(x),u_2(x)\big)$
and
\begin{align*}
  u_1(x)  & = \int_{\mathbb{R}^2_{+}} \left(\frac{x_2-y_2}{|x-y|^2}G(|x-y|)
  -\frac{x_2+y_2}{|x-\overline{y}|^2}G(|x-\overline{y}|)\right)\theta(y)\dd y \\
  & = -\int_{(\mathbb{R}_+)^2} K_1(x,y) \theta(y) \dd y,
\end{align*}
with
\begin{equation}\label{def:K1}
\begin{split}
  K_1(x,y) & =
  \frac{y_2-x_2}{|x-y|^{2}}G(|x-y|) -
  \frac{y_2-x_2}{|x-\widetilde y|^{2}}G(|x-\widetilde{y}|) -
  \frac{y_2+x_2}{|x+y|^{2}}G(|x+y|) +
  \frac{y_2+x_2}{|x-\overline y|^{2}}G(|x-\overline{y}|) \\
  & \triangleq K_{11}(x,y) - K_{12}(x,y) - K_{13}(x,y) + K_{14}(x,y),
\end{split}
\end{equation}
and
\begin{align*}
  u_2(x) & = \int_{\mathbb{R}^2_{+}} \left(\frac{y_1 -x_1}{|x-y|^2}G(|x-y|)
  -\frac{y_1 - x_1}{|x-\overline{y}|^2}G(|x-\overline{y}|)\right)\theta(y)\dd y \\
  & = \int_{(\mathbb{R}_+)^2} K_2(x,y) \theta(y)\dd y,
\end{align*}
with
\begin{equation}\label{def:K2}
\begin{split}
  K_2(x,y) & =
  \frac{y_1-x_1}{|x-y|^{2}}G(|x-y|) +
  \frac{y_1+x_1}{|x-\widetilde y|^{2}}G(|x-\widetilde{y}|) -
  \frac{y_1+x_1}{|x+y|^{2}}G(|x+y|) -
  \frac{y_1-x_1}{|x-\overline y|^{2}}G(|x-\overline{y}|) \\
  & \triangleq K_{21}(x,y) + K_{22}(x,y) - K_{23}(x,y) - K_{24}(x,y).
\end{split}
\end{equation}

Similar as \cite[Lemma 4.2]{KRYZ}, we state some useful properties about the kernel functions $K_1$ and $K_2$.
\begin{lemma}\label{lem:lmm_K}
Assume that $G(\rho)$ satisfies that
\begin{align}\label{cond:Gadd}
  \textrm{on $(0,c_0)$,\quad $G(\rho)>0$\; and\; $\frac{G(\rho)}{\rho}$ is non-increasing}.
\end{align}
For every $x,y\in (\mathbb{R}_+)^2$ such that $|x+y|\leq c_0$, the following statements hold true.
\begin{enumerate}
\item[$\mathrm{(i)}$] $K_1(x,y) \geq K_{11}(x,y) - K_{12}(x,y)$. 	
\item[$\mathrm{(ii)}$] ${\rm sgn}(y_2-x_2) \big(K_{11}(x,y) - K_{12}(x,y)\big) \geq 0$.
\item[$\mathrm{(iii)}$] $K_2(x,y) \geq K_{21}(x,y) - K_{24}(x,y)$.
\item[$\mathrm{(iv)}$] ${\rm sgn}(y_1-x_1) \big(K_{21}(x,y) - K_{24}(x,y) \big) \geq 0$.
\end{enumerate}
\end{lemma}

\begin{proof}
Due to \eqref{cond:Gadd}, 
$\frac{G(\rho)}{\rho^2}$ is non-increasing on $(0,c_0)$. 
Thus (i) follows directly from $|x-\overline{y}|\leq |x+y|$ and (ii) from $|x-y| \leq |x-\widetilde y|$.
The proofs of (iii) and (iv) are ensured by exchanging $\overline{y}$ and $\widetilde{y}$.
\end{proof}

In view of Lemma \ref{lem:lmm_K}-(ii), we shall separately estimate the "bad" part and "good" part of the integral
\begin{equation}\label{def:u1g-u1b}
\begin{split}
  u_1(x) & = - \int_{\mathbb{R}_+\times(0,x_2)} K_1(x,y) \theta(y) \dd y
  - \int_{\mathbb{R}_+\times(x_2,\infty)} K_1(x,y) \theta(y) \dd y \\
  & \triangleq u_1^{\textrm{bad}}(x)  + u_1^{\textrm{good}}(x).
\end{split}
\end{equation}
Analogously, we also have the splitting for $u_2$:
\begin{equation}\label{def:u2g-u2b}
\begin{split}
  u_2(x) & = \int_{(0,x_1)\times \mathbb{R}_+} K_2(x,y) \theta(y) \dd y
  + \int_{(x_1,\infty)\times \mathbb{R}_+} K_2(x,y) \theta(y) \dd y \\
  & \triangleq u_2^{\textrm{bad}}(x) + u_2^{\textrm{good}}(x).
\end{split}
\end{equation}

The following result is concerned with the estimation for "bad" parts of $u_1$ and $u_2$.
\begin{lemma}\label{lem:bound_u}
Let the condition \eqref{cond:Gadd} be satisfied. Assume that $\theta$ is odd in $x_1$, and $0\leq \theta \leq 1$ on $(\mathbb{R}_+)^2$
and $\theta(y)\equiv 0$ for $y\in (\mathbb{R}_+)^2 \setminus \big((0,c_0/2)\times (0,c_0/2)\big)$.
The following statements hold.
\begin{enumerate}
\item[$\mathrm{(i)}$] If $x\in \overline{(\mathbb{R}_+)^2}$ and $x_1,x_2\leq c_{\ast}$,
\begin{align*}
  u_1^{\mathrm{bad}}(x) \leq 2\int_0^{x_1}\int_0^{x_2} \frac{s_2}{s_1^2 + s_2^2} G\Big(\sqrt{s_1^2+s_2^2}\Big)\dd s_2\dd s_1.
\end{align*}
\item[$\mathrm{(ii)}$] If $x\in \overline{(\mathbb{R}_+)^2}$ and $x_1,x_2 \leq c_{\ast}$, then
\begin{align*}
  u_2^{\mathrm{bad}}(x) \geq -2\int_0^{x_2}\int_0^{x_1} \frac{s_1}{s_1^2 + s_2^2} G\Big(\sqrt{s_1^2+s_2^2}\Big) \dd s_1\dd s_2.
\end{align*}
\end{enumerate}
\end{lemma}

\begin{proof}
\textbf{(i)} Due to that $|x+y|\leq |x|+|y|\leq c_0$ for every $x,y\in\mathrm{supp}\,\theta\subset (0,c_0/2)\times (0,c_0/2)$,
it follows from Lemma \ref{lem:lmm_K} that for every $x_1,x_2\leq c_{\ast}=\frac{c_0}{4}$,
\begin{align*}
  u_1^{\textrm{bad}}(x)
  &\leq  -\int_{\mathbb{R}_+\times(0,x_2)} \left(\frac{y_2-x_2}{|x-y|^{2}}G(|x-y|)
  - \frac{y_2-x_2}{|x-\tilde y|^{2}}G(|x-\widetilde{y}|)\right) \theta(y) \dd y \\
  &\leq  -\int_{(0,c_0/2)\times(0,x_2)} \left(\frac{y_2-x_2}{|x-y|^{2}}G(|x-y|)
  - \frac{y_2-x_2}{|x-\widetilde{y}|^{2}}G(|x-\widetilde{y}|)\right) \dd y \\
  &\leq \int_{(0,2x_1)\times(0,x_2)} \frac{x_2-y_2}{|x-y|^2}G(|x-y|)  \dd y,
\end{align*}
where in the third inequality we have used the following identity
\begin{align*}
  \int_{(0,c_0/2)\times(0,x_2)} \frac{y_2-x_2}{|x-\widetilde y|^{2}}G(|x-\widetilde{y}|) \dd y
  = \int_{(2x_1,c_0/2+2x_1)\times(0,x_2)} \frac{y_2-x_2}{|x-y|^{2}}G(|x-y|)\dd y.
\end{align*}
Now, using the change of variables~$s = x-y$ and by symmetry, 
we find
\begin{align*}
  u_1^{\textrm{bad}}(x) \leq 2\int_0^{x_1}\int_0^{x_2} \frac{s_2}{s_1^2 + s_2^2} G\Big(\sqrt{s_1^2+s_2^2}\Big) \dd s_2\dd s_1.
\end{align*}
\textbf{(ii)} The proof of part (ii) is analogous to that of (i), and we omit the details.
\end{proof}

In the estimation of the "good" parts of $u_1$ and $u_2$, we shall additionally assume that
for some $x\in (\mathbb{R}_+)^2$ we have $\theta \equiv1$ 
on the following triangle:
\begin{equation}\label{eqdef:A}
  \mathbb{A}(x) \triangleq \Big\{ y=(y_1,y_2)\,: \,  y_1\in \left(x_1,x_1+\tfrac{c_{*}}{\mathrm{k}} \right),
  y_2 \in \big(x_2, x_2 + \mathrm{k} (y_1 - x_1)\big) \Big\},
\end{equation}
where $\mathrm{k}\in\mathbb{N}^\star$ is a positive integer to be fixed later.

\begin{lemma}\label{lem:es-good}
Let the condition \eqref{cond:Gadd} be satisfied. 
Assume that $\theta$ is odd in $x_1$, $0\leq \theta \leq 1$ on $(\mathbb{R}_+)^2$,
$\theta(y)\equiv 0$ for every $y\in (\mathbb{R}_+)^2 \setminus \big((0,c_0/2)\times (0,c_0/2)\big)$
and for some $x\in \overline{(\mathbb{R}_+)^2} $
we have $\theta\geq \mathbf{1}_{\mathbb{A}(x)}$ on $(\mathbb{R}_+)^2$,
with $\mathbb{A}(x)$ given by \eqref{eqdef:A}.
Then the following statements hold true.
\begin{enumerate}
\item[$\mathrm{(i)}$] If $x_1\leq \frac{c_*}{4 \mathrm{k}}$, $x_2\leq c_*$,  
we have
\begin{equation}\label{es:u1good}
\begin{aligned}
  u_1^{\mathrm{good}}(x) \leq & \, 2 \mathrm{k} G\Big(\frac{c_*}{\mathrm{k}}\Big) x_1
  - \frac{2x_1}{\frac{4}{\mathrm{k}^2}+1}  \int_{2\mathrm{k} x_1}^{c_{\ast}}  
  \frac{G\big(\sqrt{\frac{4}{\mathrm{k}^2}+1} \,s\big)}{s}  \dd s \\
  &-\int_0^{2x_1}\dd s_1\int_0^{\mathrm{k} s_1} \frac{s_2}{s_1^2 + s_2^2} 
  G\Big(\sqrt{s_1^2+s_2^2}\Big)\dd s_2\\
  &-\int_{2x_1}^{4x_1}\dd s_1\int_{\mathrm{k}(s_1-2 x_1)}^{2\mathrm{k} x_1} 
  \frac{s_2}{s_1^2+s_2^2} G\Big(\sqrt{s_1^2+s_2^2}\Big)\dd s_2.
\end{aligned}
\end{equation}
\item[$\mathrm{(ii)}$] If $x_2\leq \frac{c_*}{4\mathrm{k}^2}$, $x_1\leq c_*$, 
we have
\begin{equation}\label{es:u2good-b}
	\begin{aligned}
		u_2^{\mathrm{good}}(x)\geq \int_{\mathrm{k} x_2}^{\frac{c_\ast}{\mathrm{k}}} \int_0^{2x_2}
        \frac{s_1}{s_1^2+s_2^2}\bigg( G\Big(\sqrt{s_1^2+s_2^2}\Big) 
		- \frac{1}{\mathrm{k}^2+1} G\Big(\sqrt{\mathrm{k}^2+1} \sqrt{s_1^2+s_2^2}\Big)\bigg) \dd s_2\dd s_1.
	\end{aligned}
\end{equation}
\end{enumerate}
\end{lemma}

\begin{proof}
\textbf{(i)} Using Lemma \ref{lem:lmm_K} and the changing of variables $y_1\mapsto y_1+2x_1$, we obtain
\begin{align*}
  u_1^{\textrm{good}}(x)&\leq - \int_{\mathbb{A}(x)}
  \left(\frac{y_2-x_2}{|x-y|^{2}}G(|x-y|) -  \frac{y_2-x_2}{|x-\widetilde y|^{2}}G(|x-\widetilde{y}|)\right) \dd y \\
  &= -\int_{\mathbb{A}(x)} \frac{y_2-x_2}{|x-y|^{2}}G(|x-y|) \dd y +
  \int_{\mathbb{A}(x)+2x_1e_1} \frac{y_2-x_2}{|x-y|^{2}}G(|x-y|) \dd y,
\end{align*}
with $e_1 \triangleq (1,0)$.
After some cancellations, it directly leads to
\begin{align*}
  u_1^{\textrm{good}}(x) & \leq - \int_{\mathbb{A}_1}\frac{y_2-x_2}{|x-y|^{2}} G(|x-y|) \dd y
  + \int_{\mathbb{A}_2}\frac{y_2-x_2}{|x-y|^{2}}G(|x-y|) \dd y  \\
  & \triangleq \mathcal{T}_1 + \mathcal{T}_2,
\end{align*}
where
\begin{equation*}
\begin{split}
  \mathbb{A}_1 &\triangleq \big\{ y=(y_1,y_2) \,:\,  
  y_2 \in \left(x_2, x_2 + c_* \right), 
  y_1\in \left(x_1+\tfrac{y_2-x_2}{\mathrm{k}}, 3x_1+\tfrac{y_2-x_2}{\mathrm{k}} \right)  
  \big\},\\
  \mathbb{A}_2 & \triangleq \left(x_1+\tfrac{c_*}{\mathrm{k}}, 
  3 x_1+\tfrac{c_*}{\mathrm{k}}\right) \times 
  \left(x_2, x_2 + c_*\right).
\end{split}
\end{equation*}
Since for every $y\in \mathbb{A}_2$ we have $y_2-x_2\leq c_{*} 
\leq \mathrm{k}|x-y| \leq \mathrm{k} c_0$, we use \eqref{cond:Gadd} to infer that
\begin{align*}
   \mathcal{T}_2 \leq  |\mathbb{A}_2| \tfrac{\mathrm{k}}{c_*} 
   G(\tfrac{c_*}{\mathrm{k}}) = 2 \mathrm{k} G(\tfrac{c_*}{\mathrm{k}}) x_1.
\end{align*}
For $\mathcal{T}_1$, inspired by \cite{GanP21}, we split it into three parts
\begin{align*}
	\mathcal{T}_1=\mathcal{T}_{11}+\mathcal{T}_{12}+\mathcal{T}_{13},
\end{align*}
where $\mathcal{T}_{1i}$, $i=1,2,3$ is given by 
\begin{align*}
  \mathcal{T}_{1i}=-\int_{\mathbb{A}_{1i}} 
  \frac{y_2-x_2}{|x-y|^2} G(|x-y|) \dd y,
\end{align*}
with
\begin{align*}
  \mathbb{A}_{11}\triangleq \mathbb{A}_1\cap \big((x_1,3x_1)\times \mathbb{R}_+ \big),\quad
 \mathbb{A}_{12}\triangleq \mathbb{A}_1\cap \big(\mathbb{R}_+\times (x_2 + 2 \mathrm{k} x_1,+\infty)\big),
\end{align*}
and 
\begin{align*}
  \mathbb{A}_{13}\triangleq \big\{(y_1,y_2)\, : \, y_1\in (3x_1,5x_1),\; 
  y_2\in (x_2+ \mathrm{k} (y_1-3x_1), x_2 + 2 \mathrm{k} x_1)\big\}.
\end{align*}
One can see Figure \ref{fig:A} for the illustration of these domains.
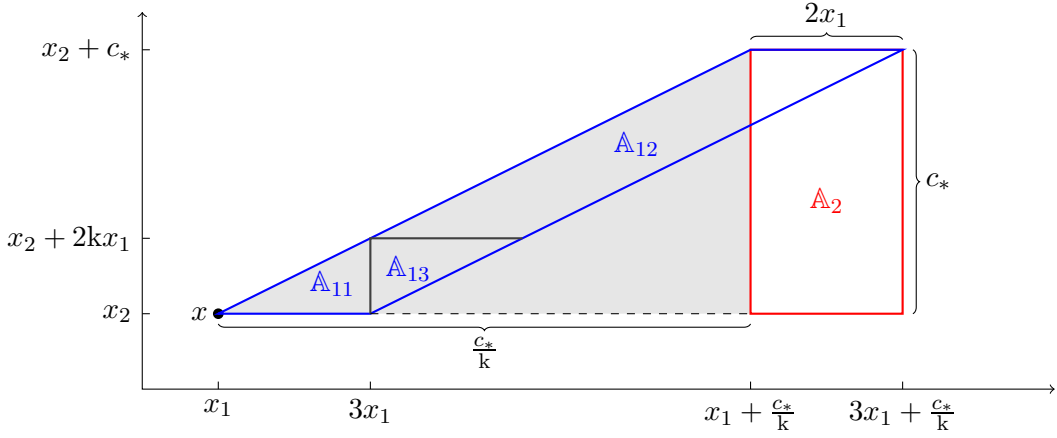
\begin{figure}[htbp] 
	\begin{tikzpicture}
        \fill[color=gray!20] (1,1) -- (8,1) -- (8,4.5);
        \draw[<->] (0,5) -- (0,0) -- (12,0);
		\node[left] () at (1,1) {$x$}; \fill (1,1) circle[radius=2pt];
	    \draw[red,thick] (8,1) rectangle (10,4.5);
		\draw[blue,thick] (1,1)--(3,1)--(10,4.5)--(8,4.5)--cycle;
		\draw[darkgray,thick] (3,1)--(3,2)--(5,2);
        \draw[dashed] (3,1)--(8,1);
  
		\node[anchor=center] () at (2.5,1.4) {\color{blue}{$\mathbb{A}_{11}$}}; 
		\node[anchor=center] () at (3.5,1.6) {\color{blue}{$\mathbb{A}_{13}$}}; 
		\node[anchor=center] () at (6.5,3.25) {\color{blue}{$\mathbb{A}_{12}$}}; 
		\node[anchor=center] () at (9,2.5) {\color{red}{$\mathbb{A}_{2}$}}; 
		
		\draw[decorate,decoration={brace,raise=1pt}] (8,0.9)--(1,0.9) node[below=2pt,midway] {$\tfrac{c_\ast}{\mathrm{k}}$};
		
		\draw[decorate,decoration={brace,raise=1pt}] (10.1,4.5)--(10.1,1) node[right=2pt,midway] {$c_{\ast}$};
		
		\draw[decorate,decoration={brace,raise=1pt}] (8,4.6)--(10,4.6) node[above=2pt,midway] {$2x_{1}$};

        \draw (1,0.1) -- (1,0) node[below]{$x_1$}; 
        \draw (3,0.1) -- (3,0) node[below]{$3x_1$};
        \draw (8,0.1) -- (8,0) node[below]{$x_1+\tfrac{c_\ast}{\mathrm{k}}$}; 
        \draw (10,0.1) -- (10,0) node[below]{$3x_1+\tfrac{c_\ast}{\mathrm{k}}$}; 
        \draw (0.1,1) -- (0,1) node[left]{$x_2$}; 
        \draw (0.1,2) -- (0,2) node[left]{$x_2+2{\mathrm{k}}x_1$}; 
        \draw (0.1,4.5) -- (0,4.5) node[left]{$x_2+c_{\ast}$};

	\end{tikzpicture}
\caption{A demonstration of the domains {\color{blue}$\mathbb{A}_1 = \mathbb{A}_{11}\cup \mathbb{A}_{12}\cup \mathbb{A}_{13}$} and {\color{red}$\mathbb{A}_2$}. The gray region represents $\mathbb{A}(x)$.}\label{fig:A}
\end{figure}

Via the change of variables, we can write $\mathcal{T}_{11}$ as 
\begin{align*}
  \mathcal{T}_{11}=
  &-\int_{x_1}^{3x_1}\dd y_1 \int_{x_2}^{x_2+ \mathrm{k}(y_1-x_1)} \frac{(y_2-x_2)}{|x-y|^2} G(|x-y|)\dd y_2 \\
  = & -\int_0^{2x_1}\dd s_1\int_0^{\mathrm{k} s_1} \frac{s_2}{s_1^2+s_2^2} G\Big(\sqrt{s_1^2+s_2^2}\Big) \dd s_2.
\end{align*}
Now we move to the estimate of $\mathcal{T}_{12}$. Since $(y_1,y_2)\in \mathbb{A}_{12}$ implies that 
$y_2-x_2\geq 2 \mathrm{k} x_1$ and $y_2-x_2\geq \mathrm{k}(y_1-3x_1) = \mathrm{k}(y_1-x_1)-2 \mathrm{k} x_1$, 
we see that $2(y_2-x_2)\geq \mathrm{k}(y_1-x_1)$, 
which leads to 
\begin{align*}
  |x-y|^2 \leq \big(\tfrac{4}{\mathrm{k}^2} +1\big) (y_2-x_2)^2.
\end{align*}
Thus by using the condition \eqref{cond:Gadd} we find
\begin{equation*}
\begin{split}
  \mathcal{T}_{12} &
  = - \int_{x_2+2 \mathrm{k} x_1}^{x_2+c_\ast} 
  \int_{x_1+\frac{y_2-x_2}{\mathrm{k}}}^{{3x_1+\frac{y_2-x_2}{\mathrm{k}}}} 
  \frac{y_2-x_2}{|x-y|^2} G(|x-y|) \dd y_1 \dd y_2 \\
  & \leq - \frac{2x_1}{\frac{4}{\mathrm{k}^2} + 1}  
  \int_{x_2 + 2 \mathrm{k} x_1}^{x_2+c_\ast}  
  \frac{G\Big(\sqrt{\frac{4}{\mathrm{k}^2}+1}(y_2-x_2) \Big)}{y_2-x_2}  \dd y_2 \\
  & = - \frac{2x_1}{\frac{4}{\mathrm{k}^2}+1}  \int_{2\mathrm{k} x_1}^{c_\ast}  
  \frac{G\Big(\sqrt{\frac{4}{\mathrm{k}^2}+1}\, s\Big)}{s}  \dd s.
\end{split}
\end{equation*}
For the term $\mathcal{T}_{13}$, we directly have
\begin{align*}
  \mathcal{T}_{13}  = &-\int_{3x_1}^{5x_1}\dd y_1\int_{x_2+\mathrm{k}(y_1-3x_1)}^{x_2+2 \mathrm{k} x_1} 
  \frac{y_2-x_2}{|x-y|^2}G(|x-y|)\dd y_2 \\
  = & - \int_{2x_1}^{4x_1}\dd s_1\int_{\mathrm{k}(s_1-2x_1)}^{2 \mathrm{k} x_1} 
  \frac{s_2}{s_1^2+s_2^2} G\Big(\sqrt{s_1^2+s_2^2}\Big) \dd s_2.
\end{align*}
Collecting the estimates of $\mathcal{T}_1$ and $\mathcal{T}_2$ yields \eqref{es:u1good}, as desired.
\vskip1mm

\textbf{(ii)} Using Lemma \ref{lem:lmm_K} and the change of variables $y_2\mapsto y_2+2x_2$, we deduce that
\begin{equation*}
\begin{split}
  u_2^{\textrm{good}}(x) &\geq  \int_{\mathbb{A}(x)}
  \left(\frac{y_1-x_1}{|x-y|^{2}}G(|x-y|) - \frac{y_1-x_1}{|x-\overline{y}|^{2}}G(|x-\overline{y}|)\right) \dd y \\
  &= \int_{\mathbb{A}(x)} \frac{y_1-x_1}{|x-y|^{2}}G(|x-y|) \dd y -
  \int_{\mathbb{A}(x) + 2x_2 e_2} \frac{y_1-x_1}{|x-y|^2}G(|x-y|) \dd y,
\end{split}
\end{equation*}
with $e_2 \triangleq (0,1)$.
After some cancellation it follows that
\begin{align*}
   u_2^{\textrm{good}}(x) \geq \int_{\mathbb{B}_1} \frac{y_1-x_1}{|x-y|^{2}}G(|x-y|)  \dd y
   - \int_{\mathbb{B}_2} \frac{y_1-x_1}{|x-y|^{2}}G(|x-y|)  \dd y,
\end{align*}
where
\begin{equation*}
\begin{split}
  \mathbb{B}_1 &\triangleq \left(x_1, x_1 + \tfrac{c_{*}}{\mathrm{k}} \right) \times \left(x_2, 3x_2\right), \\
  \mathbb{B}_2 &\triangleq \big\{ y=(y_1,y_2): y_1 \in \left(x_1, x_1 + \tfrac{c_*}{\mathrm{k}} \right),
  y_2\in \left(x_2+\mathrm{k}(y_1-x_1), 3x_2+\mathrm{k}(y_1-x_1) \right) \big\}.
\end{split}
\end{equation*}
One can see  Figure \ref{fig:B} for the illustration of $\mathbb{B}_1$ 
and $\mathbb{B}_2$.
\begin{figure}[htbp]
	\begin{tikzpicture}
         \fill[color=gray!20] (1.5,.8) -- (8.5,.8) -- (8.5,4.3);

		\draw[<->] (0,6.5) -- (0,0) -- (10,0);
		\node[below, xshift=-5pt] () at (1.5,.8) {$x$}; \fill (1.5,.8) circle[radius=2pt];
			
		\draw[red, thick] (1.5,.8)--(1.5,2.4)--(8.5,5.9)--(8.5,4.3)--cycle;
		\draw[blue, thick] (1.5,.8) rectangle (8.5,2.4);
        \draw[dashed] (8.5,4.3)--(8.5,2.4);
  
		\node[anchor=center] () at (5,1.6) {\color{blue}{$\mathbb{B}_{1}$}}; 
		\node[anchor=center] () at (5,3.35) {\color{red}{$\mathbb{B}_{2}$}}; 
  
		\draw[decorate,decoration={brace,raise=1pt}] (8.5,0.7)--(1.5,0.7) node[below=2pt,midway] {$\tfrac{c_{\ast}}{\mathrm{k}}$};
		\draw[decorate,decoration={brace,raise=1pt}] (8.6,4.3)--(8.6,0.8) node[right=2pt,midway]  {$c_{\ast}$};
		\draw[decorate,decoration={brace,raise=1pt}] (1.4,0.8)--(1.4,2.4) node[left=2pt,midway]  {$2x_{2}$};

        \draw (1.5,0.1) -- (1.5,0) node[below]{$x_1$}; 
        \draw (8.5,0.1) -- (8.5,0) node[below]{$x_1+\tfrac{c_\ast}{\mathrm{k}}$}; 
        \draw (0.1,.8) -- (0,.8) node[left]{$x_2$}; 
        \draw (0.1,2.4) -- (0,2.4) node[left]{$3x_2$}; 
        \draw (0.1,4.3) -- (0,4.3) node[left]{$x_2+c_{\ast}$};         
        \draw (0.1,5.9) -- (0,5.9) node[left]{$3x_2+c_{\ast}$};         

	\end{tikzpicture}
	\caption{A demonstration of the domains {\color{blue}$\mathbb{B}_{1}$} and {\color{red}$\mathbb{B}_2$}.  The gray region represents $\mathbb{A}(x)$.}\label{fig:B}
\end{figure}
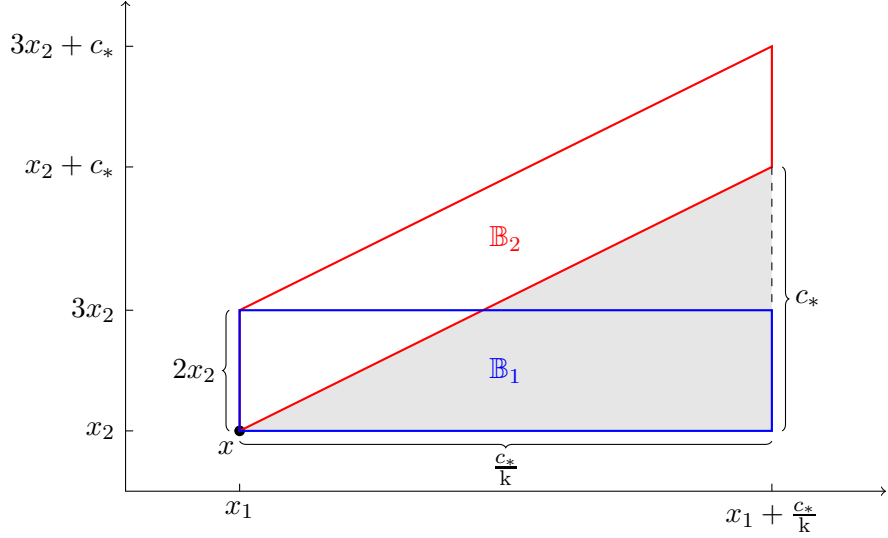

The change of variables $y_2\mapsto y_2-(y_1-x_1)$ in the above second integral gives
\begin{align*}
  &u_2^{\textrm{good}}(x) \geq\\
  &\int_{\mathbb{B}_1} \left( \frac{y_1-x_1}{|x-y|^{2}}G(|x-y|)
  - \frac{y_1-x_1}{|x-(y_1,y_2+\mathrm{k}(y_1-x_1))|^{2}}
  G\big(|x-(y_1,y_2+\mathrm{k} (y_1-x_1))|\big) \right) \dd y.
\end{align*}
Due to that the integrand is positive, and noting that for every $y\in (x_1+\mathrm{k}x_2, x_1 + \tfrac{c_*}{\mathrm{k}} ) \times (x_2, 3x_2)$,
$2(y_1-x_1)> 2\mathrm{k}x_2>\mathrm{k}(y_2-x_2)>0$ and also
\begin{align*}
  \big|x-\big(y_1,y_2+\mathrm{k}(y_1-x_1)\big)\big|^2 
  & = (\mathrm{k}^2+1)|x-y|^2+\mathrm{k}(y_2-x_2)
  \big[ 2(y_1-x_1)-\mathrm{k}(y_2-x_2)\big] \\
  & > (\mathrm{k}^2+1)|x-y|^2,
\end{align*}
we use the condition \eqref{cond:Gadd} to get
\begin{align*}
  u_2^{\textrm{good}}(x) & \geq \int_{(x_1+kx_2, x_1 + \frac{c_\ast}{\mathrm{k}} ) \times (x_2, 3x_2)}
  \frac{y_1-x_1}{|x-y|^2} \left(G(|x-y|)-\frac{1}{\mathrm{k}^2+1}
  G\big(\sqrt{\mathrm{k}^2+1}|x-y|\big) \right) \dd y \\
  &= \int_{(\mathrm{k}x_2, \frac{c_\ast}{\mathrm{k}} ) \times (0, 2x_2)}
  \frac{s_1}{s_1^{2}+s_2^2}\left( G\Big(\sqrt{s_1^2+s_2^2}\Big)-\frac{1}{\mathrm{k}^2+1}
  G\Big(\sqrt{\mathrm{k}^2+1}\sqrt{s_1^2+s_2^2}\Big) \right) \dd s_1 \dd s_2.
\end{align*}
\end{proof}

Based on Lemmas \ref{lem:bound_u} and \ref{lem:es-good}, we have the following crucial result 
on the control of the velocity field.
\begin{proposition}\label{prop:es-u}
Let the assumptions \eqref{A1}-\eqref{A2} be assumed.
Assume that $\theta$ is odd in $x_1$ and for some 
$x\in  \overline{(\mathbb{R}_+)^2} $ 
we have $\mathbf{1}_{\mathbb{A}(x)} \leq \theta \leq 1$ on 
$(\mathbb{R}_+)^2$,
and $\theta(y)=0$ if $y\in (\mathbb{R}_+)^2 \setminus (0,c_0/2)\times (0,c_0/2)$.
Then there exist a positive integer $\mathrm{k}$ and a small constant $\delta_G\in (0,c_*)$ 
such that for every $x_2\leq \mathrm{k}x_1\leq \delta_G$,
\begin{align}\label{eq:u1-bd}
  u_1(x)\leq - \mathbf{F}(x_1),
\end{align}
and for every $ \mathrm{k} x_1 \leq x_2 \leq \delta_G$, 
\begin{align}\label{eq:u2-bd}
  u_2(x) \geq \mathbf{F}(\tfrac{x_2}{\mathrm{k}}),
\end{align}
where $\mathbf{F}(\rho)$ is given by \eqref{def:F}.
\end{proposition}

\begin{proof}
We first consider the case that $G(\rho)$ satisfies assumptions
\eqref{A1}-\eqref{A2a}. 
We can simply choose $\mathrm{k}=1$. 
Gathering Lemmas \ref{lem:bound_u} and \ref{lem:es-good} gives that for every $x_2\leq x_1\leq \frac{c_*}{4}$,
\begin{equation}\label{es:u1-upbd}
  u_1(x) \leq 2\int_{0}^{x_1}\int_0^{x_1}\frac{s_2}{s_1^2 + s_2^2}G(\sqrt{s_1^2+s_2^2})\dd s_2\dd s_1
  +2 G(c_{*}) x_1-\frac{2x_1}{5} \int_{2x_1}^{c_{*}}  \frac{G(\sqrt{5}s)}{s} \dd s.
\end{equation}
For the term $u_2^{\textrm{good}}$, it follows from \eqref{es:u2good-b} and \eqref{cond:Gadd} that 
\begin{align*}
  u_2^{\textrm{good}} \geq& 
  \Big(1- \frac{1}{\sqrt{\mathrm{k}^2+1}}\Big)\int_{(\mathrm{k}x_2, \frac{c_\ast}{\mathrm{k}} ) \times (0, 2x_2)}
  \frac{s_1}{s_1^{2}+s_2^2} G\Big(\sqrt{s_1^2+s_2^2}\Big)\dd s_1\dd s_2\\
  \geq& \frac{2\big(1- \frac{1}{\sqrt{\mathrm{k}^2+1}}\big)x_2}{\frac{4}{\mathrm{k}^2}+1}
  \int_{\mathrm{k}x_2}^{\frac{c_*}{\mathrm{k}}} \frac{G\big(\sqrt{\tfrac{4}{\mathrm{k}^2}+1}s\big)}{s} \dd s
\end{align*}
where we also use $\sqrt{\frac{4}{\mathrm{k}^2}+1}s_1\geq 
\sqrt{s_1^2+s_2^2}$ in the last inequality. 
Hence, choosing $\mathrm{k}=1$, together with Lemma \ref{lem:bound_u}, for every $x_1\leq x_2\leq \frac{c_*}{4}$,
\begin{equation}\label{es:u2-lobd}
  u_2(x) \geq -2\int_0^{x_2}\int_0^{x_2} \frac{s_1}{s_1^2 + s_2^2}G(\sqrt{s_1^2+s_2^2})\dd s_1\dd s_2 +
  \frac{2\big(1-\frac{1}{\sqrt{2}}\big)x_2}{5} \int_{x_2}^{c_*} 
  \frac{G(\sqrt{5}s)}{s} \dd s .
\end{equation}

Now we begin with the estimate of \eqref{es:u1-upbd}. Direct calculation yields
\begin{align}
  \int_0^{x_1} \int_0^{x_1} \frac{2s_2G(\sqrt{s_1^2 + s_2^2})}{s_1^2 + s_2^2} \dd s_2 \dd s_1 \nonumber 
  &\leq \int_0^{\sqrt{2}x_1} \int_0^{\frac{\pi}{2}} \frac{2G(\varrho)\sin \eta}{\varrho} \varrho\,
  \dd\eta \dd \varrho \nonumber \\
  & =\int_0^{\sqrt{2}x_1}2G(\varrho)\dd \varrho 
  \leq 2C_1\int_0^{\sqrt{2}x_1} \mathcal{G}(\varrho)\dd \varrho \nonumber\\
  & \approx 2\sqrt{2}C_1x_1 \mathcal{G}(x_1),\quad \textrm{as}\;\; x_1\rightarrow 0^+, \nonumber
\end{align}
where the last line is guaranteed by the following fact (owing to L'Hospital's rule, \eqref{cond:G5}-\eqref{cond:G6} and \eqref{coro:cond-G6})
\begin{align*}
  \lim_{\rho\to 0^+}\frac{\int_0^\rho \mathcal{G}(\varrho) \dd \varrho}{\rho \mathcal{G}(\rho)}
  = \lim_{\rho\to 0^+} \frac{ \mathcal{G}(\rho)}{\mathcal{G}(\rho) + \rho \mathcal{G}'(\rho)} = 1.
\end{align*}
On the other hand, observing that (using \eqref{cond:G5}-\eqref{cond:G6} again)
\begin{align*}
  \lim_{\rho\rightarrow 0^+} \frac{\int_{2\rho}^{c_*} s^{-1}\mathcal{G}(\sqrt{5}s) \dd s}{(\log \rho^{-1}) \mathcal{G}(\rho)}
  = \lim_{\rho\rightarrow 0^+} \frac{ \rho^{-1} \mathcal{G}(2\sqrt{5}\rho) }{\rho^{-1} \mathcal{G}(\rho) - (\log r^{-1}) \mathcal{G}'(r)}
  = \frac{1}{1+\gamma},
\end{align*}
we have
\begin{equation}\label{eq:Gintg-sim}
\begin{split}
  -\frac{2x_1}{5} \int_{2x_1}^{c_{*}}  \frac{G(\sqrt{5}s)}{s} \dd s
  & \leq -\frac{2C_1x_1}{5} \int_{2x_1}^{c_{*}}  \frac{\mathcal{G}(\sqrt{5}s)}{s} \dd s \\
  & \approx - \frac{2C_1}{5(1+\gamma)} x_1 (\log x_1^{-1}) \mathcal{G}(x_1),
  \quad \textrm{as}\;\; x_1\rightarrow 0^+ .
\end{split}
\end{equation}
Hence, by letting $\delta_G\in (0,c_*)$ small enough, the contribution from the integral term in \eqref{eq:Gintg-sim}
dominates and we can find a small positive constant $c$ such that
\begin{align*}
  \mathbf{F}(\rho) = c \,\rho (\log \rho^{-1}) G(\rho),\quad \textrm{if (\textbf{A}2a) is assumed},
\end{align*}
so that \eqref{eq:u1-bd} holds true. Similarly, in light of \eqref{es:u2-lobd}, by putting $\delta_G\in (0,c_*)$ and $c>0$ 
even smaller if necessary,  the estimate \eqref{eq:u2-bd} also holds with the above $\mathbf{F}(\rho)$.
\vskip1.5mm

Next we consider the case that $G(\rho)$ satisfies assumptions
\eqref{A1}-\eqref{A2b}.
Taking advantage of Lemma \ref{lem:bound_u} and Lemma \ref{lem:es-good}, we deduce that 
\begin{align*}
  u_1(x) & \leq  2 \mathrm{k} G\Big(\frac{c_*}{\mathrm{k}}\Big) x_1
  + 2\int_{0}^{x_1}\int_0^{x_2}\frac{s_2}{s_1^2 + s_2^2}G(\sqrt{s_1^2+s_2^2})\dd s_2\dd s_1
  - \frac{2x_1}{\frac{4}{\mathrm{k}^2}+1}  \int_{2\mathrm{k} x_1}^{c_\ast}  
  \frac{G\big(\sqrt{\frac{4}{\mathrm{k}^2}+1}\,s\big)}{s}  \dd s \\
  &-\int_0^{2x_1}\dd s_1 \int_0^{\mathrm{k} s_1} \frac{s_2}{s_1^2+s_2^2} 
  G\Big(\sqrt{s_1^2+s_2^2}\Big)\dd s_2
  - \int_{2x_1}^{4x_1}\dd s_1\int_{\mathrm{k}(s_1-2x_1)}^{2\mathrm{k}x_1} 
  \frac{s_2}{s_1^2+s_2^2} G\Big(\sqrt{s_1^2+s_2^2}\Big)\dd s_2 \\
  & \triangleq  2 \mathrm{k} G\Big(\frac{c_*}{\mathrm{k}}\Big) x_1 
  + U_1 + U_2 + U_3 + U_4.
\end{align*}

Denote by $G_1(\rho)\triangleq \rho^\beta G(\rho)$. Notice that the assumptions \eqref{cond:M0}-\eqref{cond:M1}
in \eqref{A2b} imply that
\begin{align}\label{eq:G_1-limit}
  \lim_{r\rightarrow 0^+} \frac{G_1(l r)}{G_1(r)} =1,\;\;\forall l > 0,
  \quad\textrm{and}\quad \lim_{r\rightarrow 0^+} \frac{r G_1'(r)}{G_1(r)} =0.
\end{align}
First, for $x_2\leq \mathrm{k}x_1$, we have
\begin{align*}
  U_1&=2\int_{0}^{x_1}\int_0^{x_2}\frac{s_2}{s_1^2 + s_2^2}G(\sqrt{s_1^2+s_2^2})\dd s_2\dd s_1= \int_0^{x_1} \int_{s_1^2}^{s_1^2+ x_2^2} 
  \frac{G_1(\sqrt{s})}{s^{1+ \beta/2}} \dd s\dd s_1\\
  &\leq \int_0^{x_1} \int_{s_1^2}^{(\mathrm{k}^2+1)x_1^2} 
  \frac{G_1(\sqrt{s})}{s^{1+\beta/2}} \dd s\dd s_1
  \triangleq \mathbf{H}_1(x_1).
\end{align*}
Let us find an equivalent explicit form of $\mathbf{H}_1(x_1)$ as $x_1\to 0^+$.
Note that
\begin{align*}
  \mathbf{H}_1'(x_1) &  = \int_{x_1^2}^{(\mathrm{k}^2+1)x_1^2}
  \frac{G_1(\sqrt{s})}{s^{1+\beta/2}} \dd s
  + \int_0^{x_1} \frac{2(\mathrm{k}^2+1)x_1}{(\mathrm{k}^2+1)^{1+\beta/2}} 
  \frac{G_1\big(\sqrt{\mathrm{k}^2+1}\,x_1\big)}{x_1^{2+\beta}} \dd z_1 \\
  & = \int_{x_1^2}^{(\mathrm{k}^2+1) x_1^2} \frac{G_1(\sqrt{s})}{s^{1+\beta/2}}\dd s
  + \frac{2}{(\mathrm{k}^2+1)^{\frac{\beta}{2}}}
  x_1^{-\beta} G_1\big(\sqrt{\mathrm{k}^2+1} x_1\big),
\end{align*}
and
\begin{align*}
  \lim_{x_1\rightarrow 0^+} \frac{\int_{x_1^2}^{(\mathrm{k}^2+1)x_1^2} 
  s^{-1-\frac{\beta}{2}} G_1(\sqrt{s}) \dd s}{x_1^{-\beta}G_1(x_1)}
  &= \lim_{x_1\rightarrow 0^+} \frac{2(\mathrm{k}^2+1)^{-\frac{\beta}{2}} 
  x_1^{-1-\beta}  G_1(\sqrt{\mathrm{k}^2+1}x_1) -
  2 x_1^{-1 -\beta} G_1(x_1)}{x_1^{-\beta -1} G_1(x_1) 
  \big(-\beta + x_1 G_1'(x_1)/ G_1(x_1)\big) } \\
  & = \frac{2}{\beta} \Big( 1- (\mathrm{k}^2+1)^{-\frac{\beta}{2}}\Big),
\end{align*}
we obtain that
\begin{align*}
  \lim_{x_1\rightarrow 0^+} \frac{\mathbf{H}_1(x_1)}{x_1^{1-\beta} G_1(x_1)}
  & = \lim_{x_1\rightarrow 0^+} \frac{\mathbf{H}'_1(x_1)}{x_1^{-\beta} G_1(x_1) 
  \big(1-\beta +  x_1 G_1'(x_1)/G_1(x_1) \big) } \\
  & = \frac{2}{\beta(1-\beta)}\big( 1- (\mathrm{k}^2+1)^{-\frac{\beta}{2}}\big) 
  + \frac{2(\mathrm{k}^2+1)^{-\beta/2}}{1-\beta} \\
  & = \frac{2}{\beta} \Big(\frac{1}{1-\beta}-(\mathrm{k}^2+1)^{-\frac{\beta}{2}}\Big).
\end{align*}
Hence, we get that 
\begin{align*}
  U_1 \leq \bigg(\frac{2}{\beta} \Big(\frac{1}{1-\beta}-(\mathrm{k}^2+1)^{-\frac{\beta}{2}}\Big) 
  + o(x_1) \bigg) x_1^{1-\beta} G_1(x_1),
  \quad \textrm{with}\;\;\lim_{x_1\rightarrow 0^+} o(x_1)=0.
\end{align*}
For the term $U_2$, we first consider the integral
\begin{equation*}
  \mathbf{H}_2(x_1)\triangleq\int_{2\mathrm{k} x_1}^{c_\ast}  
  \frac{G_1\big(s\sqrt{\frac{4}{\mathrm{k}^2}+1}\big)}{s^{1+\beta}}  \dd s. 
\end{equation*}
From the fact that 
\begin{align*}
  \mathbf{H}'_2(x_1)=-(2k)^{-\beta}x_1^{-1-\beta}
  G_1\Big(2\mathrm{k}x_1\sqrt{\tfrac{4}{\mathrm{k}^2}+1}\Big),
\end{align*}
it follows that 
\begin{align*}
  \lim_{x_1\to 0^+}\frac{\mathbf{H}_2(x_1)}{x_1^{-\beta}G_1(x_1)}
  =  \lim_{x_1\to 0^+} \frac{\mathbf{H}_2'(x_1)}{x_1^{-\beta-1}G_1(x_1)
  \big(-\beta+x_1\frac{G_1'(x_1)}{G_1(x_1)}\big)}  
  =  \frac{(2\mathrm{k})^{-\beta}}{\beta}.
\end{align*}
Hence we deduce that 
\begin{align*}
  U_2 = - \frac{2x_1}{(\frac{4}{\mathrm{k}^2}+1)^{1+ \beta/2}}
  \mathbf{H}_2(x_1)
  \leq \Big( -\frac{2^{1-\beta}\mathrm{k}^{-\beta}}{\beta(\frac{4}{k^2}+1)^{1+ \beta/2}}+o(x_1)\Big) 
  x_1^{1-\beta}G_1(x_1).
\end{align*}
For the term $U_3$, we have 
\begin{align*}
  U_3 = U_3(x_1) = -\int_0^{2x_1}\dd s_1\int_0^{\mathrm{k}s_1} \frac{s_2}{(s_1^2 + s_2^2)^{1+\beta/2}}
  G_1\Big(\sqrt{s_1^2+s_2^2}\Big) \dd s_2.
\end{align*}
It is easy to check that 
\begin{align*}
  U'_3(x_1) = -2 \int_0^{2\mathrm{k}x_1} 
  \frac{s_2 G_1\big(\sqrt{4x_1^2+s^2_2}\big)}{(4x_1^2+s_2^2)^{1+ \beta/2}} \dd s_2. 
\end{align*}
We claim that for every $s_2 \in [0,2\mathrm{k}x_1]$, 
\begin{align}\label{eq:G1-fact}
  G_1\Big(\sqrt{4 x_1^2 + s_2^2}\Big) = \big( 1+ o(x_1) \big) G_1(x_1), \quad \textrm{with}
  \;\; \lim_{x_1\rightarrow 0^+} o(x_1) = 0.
\end{align}
Indeed, this is a direct consequence of Newon-Lebniz's formula and \eqref{eq:G_1-limit}: 
\begin{align*}
  G_1\Big(\sqrt{4 x_1^2 + s_2^2}\Big) - G_1(x_1) 
  = \int_{x_1}^{\sqrt{4x_1^2 + s_2^2}} G_1'(s) \dd s,
\end{align*}
and for every $\varepsilon >0$ there exists a $\delta_0>0$ such that $|s\, G_1'(s) | \leq \varepsilon G_1(s)$ 
for every $x_1\in(0,\delta_0)$ and $s\in [x_1, \sqrt{4 x_1^2 + s_2^2}]\subset [x_1,2\sqrt{\mathrm{k}^2 +1}x_1]$, 
thus we deduce the desired result
\begin{align*}
  \Big| \int_{x_1}^{\sqrt{4x_1^2 + s_2^2}} G_1'(s) \dd s \Big| 
  & \leq \varepsilon \int_{x_1}^{\sqrt{4 x_1^2 + s_2^2}}  \frac{1}{s} G_1(s) \dd s \\
  & \leq \varepsilon 2\sqrt{\mathrm{k}^2 +1} 
  \Big(\sup_{s\in [x_1, 2\sqrt{\mathrm{k}^2 +1} x_1]} G_1(s)\Big) \\
  & \leq \varepsilon\,4 \sqrt{\mathrm{k}^2 +1}\,G_1(x_1),
\end{align*}
where in the last line we have used the fact that the supremum of $G_1(s)$ on compact set
$[x_1,2\sqrt{\mathrm{k}^2 +1}\,x_1]$ can be achieved on a point $l x_1$ with some 
$1\leq l\leq 2\sqrt{\mathrm{k}^2 +1}$ and also $G_1(l x_1) \approx G_1(x)$ for every $x_1\in (0,\delta_0)$.

Hence,
\begin{align*}
  U'_3(x_1) & = -2\big(1+o(x_1)\big) G_1(x_1) \int_0^{2\mathrm{k}x_1} 
  \frac{s_2}{(4x_1^2+s_2^2)^{1+ \beta /2}} \dd s_2\\
  & = -\frac{2^{1-\beta}}{\beta} \Big(1-(\mathrm{k}^2+1)^{-\frac{\beta}{2}}\Big) 
  \big(1+o(x_1)\big) x_1^{-\beta} G_1(x_1),
\end{align*}
where $\lim\limits_{x_1\rightarrow 0^+} o(x_1) =0$. 
We thus infer that 
\begin{align*}
  \lim_{x_1\to 0^+} \frac{U_3(x_1)}{x_1^{1-\beta} G_1(x_1)} 
  = & \lim_{x_1\to 0^+} \frac{U_3'(x_1)}{x_1^{-\beta} G_1(x_1) 
  \big(1-\beta+\frac{G_1'(x_1)}{G_1(x_1)}\big)}\\
  = & - \frac{2^{1-\beta}}{\beta(1-\beta)} \big(1-(\mathrm{k}^2+1)^{-\frac{\beta}{2}}\big),
\end{align*}
and as a result,
\begin{align*}
  U_3 \leq \Big(-\frac{2^{1-\beta}}{\beta(1-\beta)} \big(1-(\mathrm{k}^2+1)^{-\frac{\beta}{2}} \big) + o(x_1) \Big)
  x_1^{1-\beta} G_1(x_1), \quad \textrm{with}
  \;\; \lim_{x_1\rightarrow 0^+} o(x_1) = 0.
\end{align*}
For the term $U_4$, arguing as obtaining \eqref{eq:G1-fact} (the main difference lies on that 
$G_1(\sqrt{(s_1+2x_1)^2 + s_2^2})$ attains its supremum over 
$\{0\leq s_1\leq 2 x_1, \mathrm{k}s_1 \leq s_2 \leq 2\mathrm{k} x_1\}$ 
at the point $(l_1 x_1, \mathrm{k} l_2 x_1)$ with some $0\leq l_1, l_2\leq 2$), we find that 
\begin{align*}
  U_4 = & -\int_{2x_1}^{4x_1}\dd s_1\int_{\mathrm{k}(s_1-2x_1)}^{2\mathrm{k}x_1} 
 \frac{s_2}{(s_1^2+s_2^2)^{1+\beta/2}} G_1\Big(\sqrt{s_1^2+s_2^2}\Big) \dd s_2 \\
  = & -\big(1+o(x_1)\big) G_1(x_1) \int_{2x_1}^{4x_1}\dd s_1 \int_{\mathrm{k}(s_1-2x_1)}^{2\mathrm{k}x_1} 
  \frac{s_2}{(s_1^2+s_2^2)^{1+ \beta/2}} \dd s_2,
\end{align*}
with $\lim\limits_{x_1\rightarrow 0^+} o(x_1) =0$.
In addition, since the integrand is positive, we find that 
\begin{align*}
  - \int_{2x_1}^{4x_1}\dd s_1\int_{\mathrm{k}(s_1-2x_1)}^{2\mathrm{k}x_1} 
  \frac{s_2}{(s_1^2+s_2^2)^{1+\frac{\beta}{2}}}\dd s_2 
  \le & \frac{1}{\beta}\int_{2x_1}^{3x_1} 
  \Big( \big(s_1^2+(2\mathrm{k}x_1)^2\big)^{-\frac{\beta}{2}} 
  - \big(s_1^2+\mathrm{k}^2(s_1-2x_1)^2\big)^{-\frac{\beta}{2}}\Big)\dd s_1 \\
  \leq & \frac{1}{\beta} \int_{2x_1}^{3x_1} 
  \Big( \big(4x_1^2+4\mathrm{k}^2 x_1^2\big)^{-\frac{\beta}{2}}
  - \big(9x_1^2+\mathrm{k}^2x_1^2\big)^{-\frac{\beta}{2}} \Big) \dd s_1 \\
  = & \frac{1}{\beta} \Big((4+4\mathrm{k}^2)^{-\frac{\beta}{2}}-(9+\mathrm{k}^2)^{-\frac{\beta}{2}}\Big) 
  x_1^{1-\beta},
\end{align*}
which gives that 
\begin{align*}
  U_4\leq \big(1+o(x_1)\big) \frac{1}{\beta} 
  \Big((4+ 4\mathrm{k}^2)^{-\frac{\beta}{2}}
  -(9+\mathrm{k}^2)^{-\frac{\beta}{2}} \Big) 
  x_1^{1-\beta} G_1(x_1).
\end{align*}
Therefore one has that
\begin{align*}
  u_1(x) \leq \big(\Pi_1(\beta)+o(x_1)\big) x_1 G(x_1) + 2\mathrm{k} G\big(\tfrac{c_\ast}{\mathrm{k}}\big) x_1 
  \leq \big( \Pi_1(\beta) + o(x_1) \big) x_1 G_1(x_1),
\end{align*}
where $\lim\limits_{x_1\to 0^{+}}o(x_1)=0$ and 
\begin{equation}
\begin{aligned}
  \Pi_1(\beta) & \triangleq \frac{2}{\beta}\bigg(\frac{1}{1-\beta}-(\mathrm{k}^2+1)^{-\frac{\beta}{2}}
  -\frac{2^{-\beta}}{\mathrm{k}^\beta(\frac{4}{\mathrm{k}^2}+1)^{1+\beta/2}}  -\frac{2^{-\beta}}{(1-\beta)}
  \big(1-(\mathrm{k}^2+1)^{-\frac{\beta}{2}}\big) \\
  & \quad \qquad + 2^{-1} \Big((4+4\mathrm{k}^2)^{-\frac{\beta}{2}} -(9+\mathrm{k}^2)^{-\frac{\beta}{2}} \Big)\bigg).
\end{aligned}
\end{equation}

In the sequel, we focus on the estimation of $u_2(x)$ for every $\mathrm{k}x_1\leq x_2$ small. 
It follows from Lemmas \ref{lem:bound_u} and \ref{lem:es-good} that 
\begin{align*}
  u_2(x) & \geq -2\int_0^{x_2}\int_0^{x_1} \frac{s_1}{s_1^2 + s_2^2}G(\sqrt{s_1^2+s_2^2})\dd s_1\dd s_2  \\
  & \quad + \int_{\mathrm{k}x_2}^{\frac{c_\ast}{\mathrm{k}}}
  \int_0^{2x_2} \frac{s_1}{s_1^2+s_2^2} \bigg( G\Big(\sqrt{s_1^2+s_2^2}\Big) 
  -\frac{1}{\mathrm{k}^2+1} G\Big(\sqrt{\mathrm{k}^2+1}\sqrt{s_1^2+s_2^2}\Big) \bigg) \dd s_2\dd s_1 \\
  & \triangleq V_1 + V_2.
\end{align*}
For the term $V_1$, owing to $\mathrm{k} x_1 \leq x_2$ we have 
\begin{align*}
   V_1=-\int_0^{x_2}\int_{s_2^2}^{s_2^2+x_1^2} \frac{G(\sqrt{s})}{s}\dd s\dd s_2\geq -\int_0^{x_2}\int_{s_2^2}^{(1+\frac{1}{\mathrm{k}^2})x_2^2}
  \frac{G(\sqrt{s})}{s}\dd s\dd s_2 \triangleq \mathbf{H}_3(x_2).
\end{align*}
Note that
\begin{align*}
  \mathbf{H}'_3(x_2) = -\int_{x_2^2}^{(1+\frac{1}{\mathrm{k}^2})x_2^2}
  \frac{G_1(\sqrt{s})}{s^{1+\frac{\beta}{2}}} \dd s - 2 \Big(1+\frac{1}{\mathrm{k}^2}\Big)^{-\frac{\beta}{2}}
  x_2^{-\beta} G_1\Big(\sqrt{1+\tfrac{1}{\mathrm{k}^2}}x_2\Big),
\end{align*}
and 
\begin{align*}
  \lim_{x_2\to 0^+} \frac{\int_{x_2^2}^{(1+\frac{1}{\mathrm{k}^2})x_2^2}
  \frac{G_1(\sqrt{s})}{s^{1+ \beta/2}} \dd s}{x_2^{-\beta}G_1(x_2)}
  = & \lim_{x_2\to 0^+}\frac{2x_2^{-1-\beta} \Big( \big(1+\frac{1}{\mathrm{k}^2}\big)^{-\frac{\beta}{2}}
  G_1\Big(\sqrt{1+\frac{1}{\mathrm{k}^2}}x_2\Big) - G_1(x_2)\Big)}{x_2^{-1-\beta}
  G_1(x_2) \big(-\beta + x_2\frac{G'_1(x_2)}{G_1(x_2)}\big)} \\
  = & \frac{2}{\beta} \Big(1- \Big(1+\frac{1}{\mathrm{k}^2}\Big)^{-\frac{\beta}{2}}\Big).
\end{align*}
Thus we infer that 
\begin{align*}
  \lim_{x_2\to 0^+}\frac{\mathbf{H}_3(x_2)}{x_2^{1-\beta}G_1(x_2)} 
  & = \lim_{x_2\to 0^+}\frac{\mathbf{H}'_3(x_2)}{(1-\beta)x_2^{-\beta}G_1(x_2)} \\
  & = -\frac{2}{\beta(1-\beta)} \Big(1-\Big(1+\frac{1}{\mathrm{k}^2}\Big)^{-\frac{\beta}{2}}\Big)
  - \frac{2}{1-\beta} \Big(1+\frac{1}{\mathrm{k}^2}\Big)^{-\frac{\beta}{2}}\\
  & = \frac{2}{\beta}\Big(-\frac{1}{1-\beta}+ \Big(1+\frac{1}{\mathrm{k}^2}\Big)^{-\frac{\beta}{2}}\Big),
\end{align*}
which implies that 
\begin{align*}
  V_1\geq \frac{2}{\beta}\Big(-\frac{1}{1-\beta}+ \Big(1+\frac{1}{\mathrm{k}^2}\Big)^{-\frac{\beta}{2}}+o(x_2)\Big)
  x_2^{1-\beta}G_1(x_2),\quad \textrm{with\;  $\lim_{x_2\to 0^+}o(x_2)=0$}.
\end{align*}
It remains to consider the contribution from $V_2$.
Let $N_\mathrm{k} > \mathrm{k}$ be a fixed constant depending only on $\mathrm{k}$ and $\beta$
such that $N_\mathrm{k} x_2\leq \tfrac{c_\ast}{\mathrm{k}}$. 
Note that for fixed $\mathrm{k}\in \mathbb{N}^\star$, we can choose $N_\mathrm{k}$ 
sufficiently large by restricting $x_2>0$ small enough. 
By using \eqref{eq:G2a} and the notation  $G_1(\rho) = \rho^\beta G(\rho)$, 
the integrand in $V_2$ is positive and 
\begin{align*}
  V_2  & \geq  \int_{\mathrm{k} x_2}^{N_{\mathrm{k}}x_2} 
  \int_0^{2x_2} \frac{s_1}{(s_1^2+s_2^2)^{1+\beta/2}} \bigg(G_1\Big(\sqrt{s_1^2+s_2^2}\Big) -
  \frac{1}{(\mathrm{k}^2+1)^{1+\beta/2}}  G_1\Big(\sqrt{\mathrm{k}^2+1} \sqrt{s_1^2+s_2^2}\Big) 
  \bigg) \dd s_2\dd s_1 .
\end{align*}
In a similar way as deriving \eqref{eq:G1-fact}, we find that for every $s_1\in [\mathrm{k} x_2, N_{\mathrm{k}} x_2 ]$
and $s_2\in [0,2x_2]$, 
\begin{align*}
  G_1\Big(\sqrt{s_1^2 + s_2^2}\Big) = \big(1 + o(x_2)\big) G_1(x_2),\quad
  G_1\Big(\sqrt{\mathrm{k}^2 +1} \sqrt{s_1^2 + s_2^2}\Big) = \big(1 + o(x_2)\big) G_1(x_2),
\end{align*}
with $\lim\limits_{x_2\rightarrow 0^+} o(x_2) = 0$, thus
\begin{align*}
  V_2 & \geq G_1(x_2) \big(1+ o(x_2) \big)\Big( 1 - \frac{1}{(\mathrm{k}^2+1)^{1+\beta/2}}\Big)
  \int_{\mathrm{k}x_2}^{N_\mathrm{k} x_2} \int_0^{2x_2} 
  \frac{s_1}{(s_1^2+s_2^2)^{1+\beta/2}}  \dd s_2 \dd s_1  .
\end{align*} 
In addition, we see that 
\begin{align*}
  \int_{\mathrm{k}x_2}^{N_\mathrm{k} x_2} \int_0^{2x_2}
  \frac{s_1}{(s_1^2+s_2^2)^{1+\frac{\beta}{2}}} \dd s_2\dd s_1 
  = \frac{1}{\beta} \int_0^{2x_2} \Big(\big(\mathrm{k}^2x_2^2+s_2^2\big)^{-\frac{\beta}{2}}
  - \big(N^2_\mathrm{k}x^2_2 + s_2^2\big)^{-\frac{\beta}{2}} \Big) \dd s_2,
\end{align*}
and 
\begin{align*}
  \int_0^{2x_2}(\mathrm{k}^2 x_2^2 + s_2^2)^{-\frac{\beta}{2}} \dd s_2 
  & \geq \int_0^{x_2} \frac{1}{(\mathrm{k}^2+1)^{\frac{\beta}{2}} 
  x_2^\beta} \dd s_2 + \int_{x_2}^{2x_2} (\mathrm{k}^2+1)^{-\frac{\beta}{2}} 
  s_2^{-\beta}\dd s_2 \\
  & \geq \frac{1}{(\mathrm{k}^2+1)^{\beta/2}} 
  \Big(1+\frac{2^{1-\beta}-1}{1-\beta} \Big) x_2^{1-\beta}.
\end{align*}
It follows that 
\begin{align*}
  \int_{\mathrm{k}x_2}^{N_\mathrm{k} x_2} \int_0^{2x_2}
  \frac{s_1}{(s_1^2+s_2^2)^{1+\frac{\beta}{2}}} \dd s_2\dd s_1 
  \geq \frac{1}{\beta} \bigg(\frac{1}{(\mathrm{k}^2+1)^{\beta/2}} 
  \Big(1+\frac{2^{1-\beta}-1}{1-\beta}\Big) -\frac{2}{N_\mathrm{k}^\beta}\bigg) x_2^{1-\beta}.
\end{align*}
Consequently, 
\begin{align*}
  V_2 \geq \frac{1}{\beta} \Big(1-(k^2+1)^{-1-\frac{\beta}{2}}\Big) 
  \bigg(\frac{1}{(\mathrm{k}^2+1)^{\beta/2}}
  \Big(1+\frac{2^{1-\beta}-1}{1-\beta}\Big) - \frac{2}{N_\mathrm{k}^\beta}+o(x_2)  \bigg) x_2^{1-\beta} G_1(x_2).
\end{align*}

Therefore, we infer that for $x_2$ small enough,
\begin{align*}
  u_2(x) & \geq  \Big(\Pi_2(\beta) -\frac{2}{\beta \,N_\mathrm{k}^\beta} 
  + o(x_2) \Big) x_2^{1-\beta} G_1(x_2) \\
  & \geq \frac{\mathrm{k}^{1-\beta}}{2}  \Big(\Pi_2(\beta) -\frac{2}{\beta \,N_\mathrm{k}^\beta} 
  + o(x_2) \Big) \Big(\frac{x_2}{\mathrm{k}}\Big)^{1-\beta} G_1(\tfrac{x_2}{\mathrm{k}}),
  \quad \textrm{with}\;\;\lim_{x_2\rightarrow 0^+} o(x_2)=0,
\end{align*}
with
\begin{align*}
  \Pi_2(\beta) \triangleq \frac{2}{\beta} \bigg( -\frac{1}{1-\beta} + 
  \Big(1+\frac{1}{\mathrm{k}^2}\Big)^{-\frac{\beta}{2}} 
  + \frac{1-(\mathrm{k}^2+1)^{-1-\frac{\beta}{2}}}{2(\mathrm{k}^2+1)^{\beta/2}}
  \Big(1+\frac{2^{1-\beta}-1}{1-\beta}\Big)\bigg).
\end{align*}

Noticing that $\Pi_1(\beta)$ and $\Pi_2(\beta)$ are the same indexes appearing in
\cite{GanP21}, and by picking $\mathrm{k}=5$,
for every $0<\beta <\frac{1}{3}$ (i.e. the same range as in the local well-posedness part),
we find
\begin{align*}
  \textrm{$\Pi_1(\beta)<0$\;\; and \;\; $\Pi_2(\beta) >0$}.
\end{align*}
Moreover, we can choose $x_1$ and $x_2$ small enough and $N_\mathrm{k}>\mathrm{k}$
sufficiently large so that $\Pi_1(\beta) + o(x_1) \leq -c<0$ and
$ \frac{\mathrm{k}^{1-\beta}}{2}\big(\Pi_2(\beta) - \frac{2}{\beta N_\mathrm{k}^\beta} + o(x_2) \big) \geq c >0 $.
Consequently, we have
\begin{align*}
  u_1(x) & \leq - c\, x_1 G(x_1),\qquad\,  \forall \; x_2 \leq\mathrm{k} x_1 \leq \delta_G , \\
  u_2(x) &\geq c \tfrac{x_2}{\mathrm{k}} G(\tfrac{x_2}{\mathrm{k}}) , \qquad \quad  \forall \; \mathrm{k} x_1 \leq x_2 \leq \delta_G,
\end{align*}
which corresponds to \eqref{eq:u1-bd}-\eqref{eq:u2-bd} under the assumptions \eqref{A1}-\eqref{A2b}, as desired.
\end{proof}

\subsection{The finite-time singularity analysis}\label{subsec:singularity}
This subsection is devoted to the proof of Theorem \ref{thm:main_blowup}.
\begin{proof}[Proof of Theorem \ref{thm:main_blowup}]
The proof is via a contradiction argument introduced in \cite{KRYZ}.

We turn to the setting in Section \ref{subsec:equa}.
The initial condition we consider is odd in $x_1$,
then according to Theorem \ref{thm:loc-reg}, the resulting unique $H^2$ patch solution is also odd.
Assume that such a local $H^2$ patch solution exists on $[0,T_{\theta_0})$ 
with $T_{\theta_0}>0$ the maximal existence time, 
and let $ X(t)$, $\mathbb{K}(t)$ be given by
\eqref{def:X(t)}, \eqref{def:K(t)}, respectively.
Recall that $T_* \triangleq \int_0^{3\epsilon} \frac{2}{\mathbf{F}(\rho)}\dd \rho<\infty$ and $X(T_*)=0$.
In view of Theorem \ref{thm:loc-reg}, the solution has the patch form \eqref{def:w(t)},
and we shall show that
$\mathbb{K}(t)\subseteq \Omega(t)$ for each $t\in[0,T_*]$. This is a contradiction because then the patches
$\Omega(T_*)$ and $\widetilde{\Omega}(T_*)$ touch at $0$.
\par

As $|\Omega(t)|=|\Omega_0|\leq 16c_{*}^2$, Lemma \ref{lem:u-point-es} implies that for all $t\in[0,T_*]$,
\begin{equation} \label{bound:u}
  \|u(\cdot,t)\|_{L^{\infty}} \leq \overline{C},
\end{equation}
where $\overline{C}$ depends only on $G$.
Since $\partial\Omega(t)$ is continuous in $t\in[0,T_*]$ with respect to the Hausdorff distance of sets,
Lemma \ref{lem:u-point-es} also shows that $u$ is continuous on $\overline{\mathbb{R}^2_+}\times[0,T_*]$.\par

Consider $\delta_G\in(0,\frac{c_*}{\mathrm{k}})$ and $\mathrm{k}\in\mathbb{N}^\star$ 
the constants introduced in Proposition \ref{prop:es-u} and let the constant $\epsilon>0$ be small enough so that
\begin{align*}
  \epsilon <\frac{\delta_G}{4 \mathrm{k}} \quad\text{and}\quad 
  \int_0^{3\epsilon} \frac{2}{\mathbf{F}(\rho)} \dd \rho \leq \frac{\delta_G}{2\overline{C}}.
\end{align*}
From Lemma \ref{lem:u-point-es} we know that the function
$\mathtt{d}(t)\triangleq {\rm dist}\big((\mathbb{R}_+)^2 
\setminus\overline{\Omega(t)},\mathbb{K}(t)\big)$ is continuous on $[0,T_*]$.
Hence, if $\mathbb{K}(t)$ is not contained in $\Omega(t)$ at some $t\in[0,T_*]$,
then there is the first time $t_0\in[0,T_*]$ such that $\mathtt{d}(t_0)=0$.
Since $\mathtt{d}(0)\geq\epsilon>0$, we have $t_0>0$ and $\mathbb{K}(t_0)\subseteq \Omega(t_0)$.
\par

Now we assume that such a $t_0$ exists and set
\begin{align*}
  \Omega_3 \triangleq \big(\delta_G,\tfrac{5c_*}{2})\times(0,\tfrac{5c_*}{2}\big).
\end{align*}
Note that in view of Proposition \ref{prop:flow-map} (the assumptions are satisfied due to 
that $H^2(\mathbb{T})\subset C^{1,\frac{1}{2}}(\mathbb{T})$ and there exists some 
$\sigma\in (\frac{\alpha}{1-\alpha},\frac{1}{2}]$ for every 
$0<\alpha<\frac{1}{3}$), we have $(\mathbb{R}_+)^2 \setminus \overline{\Omega(t)}
= \Phi_t((\mathbb{R}_+)^2\setminus \overline{\Omega_0})$ for every $t\in [0,T_{\theta_0})$.
Then by using symmetry, the fact $T_*\leq \frac{\delta_G}{2\overline{C}}$, and the estimate \eqref{bound:u} 
and $2\epsilon< \frac{1}{2}\delta_G < \frac {c_*}{2\mathrm{k}}$, we find
\begin{align*}
  \big[(\mathbb{R}_+)^2 \setminus\overline{\Omega(t_0)}\big] \cap \Omega_3 = \emptyset.
\end{align*}
Due to that $t_0$ is the first time with $\mathtt{d}(t_0)=0$,
it yields that there exists some
\begin{equation*}
  x \in \partial[(\mathbb{R}_+)^2 \setminus\overline{\Omega(t_0)}] \cap [\mathcal{I}_1\cup \mathcal{I}_2],
\end{equation*}
where $\mathcal{I}_1=\{X(t_0)\}\times[0,\mathrm{k}X(t_0))$ and $\mathcal{I}_2$ is the closed straight segment connecting the points
$(X(t_0),\mathrm{k}X(t_0))$ and $(\delta_G,\mathrm{k}\delta_G)$ (see Figure \ref{fig:blow_up}).\par

\begin{figure}[htbp]
	\begin{tikzpicture}	
 
    \filldraw[fill=lightgray,draw=lightgray]  (2,0) rectangle (6,5);
		\draw[dashed] (0,5)--(2,5);
		\draw[line width=1pt,purple] (1,0)--(1,1)--(4,4)--(4,0);
		
        \draw[-stealth,line width=0.2pt] (0,0) -- (8,0);
		\draw[-stealth,line width=0.2pt] (0,0) -- (0,6.2);

		\node[anchor=center] () at (3,2) {\color{purple}{$\mathbb{K}(t_0)$}}; 
		\node[anchor=center] () at (2.5,4) {$\Omega_{3}$}; 
		
		\node[anchor=north] () at (1,0) {$X(t_0)$}; 
		\node[anchor=north] () at (2,0) {$\delta_G$}; 
		\node[anchor=north] () at (4,0) {$\frac{2c_\ast}{\mathrm{k}}$}; 
		\node[anchor=north] () at (6,0) {$\frac{5c_\ast}{2}$}; 
		\node[anchor=east] () at (0,5) {$\frac{5c_\ast}{2}$}; 
		\node[anchor=north] () at (7.7,0) {$x_1$}; 
		\node[anchor=east] () at (0,6) {$x_2$}; 
		
		\draw[decorate,decoration={brace,raise=1pt},color=brown,line width=1pt] (1,1)--(2,2); 
		\node[anchor=south east] () at (1.5,1.5) {\color{brown}{$\mathcal{I}_2$}};
		\draw[decorate,decoration={brace,raise=1pt},color=blue,line width=1pt] (1,0)--(1,1); 
		\node[anchor=east] () at (1,0.5) {\color{blue}{$\mathcal{I}_1$}};
		
	\end{tikzpicture}
	\caption{The segments $\mathcal{I}_{1}$, $\mathcal{I}_2$ and the sets $\Omega_3$, $\mathbb{K}(t_0)$.}\label{fig:blow_up}
\end{figure}
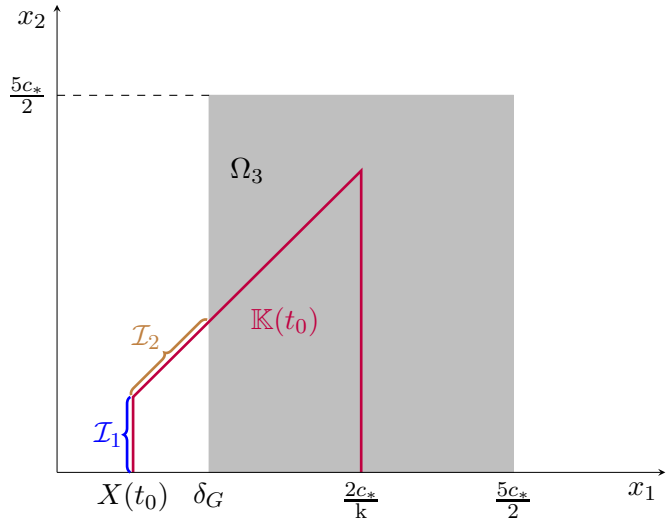

If $x\in \mathcal{I}_1$, in light of the fact $X(t_0)\leq 3\epsilon<\frac{\delta_G}{\mathrm{k}}<c_*$,
the triangle $\mathbb{A}(x)$ defined by \eqref{eqdef:A} satisfies
\begin{align*}
  \mathbb{A}(x)\subseteq \mathbb{K}(t_0)\subseteq \Omega(t_0).
\end{align*}
Consequently, Proposition \ref{prop:es-u} and $x_1=X(t_0)$ (recalling $X(t)$ satisfies \eqref{def:X(t)}) imply
\begin{align*}
  u_1(x,t_0)\leq - \mathbf{F}(x_1)<-\frac{1}{2}\mathbf{F}(x_1)= X'(t_0).
\end{align*}
Noting that $\Phi_{t_0}((\mathbb{R}_+)^2 \setminus \overline{\Omega_0}) \cap B(x,r')\neq\emptyset$
for any $r'>0$ and $u$ is continuous, we use this fact and \eqref{bound:u} to deduce
that for any sufficiently small~$s\in \big(0,\frac1{\overline{C}}[X(t_0)- \frac{x_2}{\mathrm{k}}]\big)$, 
\begin{align*}
  \Phi_{t_0-s}\big((\mathbb{R}_+)^2 \setminus \overline{\Omega_0}\big)
  \cap \Big( \big(X(t_0-s),\tfrac{2c_*}{\mathrm{k}}\big)\times \big(0,\mathrm{k} X(t_0)\big)\Big) \neq \emptyset.
\end{align*}
Since $(X(t_0-s),\frac{2c_*}{\mathrm{k}})\times(0,\mathrm{k}X(t_0))\subseteq \mathbb{K}(t_0-s)$,
we obtain $\mathtt{d}(t_0-s)=0$ for these $s$, which is a contradiction with the choice of $t_0$.

If $x\in \mathcal{I}_2$, then $\mathrm{k}x_1 = x_2\leq\delta_G$, and
Proposition \ref{prop:es-u} guarantees that
\begin{align*}
  u_1(x,t_0) & \leq - \mathbf{F}(x_1) < - \tfrac {1}{2} \mathbf{F}(x_1) =  X'(t_0), \\
  u_2(x,t_0) & \geq  \mathbf{F}(\tfrac{x_2}{\mathrm{k}}) > \tfrac{1}{2} \mathbf{F}(x_1) = - X'(t_0) , 
\end{align*}
and thus for $\mathbf{n} =\frac{1}{\sqrt{\mathrm{k}^2 +1}} (-\mathrm{k},1)$ 
the outer normal unit vector of the $\mathcal{I}_2$ segment,
\begin{align*}
  u(x, t_0) \cdot \mathbf{n} >  
  - \frac{\mathrm{k}+1}{\sqrt{\mathrm{k}^2 +1}} X'(t_0).
\end{align*}
Hence, by applying a similar argument, we infer that for any sufficiently small $s>0$,
\begin{align*}
  \Phi_{t_0-s}((\mathbb{R}_+)^2 \setminus \overline{\Omega_0})\cap
  \Big(\big(x_1+X(t_0-s)-X(t_0),\tfrac{2c_*}{\mathrm{k}}\big)
  \times \big(0,\mathrm{k} \big(x_1-X(t_0-s)+X(t_0)\big)\big)\Big)
  \neq \emptyset.
\end{align*}
Since $X(t_0)\leq x_1$ implies $(x_1+X(t_0-s)-X(t_0),\frac{2c_*}{\mathrm{k}})\times(0,\mathrm{k}x_1)
\subseteq \mathbb{K}(t_0-s)$,
we again get a contradiction. 

Therefore, we complete the proof of Theorem \ref{thm:main_blowup}.
%
%
%
%
%
%
%
%
%
%
\end{proof}

\appendix

\section{Proof of Propositions \ref{prop:flow-map}, \ref{pro:ap-es} and \ref{prop:reg-param}}\label{sec:appendix}

In this section, we present detailed proofs to Propositions \ref{prop:flow-map}, \ref{pro:ap-es} and \ref{prop:reg-param}.

\subsection{Proof of Proposition \ref{prop:flow-map}}\label{subsec:flow-map}
We only focus on the case $\mathbf{D} = \mathbb{R}^2_+$, 
since the whole space case is easier.


Denote by $d_t(x) \triangleq \mathrm{dist}\,(x,\partial D(t))$. 
To start with, for the patch solution $\theta$ given by \eqref{eq:the-patch-sol} 
and the velocity $u$ given by \eqref{eq:u-exp}, 
we claim that for any $x\in \overline{\mathbb{R}^2_+}\setminus \partial D(t)$ 
such that $d_t(x) \geq \mathrm{d}_0$, 
\begin{align}\label{eq:Lipschitz_est}
  |\nabla u(x,t)|\leq C \big( \mathrm{d}_0^{-\alpha} + 1\big),
\end{align}
which means that the velocity $u(x,t)$ is Lipschitzian about $x$ when $x$ is away from the boundary $\partial D(t)$.  
Indeed, recalling that $u(x,t)$ has the formula \eqref{eq:exp-u-extend} by odd extension and $\chi(x) = \chi(|x|)\in C^\infty_c(\mathbb{R}^2)$ 
is a cutoff function satisfying \eqref{eq:chi-prop}, and using the fact that $\theta \equiv \mathrm{const}$ on $B(x,\mathrm{d}_0)$ 
and the kernel $\frac{x^\perp}{|x|^2} G(|x|)$ is mean-zero, 
we have that for every $z\in B(x,\frac{\mathrm{d}_0}{2})$,
\begin{align*}
  u(z,t) =  \int_{\mathbb{R}^2}\frac{(z-y)^{\perp}}{|z-y|^2}G(|z-y|) 
  \Big(1 - \chi \Big( \frac{|z-y|}{\mathrm{d}_0/4}\Big) \Big) e_o[\theta](y,t)\dd y,
\end{align*}
which combined with \eqref{conds:G-s}-\eqref{conds:G-l} leads to
\begin{align*}
  |\nabla u(x,t)|  & \leq C 
  \int_{\mathbb{R}^2} \Big(\frac{|G(|x-y|)|}{|x-y|^2} + \frac{|G'(|x-y|)|}{|x-y|}  \Big) 
  \Big(1 - \chi \Big( \frac{|x-y|}{\mathrm{d}_0/4}\Big) |e_o[\theta](y,t)| \dd y \\
  & \quad +  C 
  \int_{\mathbb{R}^2} \frac{4 |G(|x-y|)|}{\mathrm{d}_0|x-y|} \Big| \chi'\Big( \frac{|x-y|}{\mathrm{d}_0/4}\Big) \Big|  |e_o[\theta](y,t)| \dd y \\
  & \leq C \Big( \|\theta(t)\|_{L^\infty} \int_{\frac{\mathrm{d}_0}{4}}^{c_0} \frac{1}{r^{1+\alpha}} \dd r + \|\theta(t)\|_{L^1} 
  + \|\theta(t)\|_{L^\infty} \int_{\frac{\mathrm{d}_0}{4}}^{\frac{\mathrm{d}_0}{2}} \frac{r^{-\alpha}+1}{\mathrm{d_0}}\dd r \Big) \\
  & \leq C \|\theta(t)\|_{L^1\cap L^\infty} \big(\mathrm{d}_0^{-\alpha} +1 \big),
\end{align*}
as desired.

For the proof of ($\mathbf{i}$), as discussed in \cite{KYZ17}, we mainly 
need to verify that there exists $0 < \mathrm{B} < +\infty$ 
such that for each $T'\in (0,T)$ and $(x,t)\in (\overline{\mathbb{R}^2_+}\setminus \partial D(0)) \times [0,T']$,
\begin{align}\label{eq:distance_claim}
  d_t(\Phi_t(x))\geq e^{-\mathrm{B} t}d_0(x),\quad \text{and}\quad \Phi_t^{(2)}(x) \geq e^{-\mathrm{B} t}x_2,
\end{align}
with $\Phi_t^{(2)}(x) = \Phi_t(x) \cdot (0,1)$.
In fact, according to \eqref{eq:Lipschitz_est} and \eqref{eq:distance_claim}, the particle trajectory $\Phi_t(x)$ 
solving ODE \eqref{eq:flow_map} is unique for each $(x,t)\in (\overline{\mathbb{R}^2_+}\setminus\partial D(0))\times [0,T)$, 
which implies that the map $\Phi_t \,:\, \overline{\mathbb{R}^2_+}\setminus \partial D(0) \rightarrow 
\overline{\mathbb{R}^2_+}\setminus \partial D(t)$ is injective for any $t\in [0,T)$; 
while for the surjectivity of the map $\Phi_t$, it follows by solving the ODE \eqref{eq:flow_map} backwards in time with 
any terminal point $y = \Phi_t(x) \in \overline{\mathbb{R}^2_+}\setminus \partial D(t)$. 
	
Next, note that the map $t\in [0,T']\mapsto d_t(\Phi_t(x))$ is Lipschitzian (by using the same argument as in \cite[Lemma 4.10]{KYZ17}). 
Thus in order to get \eqref{eq:distance_claim}, by virtue of Gr\"onwall's inequality and \eqref{eq:flow_map}, 
it suffices to prove that for any $t\in [0,T']$,
\begin{align}\label{eq:dist_diff_ineq}
  \frac{\dd d_t(\Phi_t(x))}{\dd t}\geq - \mathrm{B} d_t(\Phi_t(x)), \quad |u_2(x,t)|\leq \mathrm{B} x_2. 
\end{align}

We first prove the second inequality in \eqref{eq:dist_diff_ineq}. 
Consider $\mathbf{z} = (z_1,\cdots,z_N)$, and each $z_k(\zeta,t)$ a parametrization of the $C^{1,\sigma}$ patch boundary $\partial D_k(t)$ 
as in the subsection \ref{subsec:contour}, then analogous with \eqref{def:norm-whole}, we denote that 
\begin{align*}
  \|\mathbf{z}\|_{W_\sigma} \triangleq 
  \lVert \mathbf{z} \rVert_{C^{1,\sigma}} 
  + \sum_{k=1}^N\|F[z_k]\|_{L^\infty} + \delta[\mathbf{z}]^{-1} + 1.
\end{align*}
Since $\theta$ given by \eqref{eq:the-patch-sol} is a $C^{1,\sigma}$ patch solution to the equation
\eqref{eq:geSQG}$\&$\eqref{eq:u-exp} on $[0,T)$, we have $\|\mathbf{z}(t)\|_{W_\sigma} < +\infty$ for every $t\in (0,T)$.
Similar as deriving \eqref{eq:velocity_normal}, we deduce that
\begin{align*}
  u_2(x,t) & =  \sum_{j=1}^N a_j \int_{D_j(t)} \Big(\frac{y_1 -x_1}{|x-y|^2}G(|x-y|)
  -\frac{y_1-x_1}{|x-\overline{y}|^2}G(|x-\overline{y}|)\Big)  \dd y \\ 
  & = - \sum_{j=1}^N a_j \int_{D_j(t)} \partial_{y_1}\Big( R(|x-y|)
  - R(|x-\overline{y}|)\Big) 
  \dd y \\ 
  & = - \sum_{j=1}^N a_j \int_{\partial D_j(t)} n^{(1)}(y,t)  \Big( R(|x-y|)  - R(|x - \overline{y}|)
  \Big) \, \dd \sigma(y) \\
  & = \sum_{j=1}^{N}a_j \int_{\TT} \Big( R(|x-z_j(\eta,t)|)-R(|x - \overline{z}_j(\eta,t)|) 
  \Big)  \partial_\eta (z_j^{(2)}(\eta,t)) \, \dd \eta,
\end{align*}
where $R(\cdot)$ is defined by \eqref{def:R}, $n(y,t) = (n^{(1)}, n^{(2)})(y,t)$ for each $y\in \partial D_j(t)$ and $z_j = (z_j^{(1)}, z_j^{(2)})$. 
Thanks to \cite[Lemma 2.2]{KYZ17} (see \eqref{eq:fact-half-plane} below), we have 
\begin{align*}
  |\partial_\eta(z_j^{(2)}(\eta))|\leq C\lVert z_j\rVert^{\frac{1}{1+\sigma}}_{C^{1,\sigma}}(z_j^{(2)}(\eta))^{\frac{\sigma}{1+\sigma}}
  \leq C\|\mathbf{z}\|_{W_\sigma}^{\frac{1}{1+\sigma}} (z_j^{(2)}(\eta))^{\frac{\sigma}{1+\sigma}}. 
\end{align*}
This inequality and the mean value theorem lead to that 
\begin{align*}
  |u_2(x,t)| \leq 2 \|\mathbf{z}\|_{W_\sigma}^{\frac{1}{1+\sigma}} x_2 \sum_{j=1}^{N}|a_j| 
  \int_{\TT}\int_0^1 |R'(\tau|x-z_j(\eta)|+(1-\tau)|x-\bar{z}_j(\eta)|)|\, (z_j^{(2)}(\eta))^{\frac{\sigma}{1+\sigma}}\dd \tau \dd \eta.
\end{align*}
Observing that 
$|x-z_j|\leq |x-\overline{z}_j|$, $|x-\overline{z}_j|\geq z_j^{(2)}$,
together with \eqref{es:R-deri}, we have 
\begin{align*}
  |u_2(x,t)| & \leq C \|\mathbf{z}\|_{W_\sigma}^{\frac{1}{1+\sigma}} x_2 \sum_{j=1}^{N} 
  \bigg( \int_{\TT}\int_0^1 
  \frac{(z_j^{(2)}(\eta))^{\frac{\sigma}{1+\sigma}}}{\big(\tau|x-z_j(\eta)|+(1-\tau)|x-\overline{z}_j(\eta)|\big)^{1+\alpha}} \dd \tau\dd \eta  \\
  & \qquad\qquad \qquad \qquad  + \int_{\TT}\int_0^1 
  \frac{(z_j^{(2)}(\eta))^{\frac{\sigma}{1+\sigma}}}{\tau|x-z_j(\eta)|+(1-\tau)|x-\overline{z}_j(\eta)|} \dd \tau\dd \eta\bigg)\\
  & \leq C \|\mathbf{z}\|_{W_\sigma}^{\frac{1}{1+\sigma}} x_2 \sum_{j=1}^N \int_{\TT}\int_0^1 (1-\tau)^{-\frac{\sigma}{1+\sigma}} \dd \tau 
  \Big( |x-z_j(\eta)|^{-\frac{1}{1+\sigma}-\alpha}  + |x-z_j(\eta)|^{-\frac{1}{1+\sigma}} \Big)\dd \eta \\
  & \leq C \|\mathbf{z}\|_{W_\sigma}^{\frac{1}{1+\sigma}} x_2 \sum_{j=1}^{N} 
  \Big(\int_{\TT} |x-z_j(\eta)|^{-\frac{1}{1+\sigma}-\alpha}\dd \eta +  \int_{\TT} |x-z_j(\eta)|^{-\frac{1}{1+\sigma}}\dd \eta \Big).
\end{align*}
In addition, arguing as the estimation of $T_k$ in \cite[pp. 1309]{KYZ17}, we know that for every $\alpha\in (0,\frac{1}{2})$
and $\sigma\in (\frac{\alpha}{1-\alpha},1]$,
\begin{align*}
  \int_{\TT} |x-z_j(\eta)|^{-\frac{1}{1+\sigma}-\alpha}\dd \eta + \int_{\TT} |x-z_j(\eta)|^{-\frac{1}{1+\sigma}}\dd \eta 
  \leq C(\alpha,\sigma) \|\mathbf{z}\|_{W_\sigma},
\end{align*}
which implies the wanted estimate that for every $t\in [0,T']$, 
\begin{align*}
  |u_2(x,t)|\leq C N \|\mathbf{z}\|_{L^\infty([0,T']; W_\sigma)}^{\frac{2+\sigma}{1+\sigma}} x_2.
\end{align*}

For the first inequality in \eqref{eq:dist_diff_ineq}, as showed in \cite[Lemma 5.6]{KYZ17}, 
we only need to prove that for any $x\in \overline{\mathbb{R}^2_+}\setminus \partial D(t)$ and any $P\in \partial D(t)$ 
such that $|x-P| = d_t(x)$, we have that for every $t\in [0,T']$,
\begin{align}\label{eq:difference_est}
  \Big|(u(x)-u(P))\cdot \tfrac{x-P}{|x-P|}\Big|\leq C \|\mathbf{z}\|_{L^\infty([0,T'];W_\sigma)} |x-P|. 
\end{align}
We first prove 
\begin{align}\label{eq:gradient_est}
	\Big|\big(\tfrac{x-P}{|x-P|}\cdot\nabla\big) u(x,t)\cdot \tfrac{(x-P)}{|x-P|}\Big|\leq C \|\mathbf{z}\|_{L^\infty([0,T'];W_\sigma)}.
\end{align}
Define $\mathrm{R}_t \triangleq (4\|\mathbf{z}(t)\|_{W_\sigma})^{-\frac{1}{\sigma}-1}$. 
For every $x$ such that $|x-P|\geq \frac{\mathrm{R}_t}{2}$, according to \eqref{eq:Lipschitz_est}, and noting that 
$\alpha(\frac{1}{\sigma} + 1) < 1$, we find 
\begin{align*}
  \Big|\big(\tfrac{x-P}{|x-P|}\cdot\nabla\big) u(x)\cdot\tfrac{(x-P)}{|x-P|}\Big|\leq C \mathrm{R}_t^{-\alpha}
  \leq C \|\mathbf{z}(t)\|_{W_\sigma}.
\end{align*}
Suppose that $P\in \partial D_k(t)$ for some $k$ and consider each $x$ satisfying 
$|x-P| \leq \frac{\mathrm{R}_t}{2}$. We split $u(x,t)$ given by \eqref{eq:u-exp} as 
\begin{align*}
  u(x,t) = \sum_{j=1}^N a_j \big(\mathrm{u}^j(x,t) - \widetilde{\mathrm{u}}^j(x,t)\big),
\end{align*}
where $\widetilde{D}_j(t) = \{(x_1,x_2)|(x_1,-x_2)\in D_j(t)\}$ and 
\begin{align*}
  \mathrm{u}^j(x,t) = \int_{D_j(t)}\frac{(x-y)^\perp}{|x-y|^2}G(|x-y|)\dd y,\quad \widetilde{\mathrm{u}}^j(x,t)
  = \int_{\widetilde{D}_j(t)} \frac{(x-y)^{\perp}}{|x-y|^2}G(|x-y|)\dd y.
\end{align*}
For every $j\ne k$, we have $\mathrm{dist}\,(x,\partial D_j(t)) \geq \mathrm{dist}\,(\partial D_j(t),\partial D_k(t))
- \mathrm{dist}\,(x,\partial D_k(t)) \geq \tfrac{1}{\|\mathbf{z}(t)\|_{W_\sigma}}
-\tfrac{\mathrm{R}_t}{2}\geq \frac{\mathrm{R}_t}{2}$. 
In addition, we also see $\mathrm{dist}\,(x,\partial \widetilde{D}_j(t)) \geq \mathrm{dist}\,(x,\partial D_j(t))$. 
Thus, arguing as \eqref{eq:Lipschitz_est}, we find that for every $j\neq k$,
\begin{align*}
  \Big|\big(\tfrac{x-P}{|x-P|}\cdot\nabla\big) \mathrm{u}^j(x,t)\cdot\tfrac{(x-P)}{|x-P|}\Big|\leq |\nabla \mathrm{u}^j(x,t)|
  \leq C \mathrm{R}_t^{-\alpha} \leq C \|\mathbf{z}(t)\|_{W_\sigma},
\end{align*} 
and 
\begin{align*}
  \Big|\big(\tfrac{x-P}{|x-P|}\cdot\nabla\big) \widetilde{\mathrm{u}}^j(x,t)\cdot\tfrac{(x-P)}{|x-P|}\Big|
  \leq C \|\mathbf{z}(t)\|_{W_\sigma}. 
\end{align*} 
For the term involving $\mathrm{u}^k$, noting that $n_P(x)\triangleq\frac{x-P}{|x-P|}$ is the normal vector at 
$P\in \partial D_k(t)$, 
we have to develop some cancellation of the normal derivative of $\mathrm{u}^k$. By setting 
\begin{align*}
  S_P(x)\triangleq \Big\{ y\in \RR^2:(y-P)\cdot n_P\in (- \mathrm{R}_t,0),\;
  \big|(y-P)\cdot n^{\perp}_P\big|\le \mathrm{R}_t \Big\},
\end{align*}
and
\begin{align*}
  u_{S_P}(x) \triangleq \int_{S_P(x)} \tfrac{(x-y)^\perp}{|x-y|^2} G(|x-y|) \dd y,
\end{align*}
and using symmetry, 
we observe that
\begin{align*}
  n_P(x)\cdot u_{S_P}(x)=\tfrac{x-P}{|x-P|}\cdot \int_{S_P(x)} \tfrac{(x-y)^\perp}{|x-y|^2} G(|x-y|) \dd y = 0.
\end{align*}
Let $\mathrm{x}_r \triangleq  x+r\tfrac{x-P}{|x-P|}$ with $|r|\ll |x-P|$. 
Due to that $n_P(\mathrm{x}_r)$ is also the normal vector at $P\in \partial D_k(t)$ 
and $n_P(\mathrm{x}_r)=n_P(x)$, it follows that 
\begin{align}\label{eq:cancellations_r}
  0=n_P(\mathrm{x}_r)\cdot u_{S_P}(\mathrm{x}_r)=n_P(x)\cdot \int_{S_P(x)} 
  \tfrac{(\mathrm{x}_r-y)^\perp}{|\mathrm{x}_r-y|^2} G(|\mathrm{x}_r-y|) \dd y,\quad \forall |r|\ll |x-P|. 
\end{align} 
Differentiating \eqref{eq:cancellations_r} with respect to $r$ at $r=0$, it follows that 
\begin{align*}
  \Big((n_P(x)\cdot \nabla_{z})\int_{S_P(x)} \tfrac{(z-y)^\perp}{|z-y|^2} G(|z-y|) \dd y \Big)
  \Big|_{z=x}\cdot n_P(x)=0.
\end{align*}
Thus
\begin{align*}
  \Big|\big( n_P(x) \cdot \nabla_x\big) \mathrm{u}^k(x,t) \cdot n_P(x)\Big|
  & = \Big|\big( n_P(x) \cdot \nabla_z\big) \Big(\mathrm{u}^k(z,t) 
  - \int_{S_P(x)} \tfrac{(z-y)^\perp}{|z-y|^2} G(|z-y|) \dd y \Big)\Big|_{z=x} \cdot n_P(x)\Big| \\
  & \leq C \int_{D_k(t)\triangle S_P} \Big(\tfrac{|G(|x-y|)|}{|x-y|^2}  + \tfrac{|G'(|x-y|)|}{|x-y|}\Big) \dd y,
\end{align*}
where we have used the notation $A\triangle B = (A\setminus B ) \cup (B\setminus A)$. 
In view of \eqref{conds:G-s}-\eqref{conds:G-l} and without loss of generality assuming $\mathrm{R}_t\leq c_0$, 
we deduce that 
\begin{equation}\label{eq:u_k_est}
\begin{aligned}
  \int_{D_k(t)\triangle S_P} \Big(\tfrac{|G(|x-y|)|}{|x-y|^2}  + \tfrac{|G'(|x-y|)|}{|x-y|}\Big) \dd y
  & \leq C \int_{(D_k(t)\triangle S_P) \cap B(P,\mathrm{R}_t)}  \tfrac{1}{|x-y|^{2+\alpha}}  \dd y \\
  & \quad + C \int_{(D_k(t)\cup S_P) \setminus B(P,\mathrm{R}_t)} 
  \Big(\tfrac{1}{|x-y|}+\tfrac{1}{|x-y|^{2+\alpha}}\Big) \dd y.
\end{aligned}
\end{equation}
Since $|x-P|\leq \tfrac{\mathrm{R}_t}{2}$, $|D_k(t)|=|D_k(0)|\leq C$ and $|S_P|\leq C\mathrm{R}^2_t$, 
by using the rearrangement inequality, we infer that  
\begin{align*}
  \int_{(D_k(t)\cup S_P)\setminus B(P,\mathrm{R}_t)}\tfrac{1}{|x-y|} \dd y
  \leq \int_{B(x,\sqrt{|D_k(0)|/\pi})} \frac{1}{|x-y|} \dd y + \int_{B(x,C \mathrm{R}_t)} \frac{1}{|x-y|} \dd y 
  \leq C,
\end{align*}
and
\begin{align*}
  \int_{(D_k(t)\cup S_P)\setminus B(P,\mathrm{R}_t)} \frac{1}{|x-y|^{2+\alpha}}\dd y
  \leq \int_{\mathbb{R}^2\setminus B(x,\mathrm{R}_t)} \frac{1}{|x-y|^{2+\alpha}} \dd y 
  \leq C \mathrm{R}_t^{-\alpha} \leq C \|\mathbf{z}(t)\|_{W_\sigma}.
\end{align*}
The remaining term in \eqref{eq:u_k_est} is the same with $I_1$ in \cite[Lemma 5.1]{KYZ17}, hence we conclude that
\begin{align*}
  \Big|(\tfrac{x-P}{|x-P|}\cdot\nabla) \mathrm{u}^k(x,t)\cdot \tfrac{x-P}{|x-P|}\Big|\leq C \|\mathbf{z}(t)\|_{W_\sigma}. 
\end{align*}
For the term involving $\widetilde{\mathrm{u}}^k$, we can bound it as the estimation of $\nabla\tilde{v}(x) n_P$ 
in \cite[Proposition 5.3]{KYZ17}:
\begin{align*}
  \Big|\big(\tfrac{x-P}{|x-P|}\cdot\nabla\big) \widetilde{\mathrm{u}}^k(x,t)\cdot \tfrac{x-P}{|x-P|}\Big|\leq C \|\mathbf{z}(t)\|_{W_\sigma}. 
\end{align*}
Collecting the above estimates immediately yields \eqref{eq:gradient_est}. 
Next, noting that if $x_\tau= P+\tau(x-P)$ for $\tau\in [0,1]$, 
then $\mathrm{dist}\,(x_\tau,P) = \mathrm{dist}\,(x_\tau,\partial D(t))$, and by using the mean value theorem, 
there exists some $\tau\in (0,1)$ such that 
\begin{align*}
  \Big|(u(x,t)-u(P,t))\cdot \tfrac{x-P}{|x-P|}\Big|=&\Big|\big(\tfrac{x-P}{|x-P|}\cdot\nabla\big) u(x_\tau,t)\cdot \tfrac{x-P}{|x-P|}\Big||x-P|\\
  =&\Big|\big(\tfrac{x_\tau - P}{|x_\tau-P|}\cdot\nabla\big) u(x_\tau,t)\cdot \tfrac{x_{\tau}-P}{|x_{\tau}-P|}\Big||x-P|.
\end{align*}
Hence, this equality and \eqref{eq:gradient_est} imply \eqref{eq:difference_est}. 

The proof of $\mathbf{(ii)}$ is exactly the same with the proof of \cite[Proposition 1.3 (b)]{KYZ17}, and we omit the details. 

\subsection{Proof of Proposition \ref{pro:ap-es}}\label{subsec:ap-es}
Taking the second derivative with respect to $\zeta$ in both sides of \eqref{eq:main-eq-GP}
and making a dot product with $\partial^2_\zeta z_k$,
we obtain that
\begin{align}\label{eq:z_k-H2}
  \frac{1}{2}\frac{\dd }{\dd t}\lVert z_k\rVert^2_{\dot{H}^2} = \int_{\TT}\partial^2_\zeta z_k(\zeta)\cdot
  \partial^2_\zeta \Big(\mathrm{NL}_{j=k}(\zeta,t) + \mathrm{NL}_{j\ne k}(\zeta,t)
  + \lambda_k(\zeta,t) \partial_\zeta z_k(\zeta,t)\Big) \dd \zeta,
\end{align}
where $\mathrm{NL}_{j=k}$ is the term with $j = k$ in the sum of \eqref{eq:NL-k-GP}
and $\mathrm{NL}_{j\ne k}$ are the other terms in the sum of \eqref{eq:NL-k-GP}.
In the sequel, we shall separately estimate the nonlinear terms $\mathrm{NL}_{j=k}$, $\mathrm{NL}_{j\ne k}$
and the tangential term $\lambda_{k}(\zeta,t)$.

\subsubsection{Estimation of the $\mathrm{NL}_{j=k}$ term in \eqref{eq:z_k-H2}}

First, $\mathrm{NL}_{j=k}$ can be split into the following
\begin{align}\label{eq:NLj=k}
  \mathrm{NL}_{j=k}(\zeta,t) =\, \mathrm{O}_k(\zeta,t) + \, \mathrm{N}_k(\zeta,t),
\end{align}
where
\begin{align*}
  \mathrm{O}_k(\zeta,t) \triangleq  a_k\int_{\mathbb{T}}\big(\partial_\zeta z_k(\zeta,t)
  - \partial_\zeta z_k(\zeta - \eta, t)\big) R\big(| z_k(\zeta,t)  - z_k(\zeta-\eta, t)|\big)  \dd\eta,
\end{align*}
\begin{align*}
  \mathrm{N}_k(\zeta,t) \triangleq a_k\int_{\mathbb{T}}\big(\partial_\zeta z_k(\zeta,t)
  - \partial_\zeta \overline{z}_k(\zeta - \eta, t)\big) R\big(| z_k(\zeta,t)  - \overline{z}_k(\zeta-\eta, t)|\big)  \dd\eta.
\end{align*}
We shall focus on $N_k$, which contains the additional term from the 
half-plane setting and is more singular.
By virtue of Leibniz's rule, we have
\begin{align*}
  \int_{\TT}\partial^2_\zeta z_k(\zeta)\cdot \partial^2_\zeta \mathrm{N}_k(\zeta)\, \dd \zeta = I_1 + I_2 + I_3,
\end{align*}
where
\begin{align*}
  I_1 & \triangleq \int_{\mathbb{T}}\int_{\mathbb{T}} \partial^2_\zeta z_k(\zeta)\cdot \big(\partial^3_\zeta z_k(\zeta)
  - \partial^3_\zeta \overline{z}_k(\zeta - \eta)\big) R\big(| z_k(\zeta)  - \overline{z}_k(\zeta -\eta)|\big)  \dd\eta \dd \zeta,\\
  I_2 & \triangleq 2\int_{\mathbb{T}}\int_{\mathbb{T}} \partial^2_\zeta z_k(\zeta)\cdot \big(\partial^2_\zeta z_k(\zeta)
  - \partial^2_\zeta \overline{z}_k(\zeta - \eta)\big) \partial_\zeta R\big(| z_k(\zeta)  - \overline{z}_k(\zeta-\eta)|\big)
  \dd\eta \dd \zeta, \\
  I_3 & \triangleq \int_{\mathbb{T}}\int_{\mathbb{T}} \partial^2_\zeta z_k(\zeta)\cdot \big(\partial_\zeta z_k(\zeta)
  - \partial_\zeta \overline{z}_k(\zeta - \eta)\big) \partial^2_\zeta R\big(| z_k(\zeta)  - \overline{z}_k(\zeta-\eta)|\big)
  \dd\eta \dd \zeta.
\end{align*}
Noting that
\begin{equation}\label{eq:I1-symm}
\begin{split}
  I_1 & = \int_{\mathbb{T}}\int_{\mathbb{T}} \partial_\zeta^2 \overline{z}_k(\zeta) \cdot
  \big( \partial_\zeta^3 \overline{z}_k(\zeta) - \partial_\zeta^3 z_k(\zeta -\eta) \big)
  R\big( |z_k(\zeta) - \overline{z}_k(\zeta-\eta)|\big) \dd \eta \dd \zeta \\
  & = \int_{\mathbb{T}} \int_{\mathbb{T}} \partial_\zeta^2 \overline{z}_k(\zeta -\eta)
  \cdot \big(\partial_\zeta^3 \overline{z}_k(\zeta -\eta) - \partial_\zeta^3 z_k(\zeta) \big)
  R\big(|z_k(\zeta) - \overline{z}_k(\zeta -\eta)| \big) \dd \eta \dd \zeta,
\end{split}
\end{equation}
we use the integration by parts to deduce
\begin{equation}\label{eq:GP-I_1}
\begin{split}
  I_1  = - \frac{1}{4}\int_{\mathbb{T}}\int_{\mathbb{T}}|\partial^2_\zeta z_k(\zeta)-
  \partial^2_\zeta \overline{z}_k(\zeta - \eta)|^2
  \partial_\zeta R\big(| z_k(\zeta)  - \overline{z}_k(\zeta-\eta)|\big)  \dd\eta \dd \zeta .
\end{split}
\end{equation}
For every $k,j\in\{1,\cdots,N\}$, we denote
$Z_{k,j}(\zeta,\eta,t) \triangleq z_k(\zeta,t)  - \overline{z}_j(\zeta-\eta, t)$,
and we also abbreviate $Z_{k,j}(\zeta,\eta,t)$ as $Z_{k,j}$ if the variables $(\zeta,\eta,t)$ are clear in the context.
Notice that
\begin{align*}
  \partial_\zeta R(|Z_{k,j}|) = R'(|Z_{k,j}|) \frac{ \partial_\zeta Z_{k,j}(\zeta,\eta,t)
  \cdot Z_{k,j}(\zeta,\eta,t)}{|Z_{k,j}(\zeta,\eta,t)|}.
\end{align*}
In order to control $|\partial_\zeta Z_{k,k}(\zeta,\eta)|$, 
we recall an important observation in \cite[Lemma 2.2]{KYZ17} 
that for every $g(\zeta)\geq 0$,
\begin{align}\label{eq:fact-half-plane}
  |\partial_\zeta g(\zeta)| \leq 2 \|\partial_\zeta g\|_{C^\sigma}^{\frac{1}{1+\sigma}} g(\zeta)^{\frac{\sigma}{1+\sigma}},
  \quad \forall \zeta\in\mathbb{T}, \sigma\in[0,1].
\end{align}
Thus, in combination with the following facts that
\begin{align}\label{eq:Zkk-fact}
  |\eta| \leq F[z_k](\zeta,\eta) |z_k(\zeta) - z_k(\zeta -\eta)|
  \leq \|F[z_k]\|_{L^\infty} |Z_{k,k}(\zeta,\eta)|,
\end{align}
and (recalling $z_k = (z_k^{(1)}, z_k^{(2)})^T$)
\begin{equation}\label{eq:Zkk-fact2}
\begin{split}
  2 z_k^{(2)}(\zeta -\eta) & \leq |z_k(\zeta) -z_k(\zeta -\eta)| + |Z_{k,k}(\zeta,\eta)| \\
  & \leq \|z_k\|_{C^1} |\eta| + |Z_{k,k}(\zeta,\eta)|
  \leq \big(\|z_k\|_{C^1} \|F[z_k]\|_{L^\infty} + 1 \big) |Z_{k,k}(\zeta,\eta)|,
\end{split}
\end{equation}
we find the crucial inequality as follows (see also \cite[Eq.(18)]{GanP21}),
\begin{align}\label{eq:diff-pointwise-es}
  |\partial_\zeta Z_{k,k}(\zeta,\eta)|
  & \leq |\partial_\zeta z_k(\zeta) - \partial_\zeta z_k(\zeta  -\eta)| + 2 |\partial_\zeta z_k^{(2)}(\zeta-\eta)| \nonumber \\
  & \leq \|\partial_\zeta z_k\|_{C^{\frac{1}{3}}} |\eta|^{\frac{1}{3}}
  + 4 \|\partial_\zeta z_k\|_{C^{\frac{1}{2}}}^{\frac{2}{3}} \big(z_k^{(2)}(\zeta-\eta) \big)^{\frac{1}{3}} \nonumber \\
  & \leq C \big(\lVert z_k\rVert_{H^2}
  \lVert F[z_k]\rVert^{\frac{1}{3}}_{L^\infty}
  + \lVert z_k\rVert^{\frac{2}{3}}_{H^2}\big)|Z_{k,k}(\zeta,\eta)|^{\frac{1}{3}}.
\end{align}
In addition, from \eqref{es:R-deri} we see that
$|R'(|Z_{k,j}|)| \leq C \big(|Z_{k,j}|^{-1-\alpha} + |Z_{k,j}|^{-1} \big)$.
Hence, we use the above estimates to obtain that
\begin{align}\label{es:I1}
  & |I_1| + |I_2| \leq C \int_{\mathbb{T}} \int_{\mathbb{T}} \big(|\partial_\zeta^2 z_k(\zeta)|^2
  + |\partial_\zeta^2 \overline{z}_k(\zeta -\eta)|^2\big) |R'(|Z_{k,k}|)|\, |\partial_\zeta Z_{k,k}(\zeta,\eta)| \dd \eta \dd \zeta \nonumber \\
  & \leq C  \big(\lVert z_k\rVert_{H^2} \lVert F[z_k]\rVert^{\frac{1}{3}}_{L^\infty}
  + \lVert z_k\rVert^{\frac{2}{3}}_{H^2}\big) \int_{\mathbb{T}} \int_{\mathbb{T}}
  \big( |\partial_\zeta^2 z_k(\zeta)|^2 + |\partial_\zeta^2 \overline{z}_k(\zeta -\eta)|^2 \big)
  \big( |Z_{k,k}|^{-\frac{2}{3}-\alpha}  + |Z_{k,k}|^{-\frac{2}{3}}\big) \dd \eta \dd \zeta \nonumber \\
  & \leq C  \big(\lVert z_k\rVert_{H^2} \lVert F[z_k]\rVert_{L^\infty}^{\frac{1}{3}}
  + \lVert z_k\rVert^{\frac{2}{3}}_{H^2}\big) \big(\lVert F[z_k]\rVert_{L^\infty}^{\frac{2}{3}+\alpha} + 1\big)
  \|z_k\|_{H^2}^2 \int_{\mathbb{T}} \big( |\eta|^{-\frac{2}{3}-\alpha}  + |\eta|^{-\frac{2}{3}}\big) \dd \eta \nonumber  \\
  & \leq C \big(\lVert F[z_k]\rVert^{1+\alpha}_{L^\infty} + 1 \big) \big(\lVert z_k\rVert_{H^2}  + 1 \big)
  \lVert z_k\rVert^2_{H^2}.
\end{align}


By direct computation, $I_3$ can be rewritten as
\begin{align}
  I_3 & = \int_{\TT} \int_{\TT} \partial^2_\zeta z_k(\zeta)\cdot \partial_\zeta Z_{k,k}(\zeta,\eta)\, R'(|Z_{k,k}|)
  \frac{Z_{k,k}\cdot \big(\partial^2_\zeta z_k(\zeta) - \partial_\zeta^2 \overline{z}_k(\zeta -\eta) \big)}
  {|Z_{k,k}|} \dd \eta \dd \zeta \nonumber \\
  & \quad + \int_{\mathbb{T}}\int_{\TT} \partial_\zeta^2 z_k(\zeta)
  \cdot  \partial_\zeta Z_{k,k}(\zeta,\eta)  R'(|Z_{k,k}|)
  \frac{|\partial_\zeta Z_{k,k}|^2}{|Z_{k,k}|} \dd \eta \dd \zeta \nonumber \\
  & \quad + \int_{\mathbb{T}}\int_{\TT} \partial_\zeta^2 z_k(\zeta)
  \cdot \partial_\zeta Z_{k,k} \Big(-R'(|Z_{k,k}|)
  \frac{(Z_{k,k}\cdot \partial_\zeta Z_{k,k})^2}{|Z_{k,k}|^3}
  + R''(|Z_{k,k}|) \frac{(Z_{k,k}\cdot \partial_\zeta Z_{k,k})^2}{|Z_{k,k}|^2}\Big) \dd \eta \dd \zeta \nonumber \\
  & \triangleq I_{31} + I_{32} + I_{33}.
\end{align}
The term $I_{31}$ can be estimated exactly the same as $I_2$.
By using the symmetrization trick as in \eqref{eq:I1-symm}, we  have
\begin{align*}
  I_{32} & = \frac{1}{2}\int_\TT \int_{\TT} \partial^2_\zeta Z_{k,k} \cdot \partial_\zeta Z_{k,k} R'(|Z_{k,k}|)
  \frac{|\partial_\zeta Z_{k,k}|^2}{|Z_{k,k}|}\dd \eta \dd \zeta \\
  & = \frac{1}{4}\int_{\TT}\int_{\TT} \partial_\zeta \big(| \partial_\zeta Z_{k,k}|^2\big)
  R'(|Z_{k,k}|) \frac{|\partial_\zeta Z_{k,k}|^2}{|Z_{k,k}|} \dd \eta \dd \zeta.
\end{align*}
Then through the integration by parts, we find
\begin{align*}
  I_{32} = -\frac{1}{8}\int_{\TT}\int_{\TT}| \partial_\zeta Z_{k,k}|^4
  \frac{Z_{k,k}\cdot \partial_\zeta Z_{k,k}}{|Z_{k,k}|}
  \Big(\frac{R''(|Z_{k,k}|)}{|Z_{k,k}|} - \frac{R'(|Z_{k,k}|)}{|Z_{k,k}|^2}\Big)\dd \eta \dd \zeta.
\end{align*}
Hence, in view of Lemma \ref{lem:G-lemma} and \eqref{eq:Zkk-fact}, \eqref{eq:diff-pointwise-es},
we can deduce that for every $\epsilon \in (0,\frac{1}{3}-\alpha)$,
\begin{align*}
  |I_{32}| & \leq C
  \int_{\TT} \int_{\TT} | \partial_\zeta Z_{k,k}|^5
  \Big(\frac{| R''(|Z_{k,k}|)|}{|Z_{k,k}|} + \frac{|R'(|Z_{k,k}|)|}{|Z_{k,k}|^2}\Big)\dd \eta\dd \zeta \\
  & \leq C \big(\|z_k\|_{H^2} \|F[z_k]\|_{L^\infty}^{\frac{1}{3}}
  + \|z_k\|_{H^2}^{\frac{1}{3}} \big) \int_{\mathbb{T}} \int_{\mathbb{T}} |\partial_\zeta Z_{k,k}|^4
  \Big(\frac{1}{|Z_{k,k}|^{8/3 +\alpha}} + \frac{1}{|Z_{k,k}|^{5/3}} \Big) \dd \eta \dd \zeta \\
  & \leq C \big(\lVert z_k\rVert_{H^2} \lVert F[z_k]\rVert^{\frac{4}{3}-\epsilon}_{L^\infty}
  + \lVert z_k\rVert^{\frac{2}{3}}_{H^2} \lVert F[z_k] \rVert^{1-\epsilon}_{L^\infty}\big) \int_{\TT}\int_{\TT}|\eta|^{\epsilon-1}
  \frac{| \partial_\zeta Z_{k,k}|^4}{|Z_{k,k}|^{\alpha+ 5/3 + \epsilon}} \dd \eta \dd \zeta \\
  & \quad + C \big(\lVert z_k\rVert_{H^2} \lVert F[z_k]\rVert^{\frac{2}{3}}_{L^\infty}
  + \lVert z_k\rVert^{\frac{2}{3}}_{H^2} \lVert F[z_k] \rVert^{\frac{1}{3}}_{L^\infty}\big) \int_{\TT}\int_{\TT}|\eta|^{-\frac{1}{3}} \frac{|\partial_\zeta Z_{k,k}|^4}{|Z_{k,k}|^{4/3}}
  \dd \eta \dd \zeta .
\end{align*}
Now we recall the following crucial result (see \cite[Lemma 7]{GanP21})
that for every $2\pi$-periodic positive function $0< g \in H^2$,
\begin{align}\label{eq:f-fact}
  \int_{\mathbb{T}} \frac{|g'(\mu)|^4}{g(\mu)^\beta} \dd \mu
  \leq C_\beta \|g\|_{H^2}^{4-\beta},\quad \forall \beta\in (1,2],
\end{align}
and in combination with \eqref{eq:Zkk-fact}-\eqref{eq:Zkk-fact2},
we find that for every $\alpha \in (0,\frac{1}{3})$,
\begin{align*}
  \int_{\TT}\int_{\TT}|\eta|^{\epsilon -1} & \frac{| \partial_\zeta Z_{k,k}|^4}{|Z_{k,k}|^{\alpha+ 5/3 +\epsilon}}
   \dd \eta \dd \zeta \leq C \int_{\mathbb{T}} \int_{\TT} |\eta|^{\epsilon-1}
  \frac{|\partial_\zeta z_k(\zeta) - \partial_\zeta z_k(\zeta -\eta)|^4
  + |\partial_\zeta z^{(2)}_k(\zeta -\eta)|^4 }{|Z_{k,k}(\zeta,\eta)|^{\alpha+5/3+\epsilon}} \dd \eta \dd \zeta \\
  & \leq C \|F[z_k]\|_{L^\infty}^{\alpha+\frac{5}{3} +\epsilon} \int_{\TT} \int_{\TT}
  \frac{|\partial z_k(\zeta) - \partial_\zeta z_k(\zeta -\eta)|^4}{|\eta|^{\alpha+8/3}} \dd \eta \dd \zeta \\
  & \quad + C \big(\|z_k\|_{H^2} \|F[z_k]\|_{L^\infty} +1 \big)^{\alpha+\frac{5}{3} +\epsilon}
  \int_{\TT} \int_{\TT} |\eta|^{\epsilon -1}
  \frac{|\partial_\zeta z_k^{(2)}(\zeta -\eta)|^4}{\big(z_k^{(2)}(\zeta -\eta)\big)^{\alpha + 5/3 + \epsilon}} \dd \eta \dd \zeta \\
  & \leq C  \big( (\|z_k\|_{H^2} + 1) \|F[z_k]\|_{L^\infty}  +1 \big)^{\alpha+\frac{5}{3} +\epsilon}
  \Big(\lVert z_k\rVert^{\frac{7}{3}-\alpha-\epsilon}_{H^2}
  + \lVert z_k\rVert^{4}_{H^2}\Big),
\end{align*}
and
\begin{align*}
  \int_{\TT}\int_{\TT} |\eta|^{-\frac{1}{3}} \frac{|\partial_\zeta Z_{k,k}|^4}{|Z_{k,k}|^{4/3}} \dd \eta \dd \zeta
  \leq C \big( (\|z_k\|_{H^2} +1) \|F[z_k]\|_{L^\infty}  +1 \big)^{\frac{4}{3} }
  \Big(\lVert z_k\rVert^{\frac{8}{3}}_{H^2}
  + \lVert z_k\rVert^{4}_{H^2}\Big).
\end{align*}
Hence combining the above estimates leads to
\begin{align*}
  |I_{32}| \leq C \big(\|F[z_k]\|_{L^\infty}^{3+\alpha} +1 \big) \big(\|z_k\|_{H^2}^{\frac{7}{3}} +1 \big)
  \Big(\|z_k\|_{H^2}^2 + \|z_k\|_{H^2}^{\frac{14}{3}} \Big) .
\end{align*}
The term $I_{33}$ can be estimated in a similar manner: by using Lemma \ref{lem:G-lemma}, \eqref{eq:Zkk-fact}-\eqref{eq:diff-pointwise-es} and \eqref{eq:f-fact} again, we find
\begin{align*}
  |I_{33}| & \leq \| z_k\|_{H^2} \Big\|\int_{\TT} |\partial_\zeta Z_{k,k}|^3 \Big(\frac{|R'(Z_{k,k})|}{|Z_{k,k}|}
  + |R''(|Z_{k,k}|)| \Big) \dd \eta \Big\|_{L^2} \\
  & \leq C \|z_k\|_{H^2} \big(\lVert z_k\rVert_{H^2} \lVert F\rVert^{\frac{1}{3}}_{L^\infty}
  + \lVert z_k\rVert^{\frac{2}{3}}_{H^2}\big)
  \Big\|\int_{\TT} |\partial_\zeta Z_{k,k}|^2 \Big( \frac{1}{|Z_{k,k}|^{5/3+\alpha}} + \frac{1}{|Z_{k,k}|^{2/3}} \Big)\dd \eta
  \Big\|_{L^2},
\end{align*}
and for every $\epsilon \in (0,\frac{1}{3}-\alpha)$,
\begin{align*}
  \Big\|\int_{\TT} & \frac{|\partial_\zeta Z_{k,k}|^2}{|Z_{k,k}|^{5/3 +\alpha}}\dd \eta \Big\|_{L^2}
  \leq C \Big\| \int_{\TT} \frac{|\partial_\zeta z_k(\zeta) - \partial_\zeta z_k(\zeta -\eta)|^2 +
  |\partial_\zeta z_k^{(2)}(\zeta -\eta)|^2 }{|Z_{k,k}(\zeta,\eta)|^{5/3 +\alpha}} \dd \eta\Big\|_{L^2} \\
  & \leq C \|F[z_k]\|_{L^\infty}^{\alpha + \frac{5}{3}} \Big\| \int_{\TT}
  \frac{|\partial_\zeta z_k(\zeta) - \partial_\zeta z_k(\zeta -\eta)|^2}{|\eta|^{5/3 +\alpha}} \dd \eta \Big\|_{L^2} \\
  & \quad + C \big((\|z_k\|_{H^2} + 1) \|F[z_k]\|_{L^\infty}  +1 \big)^{\alpha+\frac{5}{3}}
  \Big\|\int_{\TT} |\eta|^{\epsilon -1} \frac{|\partial_\zeta z_k^{(2)}(\zeta -\eta)|^2}
  {\big(z_k^{(2)}(\zeta -\eta)\big)^{\alpha+ 2/3 + \epsilon}} \dd \eta \Big \|_{L^2} \\
  & \leq C \|F[z_k]\|_{L^\infty}^{\alpha + \frac{5}{3}} \int_{\TT}
  \frac{\|\partial_\zeta z_k(\zeta) - \partial_\zeta z_k(\zeta -\eta)\|_{L^4}^2}{|\eta|^{5/3 +\alpha}} \dd \eta \\
  & \quad + C \big((\|z_k\|_{H^2} + 1) \|F[z_k]\|_{L^\infty}  +1 \big)^{\alpha+\frac{5}{3}}
  \Big( \int_{\TT} \frac{|\partial_\zeta z_k^{(2)}(\zeta)|^4}
  {\big(z_k^{(2)}(\zeta)\big)^{2\alpha+ 4/3 + 2\epsilon}} \dd \zeta \Big)^{1/2} \\
  & \leq C \big((\|z_k\|_{H^2} + 1) \|F[z_k]\|_{L^\infty}  +1 \big)^{\alpha+\frac{5}{3}}
  \big(\|z_k\|_{H^2}^2 + \|z_k\|_{H^2}^{\frac{4}{3}-\alpha -\epsilon} \big),
\end{align*}
and
\begin{align*}
  \Big\|\int_{\TT} \frac{|\partial_\zeta Z_{k,k}|^2}{|Z_{k,k}|^{2/3}}\dd \eta \Big\|_{L^2}
  & \leq C \|F[z_k]\|_{L^\infty}^{\frac{2}{3}}  \int_{\TT}
  \frac{\|\partial_\zeta z_k(\zeta) - \partial_\zeta z_k(\zeta-\eta) \|_{L^4}^2}{|\eta|^{2/3}} \dd \eta  \\
  & \quad + C \big((\|z_k\|_{H^2} + 1) \|F[z_k]\|_{L^\infty}  +1 \big)^{\frac{2}{3}}
  \Big\|\int_{\TT} |\eta|^{-\frac{1}{9}} \frac{|\partial_\zeta z_k^{(2)}(\zeta -\eta)|^2}
  {\big(z_k^{(2)}(\zeta -\eta)\big)^{ 5/9 }} \dd \eta \Big\|_{L^2} \\
  & \leq C \big((\|z_k\|_{H^2} + 1) \|F[z_k]\|_{L^\infty}  +1 \big)^{\frac{2}{3}}
  \big( \|z_k\|_{H^2}^2 + \|z_k\|_{H^2}^{\frac{13}{9}} \big).
\end{align*}
Thus, gathering the above estimates yields
\begin{align*}
  |I_{33}| \leq C \big((\|z_k\|_{H^2} + 1) \|F[z_k]\|_{L^\infty}  +1 \big)^{\alpha+2}
  \big( \|z_k\|_{H^2}^{\frac{8}{3}} + \|z_k\|_{H^2}^{\frac{11}{3}} \big),
\end{align*}
and moreover,
\begin{align}
  \Big|\int_{\TT}\partial^2_\zeta z_k(\zeta)\cdot \partial^2_\zeta \mathrm{N}_k(\zeta)\, \dd \zeta\Big|
  \leq C \big(\|F[z_k]\|_{L^\infty}^{3+\alpha} +1 \big) \big(\|z_k\|_{H^2}^{\frac{7}{3}} +1 \big)
  \Big(\|z_k\|_{H^2}^2 + \|z_k\|_{H^2}^{\frac{14}{3}} \Big) .
\end{align}

\subsubsection{Estimation of the $\mathrm{NL}_{j\ne k}$ term in \eqref{eq:z_k-H2}}
First, for $\delta[\mathbf{z}]$ defined by \eqref{def:del-z} we claim that
\begin{equation}
\begin{split}
  \frac{\dd }{\dd t}\Big(\delta[\mathbf{z}]^{-1}\Big)
  \leq C \Big(\sum^N_{k=1} \lVert F[z_k]\rVert^\alpha_{L^\infty}
  + \delta[\mathbf{z}]^{-\alpha}\Big) \lVert \mathbf{z}\rVert^3_{H^2}.
\end{split}
\end{equation}

Indeed, from Lemma \ref{lem:G-lemma} and \eqref{eq:Zkk-fact}, notice that
\begin{align*}
  & \Big\lVert\int_{\TT}\big(\partial_\zeta z_k(\zeta)- \partial_\zeta \overline{z}_j(\zeta-\eta)\big)
  R(|z_k(\zeta)-\overline{z}_j(\zeta-\eta)|)\dd \eta \Big\rVert_{L^\infty} \\
  & \leq C \Big\| \int_{\TT} \big|\partial_\zeta z_k(\zeta) - \partial_\zeta\overline{z}_j(\zeta -\eta)\big|
  \Big(|z_k(\zeta) -\overline{z}_j(\zeta-\eta)|^{-\alpha} + |z_k(\zeta) -\overline{z}_j(\zeta -\eta)| \Big) \dd \eta \Big\|_{L^\infty} \\
  & \leq
  \begin{cases}
    C \lVert \mathbf{z}\rVert_{C^1} \big(\delta[\mathbf{z}]^{-\alpha} + \|\mathbf{z}\|_{L^\infty} \big),\quad & \textrm{for}\;\; j\neq k, \\
    C \|\mathbf{z}\|_{C^1} \big( \|F[z_k]\|_{L^\infty}^\alpha + \|\mathbf{z}\|_{L^\infty} \big), \quad &
    \textrm{for}\;\; j=k,
  \end{cases}
\end{align*}
and the same upper bound holds for estimating
$\lVert\int_{\TT}\big(\partial_\zeta z_k(\zeta)- \partial_\zeta z_j(\zeta-\eta)\big)
  R(|z_k(\zeta)- z_j(\zeta-\eta)|)\dd \eta \rVert_{L^\infty}$,
thus, recalling $\mathrm{NL}_k$ given by \eqref{eq:NL-k-GP}, we obtain that
\begin{align*}
  \lVert \mathrm{NL}_k\rVert_{L^\infty} \leq C \lVert \mathbf{z}\rVert_{C^1} \big(\lVert F[z_k]\rVert^\alpha_{L^\infty}
  + \delta[\mathbf{z}]^{-\alpha} + \|\mathbf{z}\|_{L^\infty} \big);
\end{align*}
in addition, via integrating by parts in the expression of $\lambda_k$ (see \eqref{eq:lambda-def}), we have
\begin{align*}
  \lVert \lambda_k(\zeta) \partial_\zeta z_k(\zeta)\rVert_{L^\infty}
  & \leq C \lVert \mathrm{NL}_k\rVert_{L^\infty} \lVert z_k\rVert_{H^2} \lVert z_k\rVert_{C^1} \\
  & \leq C \big(\lVert F[z_k]\rVert^\alpha_{L^\infty}
  + \delta[\mathbf{z}]^{-\alpha} + \|\mathbf{z}\|_{H^2} \big)
  \lVert \mathbf{z}\rVert^3_{H^2};
\end{align*}
hence, from the contour equation \eqref{eq:contour} we find that
\begin{align*}
  \lVert \partial_t z_k\rVert_{L^\infty} \leq
  C \big(\lVert F[z_k] \rVert^\alpha_{L^\infty}
  + \delta[\mathbf{z}]^{-\alpha} + \|\mathbf{z}\|_{H^2}\big)
  \big(\lVert \mathbf{z}\rVert_{H^2} + \|\mathbf{z}\|_{H^2}^3 \big),
\end{align*}
and thus,
\begin{align*}
  \frac{\dd }{\dd t}\Big(\delta[\mathbf{z}]^{-1}\Big) = \delta[\mathbf{\mathbf{z}}]^{-2} \frac{\dd }{\dd t}\delta[\mathbf{z}](t)
  \leq C  \big(\lVert F[z_k] \rVert^\alpha_{L^\infty}
  + \delta[\mathbf{z}]^{-\alpha} + \|\mathbf{z}\|_{H^2}\big)
  \big(\lVert \mathbf{z}\rVert_{H^2} + \|\mathbf{z}\|_{H^2}^3 \big).
\end{align*}

Now we start the estimation of the term $\mathrm{NL}_{j\ne k}$ in \eqref{eq:z_k-H2}.
For every $j\ne k$, we split it as
\begin{align}\label{eq:NLjnek}
  \mathrm{NL}_{j\ne k}(\zeta,t) = \mathrm{O}_{j\ne k}(\zeta,t) + \mathrm{N}_{j\ne k}(\zeta,t),
\end{align}
where
\begin{align*}
  \mathrm{O}_{j\ne k} (\zeta,t) \triangleq \sum_{j\neq k} a_j \int_{\mathbb{T}}(\partial_\zeta z_k(\zeta,t)
  - \partial_\zeta z_j(\zeta - \eta, t)) R(| z_k(\zeta,t)  - z_j(\zeta - \eta, t)|)  \dd\eta,
\end{align*}
\begin{align*}
  \mathrm{N}_{j\ne k} (\zeta,t) \triangleq \sum_{j\neq k} a_j \int_{\mathbb{T}} (\partial_\zeta z_k(\zeta,t)
  - \partial_\zeta \overline{z}_j(\zeta - \eta, t)) R(| z_k(\zeta,t)  - \overline{z}_j(\zeta - \eta, t)|)  \dd\eta.
\end{align*}
We only consider the estimation of the term $N_{j\ne k}$, since the treating of $O_{j\ne k}$ is easier and can be implemented in the same way.

By using the notation \eqref{def:Zkj0}, we have
\begin{align*}
  \int_{\TT} \partial_\zeta^2 z_k(\zeta) \cdot \partial_\zeta^2 N_{j\ne k } \dd \zeta & =
  \sum_{j\neq k} a_j \int_{\TT} \int_{\TT} \partial_\zeta^2 z_k(\zeta) \cdot
  \partial_\zeta^2 \big(\partial_\zeta Z_{k,j}(\zeta,\eta) \, R(|Z_{k,j}|)\big)
  \dd \eta \dd \zeta \\
  &= \sum_{j\neq k} a_j\int_{\TT} \int_{\TT} \partial_\zeta^2 z_k(\zeta)
  \cdot \partial^3_\zeta Z_{k,j}(\zeta,\eta)\, R(|Z_{k,j}|) \dd \eta\dd \zeta \\
  & \quad + \sum_{j\neq k} a_j \int_{\TT} \int_{\TT} \partial_\zeta^2 z_k(\zeta) \cdot \partial^2_\zeta Z_{k,j}(\zeta,\eta)
  \frac{Z_{k,j} \cdot \partial_\zeta Z_{k,j}}{|Z_{k,j}(\zeta,\eta)|} R'(|Z_{k,j}|) \dd \eta \dd \zeta \\
  & \quad + \sum_{j\neq k} a_j\int_{\TT} \int_{\TT} \partial_\zeta^2 z_k(\zeta) \cdot \partial_\zeta Z_{k,j}(\zeta,\eta)\,
  \partial^2_\zeta R(|Z_{k,j}|) \dd \eta \dd \zeta \\
  & \triangleq J_1 + J_2 + J_3.
\end{align*}

Through the integration by parts, observe that
\begin{align*}
  J_1 &  = \sum_{j\neq k} a_j \int_{\TT} \int_{\TT} \partial^2_\zeta z_k(\zeta)
  \cdot \big(\partial^3_\zeta z_k(\zeta,t)-\partial^3_\zeta \overline{z}_j(\zeta-\eta,t)\big)
  R(|Z_{k,j}|) \dd \eta\dd \zeta \\
  & = \frac{1}{2} \sum_{j\neq k} a_j \int_{\TT} \int_{\TT} \partial_\zeta (|\partial^2_\zeta z_k(\zeta)|^2)
  R(|Z_{k,j}|) \dd \eta\dd \zeta  \\
  & \quad + \sum_{j\neq k} a_j \int_{\TT} \int_{\TT} \partial^2_\zeta z_k(\zeta)\cdot
  \partial_{\eta}\partial^2_\zeta \overline{z}_j(\zeta-\eta) R(|Z_{k,j}|)
  \dd \eta\dd \zeta \\
  & = -\frac{1 }{2} \sum_{j\neq k} a_j \int_{\TT} \int_{\TT}
  |\partial^2_\zeta z_k(\zeta)|^2 \frac{Z_{k,j}\cdot \partial_\zeta Z_{k,j}}{|Z_{k,j}|}
  R'(|Z_{k,j}|) \dd \eta\dd \zeta \\
  & \quad - \sum_{j\neq k} a_j \int_{\TT} \int_{\TT} \partial_\zeta^2 z_k(\zeta)
  \cdot \partial_\zeta^2 \overline{z}_j(\zeta -\eta) \frac{Z_{k,j}\cdot \partial_\eta Z_{k,j}}{|Z_{k,j}|} R'(|Z_{k,j}|)
  \dd \eta \dd \zeta.
\end{align*}
Thus, by using \eqref{es:R-deri}, we obtain that
\begin{align*}
  |J_1| + |J_2|  \leq C \lVert z_k\rVert^2_{H^2} \lVert \mathbf{z}\rVert_{C^1} \big(1+ \delta[\mathbf{z}]^{-2-\alpha} \big) .
\end{align*}
For $J_3$, we see that
\begin{align*}
  |J_3| \leq  & \sum_{j\neq k} a_j \int_{\TT} \int_{\TT} |\partial_\zeta^2 z_k(\zeta)|\,
  |\partial_\zeta Z_{k,j}|^3\, |R''(|Z_{k,j}|)|
  \dd \eta \dd \zeta \\
  & +  2 \sum_{j\neq k} a_j \int_{\TT} \int_{\TT} |\partial_\zeta^2 z_k(\zeta)| \, | R'(|Z_{k,j}|)|
  \frac{|\partial_\zeta Z_{k,j}|^3}{|Z_{k,j}|} \dd \eta \dd \zeta \\
  & + \sum_{j\neq k} a_j \int_{\TT} \int_{\TT} |\partial_\zeta^2 z_k(\zeta)| \, |R'(|Z_{k,j}|)|
  \frac{|\partial_\zeta Z_{k,j}|^2 |\partial^2_\zeta Z_{k,j}|}{|Z_{k,j}|}\dd \eta \dd \zeta.
\end{align*}
Hence, by virtue of Lemma \ref{lem:G-lemma}, we have
\begin{align*}
  |J_3| & \leq C  \Big(\lVert \mathbf{z}\rVert_{C^1}^3
  \| \mathbf{z} \|_{H^2} + \|\mathbf{z}\|_{C^1}^2 \|\mathbf{z}\|_{H^2}^2 \Big) \,
  \sum_{j\neq k} \sup_{(\zeta,\eta)\in \mathbb{T}^2}
  \Big(\frac{1}{|Z_{k,j}|^{2+\alpha}} + \frac{1}{|Z_{k,j}|} \Big) \\
  & \leq C \lVert \mathbf{z}\rVert_{H^2}^4 \big( 1 + \delta[\mathbf{z}]^{-2-\alpha}\big).
\end{align*}
Collecting the above estimates gives
\begin{align*}
  \Big|\int_{\TT} \partial_\zeta^2 z_k(\zeta) \cdot \partial_\zeta^2 N_{j\ne k } \dd \zeta \Big|
  \leq C \big(\|\mathbf{z}\|_{H^2}^3 +  \lVert \mathbf{z}\rVert_{H^2}^4 \big)
  \big( 1 + \delta[\mathbf{z}]^{-2-\alpha}\big) .
\end{align*}

\subsubsection{Estimation of the tangential term in \eqref{eq:z_k-H2}}\label{subsec:es-tang}
Now, we deal with the tangential term in \eqref{eq:z_k-H2}.
We have
\begin{align*}
  \int_{\TT} \partial_\zeta^2 z_k(\zeta) \cdot \partial^2_\zeta \big(\lambda_k(\zeta) \partial_\zeta z_k(\zeta)\big)\dd \zeta
  & = \int_{\TT} \partial_\zeta^2 z_k(\zeta) \cdot \partial^3_\zeta z_k(\zeta)\lambda_k(\zeta) \dd \zeta
  + 2\int_{\TT} \partial_\zeta^2 z_k(\zeta) \cdot
  \partial^2_\zeta z_k(\zeta)\partial_\zeta \lambda_k(\zeta) \dd \zeta \\
  & \quad + \int_{\TT} \partial_\zeta^2 z_k(\zeta)\cdot\partial_{\zeta} z_k(\zeta) \,\partial^2_\zeta \lambda_k(\zeta) \dd \zeta \\
  & = \frac{3}{2} \int_{\TT} |\partial_\zeta^2 z_k(\zeta)|^2 \partial_\zeta \lambda_k(\zeta)\dd \zeta
  + \int_{\TT}\partial_\zeta^2 z_k(\zeta)\cdot
  \partial_\zeta z_k(\zeta) \, \partial^2_\zeta \lambda_k(\zeta)\dd \zeta .
\end{align*}
Since $A_k(t)=|\partial_\zeta z_k|^2$ depends only on $t$, we have
$\partial^2_\zeta z_k(\zeta) \cdot \partial_\zeta z_k(\zeta)
  = \frac{1}{2}\partial_\zeta A_k(t) = 0$,
and thus
\begin{align}\label{eq:tang-term1}
  \int_{\TT} \partial_\zeta^2 z_k(\zeta) \cdot \partial^2_\zeta \big(\lambda_k(\zeta) \partial_\zeta z_k(\zeta) \big)\dd \zeta
  =  \frac{3}{2} \int_{\TT} |\partial_\zeta^2 z_k(\zeta)|^2 \partial_\zeta \lambda_k(\zeta)\dd \zeta
  \leq \frac{3}{2} \|z_k\|_{H^2}^2 \|\partial_\zeta\lambda_k(\zeta)\|_{L^\infty} .
\end{align}
From \eqref{eq:lambda-def}, we see that
\begin{align*}
  \partial_\zeta \lambda_k(\zeta)  = -\frac{\partial_\zeta z_k(\zeta) \cdot \partial_\zeta \mathrm{NL}_k(\zeta)}{A_k(t)}
  + \frac{1}{2\pi} \int_{\TT}\frac{\partial_\zeta z_k(\eta)\cdot \partial_\zeta \mathrm{NL}_k(\eta)}{A_k(t)}\dd \eta ,
\end{align*}
with $A_k(t) = |\partial_\zeta z_k(\zeta,t)|^2$ for every $k\in\{1,\cdots, N\}$.

It remains to estimate $\frac{1}{A_k(t)} \|\partial_\zeta z_k(\zeta) \cdot \partial_\zeta \mathrm{NL}_k(\zeta)\|_{L^\infty}$.
We only consider those conjugate terms of $\mathrm{NL}_k$, as the other terms in $\mathrm{NL}_k$ are similar and with more cancellations.
Recalling that $\mathrm{NL}_{j=k}(\zeta)$ is given by \eqref{eq:NLj=k},  the conjugate term in $\partial_\zeta \mathrm{NL}_{j=k}(\zeta)$
is decomposed as follows
\begin{align*}
  \partial_\zeta \int_{\TT}\partial_\zeta Z_{k,k}(\zeta,\eta,t) R(|Z_{k,k}|)\dd \eta
  = K_1(\zeta) + K_2(\zeta),
\end{align*}
where
\begin{align*}
  K_1 \triangleq \int_{\TT}\partial_\zeta Z_{k,k}\, R'(|Z_{k,k}|)
  \frac{Z_{k,k}\cdot \partial_\zeta Z_{k,k}}{|Z_{k,k}|}\dd \eta, \quad
  K_2 \triangleq \int_{\TT} \partial^2_\zeta Z_{k,k} \,
  R(|Z_{k,k}|)\dd \eta,
\end{align*}
with $Z_{k,k}$ defined by \eqref{def:Zkj0}.
Using \eqref{eq:Zkk-fact}, \eqref{eq:diff-pointwise-es},
and noting that 
\[|A(t)|^{-\frac{1}{2}} = |\partial_\zeta z_k(\zeta,t)|^{-1} \leq \|F[z_k]\|_{L^\infty},\]
we get
\begin{align*}
  \frac{1}{A_k(t)}|\partial_\zeta z_k(\zeta)\cdot K_1 (\zeta)| & \leq C
  \frac{1}{A(t)^{1/2}} \big(\lVert z_k\rVert_{H^2} \lVert F[z_k]\rVert^{\frac{1}{3}}_{L^\infty}
  + \lVert z_k\rVert^{\frac{2}{3}}_{H^2}\big)^2 \int_{\mathbb{T}} \Big(\frac{1}{|Z_{k,k}|^{\alpha + 1/3}}
  + \frac{1}{|Z_{k,k}|^{1/3}}\Big) \dd \eta \\
  & \leq C \|F[z_k]\|_{L^\infty} \|z_k\|_{H^2}^{\frac{4}{3}} \big(\|z_k\|_{H^2}^{\frac{2}{3}} +1 \big)
  \big( \lVert F[z_k]\rVert_{L^\infty}^{1+\alpha} + 1\big) .
\end{align*}
Notice that $\partial_\zeta z_k(\zeta)\cdot \partial^2_\zeta z_k(\zeta) = \partial_\zeta A_k(t)=0$,
we apply \eqref{es:R-deri} and \eqref{eq:Zkk-fact} to deduce that
\begin{align*}
  \frac{1}{A_k(t)} |\partial_\zeta z_k(\zeta) \cdot K_2(\zeta)|
  & = \frac{1}{A_k(t)} \Big|\int_{\TT}\partial_\zeta z_k(\zeta)
  \cdot \partial^2_\zeta \overline{z}_k(\zeta-\eta) R(|Z_{k,k}(\zeta,\eta,t)|)\dd \eta\Big|\\
  &\leq C \frac{1}{A(t)^{1/2}} \lVert \mathbf{z}\rVert_{H^2}
  \int_{\mathbb{T}} \Big( \frac{1}{|Z_{k,k}|^\alpha} + |Z_{k,k}| \Big) \dd \eta \\
  & \leq C \|F[z_k]\|_{L^\infty} \|\mathbf{z}\|_{H^2}  \big(\lVert F[z_k]\rVert^\alpha_{L^\infty} + \|z_k\|_{L^\infty} \big).
\end{align*}
For $\mathrm{NL}_{j\ne k}$ given by \eqref{eq:NLjnek}, the derivative of the conjugate term satisfies
\begin{align*}
  \partial_\zeta \int_{\TT}\partial_\zeta Z_{k,j}(\zeta,\eta,t) R(|Z_{k,j}|)\dd \eta
  = K_3(\zeta) + K_4(\zeta)
\end{align*}
where
\begin{align*}
  K_3 \triangleq \int_{\TT} \partial_\zeta Z_{k,j}\,R'(|Z_{k,j}|)
  \frac{Z_{k,j}\cdot \partial_\zeta Z_{k,j}}{|Z_{k,j}|}\dd \eta, \quad
  K_4 \triangleq \int_{\TT} \partial^2_\zeta Z_{k,j}\, R(|Z_{k,j}|)\dd \eta.
\end{align*}
By virtue of Lemma \ref{lem:G-lemma}, we immediately have
\begin{align*}
  \frac{1}{A_k(t)} |\partial_\zeta z_k(\zeta) \cdot K_3(\zeta)| \leq C \|F[z_k]\|_{L^\infty} \lVert \mathbf{z}\rVert_{H^2}^2
  \big(\delta[\mathbf{z}]^{-1-\alpha} + 1 \big),
\end{align*}
and
\begin{align*}
  \frac{1}{A_k(t)} |\partial_\zeta z_k(\zeta)\cdot K_4(\zeta)|
  & = \frac{1}{A_k(t)} \Big| \int_{\TT}\partial_\zeta z_k(\zeta)\cdot \partial^2_\zeta \overline{z}_j(\zeta-\eta)R(|Z_{k,j}|)\dd \eta \Big| \\
  &\leq C \|F[z_k]\|_{L^\infty} \|\mathbf{z} \|_{H^2} \big(\delta[\mathbf{z}]^{-\alpha}
  + \|\mathbf{z}\|_{L^\infty} \big).
\end{align*}
Gathering the above estimates, we find
\begin{align}\label{eq:par-lambda}
  \|\partial_\zeta \lambda_k\|_{L^\infty} \leq C \|F[z_k]\|_{L^\infty}
  \|\mathbf{z}\|_{H^2} \big(\|\mathbf{z}\|_{H^2} +1 \big)
  \big(\|F[z_k]\|_{L^\infty}^{1+\alpha} + \delta[\mathbf{z}]^{-1-\alpha} + \|\mathbf{z}\|_{L^\infty}
  + 1 \big).
\end{align}

Hence, the tangential term in \eqref{eq:z_k-H2} can be bounded by combining \eqref{eq:tang-term1} with \eqref{eq:par-lambda}.

\subsubsection{Estimation of $\lVert F[z_k]\rVert_{L^\infty}$}

Recalling that $F[z_k]$ is given by \eqref{def:F[z_k]}, we infer that
\begin{align*}
  \partial_t F[z_k] = -\frac{|\eta|(z_k(\zeta) - z_k(\zeta-\eta)) \cdot
  (\partial_t z_k(\zeta) - \partial_t z_k(\zeta-\eta))}{|z_k(\zeta)- z_k(\zeta-\eta)|^3}.
\end{align*}
From the contour equation \eqref{eq:main-eq-GP}-\eqref{eq:lambda-def}, and using the notations \eqref{def:Zkj0},
we can write 
\[\partial_t z_k(\zeta) - \partial_t z_k(\zeta - \eta) = \sum_{i=1}^7 N_j,\]
where
\begin{align*}
  N_1 & \triangleq \sum_{j=1}^N a_j \big(\partial_\zeta z_k(\zeta) - \partial_\zeta z_k(\zeta-\eta)\big)
  \int_{\TT}R\big(|\mathbf{Z}_{k,j}(\zeta,\mu)| \big) \dd \mu, \\
  N_2 & \triangleq - \sum_{j=1}^N a_j \int_{\TT} \big(\partial_\zeta z_k(\zeta-\mu) - \partial_\zeta z_k(\zeta-\mu-\eta)\big)
  R\big(|\mathbf{Z}_{k,j}(\zeta,\mu)|\big) \dd \mu, \\
  N_3 & \triangleq \sum_{j=1}^N a_j \int_{\TT} \big(\partial_\zeta z_k(\zeta-\eta) - \partial_\zeta z_k(\zeta-\eta-\mu)\big)
  \Big(R\big(|\mathbf{Z}_{k,j}(\zeta,\mu)|\big) - R\big(|\mathbf{Z}_{k,j}(\zeta-\eta,\mu)|\big)\Big) \dd \mu,
\end{align*}
and
\begin{align*}
  N_4 & \triangleq \sum_{j=1}^N a_j \big(\partial_\zeta z_k(\zeta) - \partial_\zeta z_k(\zeta-\eta)\big)
  \int_{\TT}R\big(|Z_{k,j}(\zeta,\mu)| \big) \dd \mu, \\
  N_5 & \triangleq -\sum_{j=1}^N a_j \int_{\TT} \big(\partial_\zeta \overline{z}_k (\zeta - \mu)
  - \partial_\zeta \overline{z}_k(\zeta -\mu-\eta)\big)  R\big(|Z_{k,j}(\zeta,\mu)|\big) \dd \mu, \\
  N_6 & \triangleq \sum_{j=1}^N a_j \int_{\TT} \partial_\zeta Z_{k,k}(\zeta-\eta,\mu)
  \Big(R\big(|Z_{k,j}(\zeta,\mu)|\big) - R\big(|Z_{k,j}(\zeta-\eta,\mu)|\big)\Big) \dd \mu,
\end{align*}
and
\begin{align*}
  N_7 \triangleq \lambda_k(\zeta) \partial_\zeta z_k(\zeta)-\lambda_k(\zeta-\eta)\partial_\zeta z_k(\zeta-\eta).
\end{align*}
Thus, we have
\begin{align*}
  \partial_t F[z_k] = - \sum_{i=1}^7 \frac{|\eta|\big(z_k(\zeta) - z_k(\zeta-\eta)\big)}{|z_k(\zeta) - z_k(\zeta-\eta)|^3}
  \cdot  N_i.
\end{align*}
Since the estimation of the terms involving $N_i$ ($i=1,2,3$) is similar to (and easier than)
that of the terms involving
$N_{i+3}$ ($i=1,2,3$), below we mainly focus on the terms containing $N_4$-$N_7$.
For $N_4$, we have
\begin{align*}
  - & \frac{|\eta|\big(z_k(\zeta) - z_k(\zeta-\eta)\big)}{|z_k(\zeta) - z_k(\zeta-\eta)|^3}
  \cdot  N_4 \\
  & \quad= \frac{|\eta|\big(z_k(\zeta) - z_k(\zeta - \eta)\big) \cdot \big(\partial_\zeta z_k(\zeta)
  - \partial_\zeta z_k(\zeta-\eta)\big)}{|z_k(\zeta) - z_k(\zeta - \eta)|^3}
  \sum_{j=1}^N \theta_j \int_{\TT} R\big(|Z_{k,j}(\zeta,\mu)|\big) \dd \mu \\
  & \quad = N_{41} + N_{42},
\end{align*}
where
\begin{align*}
  N_{41} \triangleq & \frac{|\eta|\big(z_k(\zeta) - z_k(\zeta-\eta) - \eta \partial_\zeta z_k(\zeta)\big)\cdot
  \big(\partial_\zeta z_k(\zeta) - \partial_\zeta z_k(\zeta - \eta)\big)}{|z_k(\zeta) - z_k(\zeta-\eta)|^3}
  \sum_{j=1}^N a_j \int_{\TT} R\big(|Z_{k,j}(\zeta,\mu)|\big) \dd \mu,
\end{align*}
\begin{align*}
  N_{42} \triangleq \frac{|\eta|\big(\eta\partial_\zeta z_k(\zeta)\big)\cdot
  \big(\partial_\zeta z_k(\zeta) - \partial_\zeta z_k(\zeta-\eta)\big)}{|z_k(\zeta)- z_k(\zeta-\eta)|^3}
  \sum_{j=1}^N a_j \int_{\TT} R\big(|Z_{k,j}(\zeta,\mu)|\big) \dd \mu.
\end{align*}
Noting that
\begin{align*}
  \Big|\sum_{j=1}^N a_j \int_{\mathbb{T}} R\big(|Z_{k,j}(\zeta,\mu)|\big)\dd \mu\Big|
  & \leq C \sum_{j=1}^N\int_{\mathbb{T}} \Big(\frac{1}{|Z_{k,j}(\zeta,\mu)|^\alpha} +
  |Z_{k,j}(\zeta,\mu)|\Big) \dd \mu \\
  & \leq C \big(\|F[z_k]\|_{L^\infty}^\alpha + \delta[\mathbf{z}]^{-\alpha}
  + \|\mathbf{z}\|_{L^\infty}  \big),
\end{align*}
we can bound $N_{41}$ as follows:
\begin{align*}
  |N_{41}| & \leq C \lVert F[z_k]\rVert^3_{L^\infty}
  \lVert \partial_\zeta z_k\rVert^2_{C^{\frac{1}{2}}}
  \Big|\sum_{j=1}^N \theta_j \int_{\mathbb{T}} R\big(|Z_{k,j}(\zeta,\mu)|\big)\dd \mu\Big| \\
  & \leq C \lVert F[z_k]\rVert^{3}_{L^\infty} \lVert z_k\rVert^2_{H^2}
  \big(\|F[z_k]\|_{L^\infty}^\alpha + \delta[\mathbf{z}]^{-\alpha}
  + \|\mathbf{z}\|_{L^\infty}  \big) .
\end{align*}
Using the fact
$ |\partial_\zeta z_k(\zeta)|^2 = A_k(t)= \frac{1}{2}|\partial_\zeta z_k(\zeta)|^2
  + \frac{1}{2} |\partial_\zeta z_k(\zeta-\eta)|^2$,
we can establish the following identity
\begin{align}\label{eq:id-diff-1}
  \partial_\zeta z_k(\zeta) \cdot (\partial_\zeta z_k(\zeta)- \partial_\zeta z_k(\zeta-\eta))
  = &\, \frac{1}{2} \Big(|\partial_\zeta z_k(\zeta)|^2-2\partial_\zeta z_k(\zeta)\cdot
  \partial_\zeta z_k(\zeta-\eta) + |\partial_\zeta z_k(\zeta-\eta)|^2\Big) \nonumber \\
  = &\, \frac{1}{2} |\partial_\zeta z_k(\zeta)-\partial_\zeta z_k(\zeta-\eta)|^2.
\end{align}
Then the term $N_{42}$ can be estimated as follows
\begin{align*}
  |N_{42}| & \leq \Big|\frac{|\eta|\big(\eta\partial_\zeta z_k(\zeta)\big) \cdot
  \big(\partial_\zeta z_k(\zeta) - \partial_\zeta z_k(\zeta-\eta)\big)}{|z_k(\zeta) - z_k(\zeta-\eta)|^3}
  \sum_{j=1}^N a_j \int_{\TT} R\big(|Z_{k,j}(\zeta,\mu)|\big) \dd \mu \Big| \\
  & \leq C  \lVert F[z_k]\rVert^3_{L^\infty}
  \big(\|F[z_k]\|_{L^\infty}^\alpha + \delta[\mathbf{z}]^{-\alpha}
  + \|\mathbf{z}\|_{L^\infty}  \big) \frac{|\partial_\zeta z_k(\zeta) - \partial_\zeta z_k(\zeta - \eta)|^2}{|\eta|} \\
  & \leq C \lVert F[z_k]\rVert^3_{L^\infty} \lVert \partial_\zeta z_k\rVert^2_{C^{\frac{1}{2}}} \big(\|F[z_k]\|_{L^\infty}^\alpha + \delta[\mathbf{z}]^{-\alpha}
  + \|\mathbf{z}\|_{L^\infty}  \big).
\end{align*}
Now we move to the term involving $N_5$.
Direct computation leads to
\begin{align*}
  \Big|& \frac{|\eta| \big(z_k(\zeta) - z_k(\zeta-\eta)\big) \cdot N_5}{|z_k(\zeta) - z_k(\zeta-\eta)|^3} \Big|
  \leq C \lVert F[z_k] \rVert^2_{L^\infty}
  \sum_{j=1}^N  \int_0^1 \dd s \int_{\TT} |\partial^2_\zeta \overline{z}_k(\zeta-\mu- s \eta)|\,
  \big|R \big(|Z_{k,j}(\zeta,\mu)|\big)\big| \dd \mu \\
  &\qquad \leq C \lVert F[z_k]\rVert^2_{L^\infty} \sum_{j=1}^N \int_0^1 \dd s
  \int_{\mathbb{T}} |\partial^2_\zeta \overline{z}_k(\zeta-\mu - s \eta)|\, \Big(\frac{1}{|Z_{k,j}(\zeta,\mu)|^\alpha}
  + |Z_{k,j}(\zeta,\mu)|\Big) \dd \mu  \\
  &\qquad \leq C \lVert F[z_k]\rVert^2_{L^\infty} \lVert z_k\rVert_{H^2}
  \big( \|F[z_k]\|_{L^\infty}^\alpha + \delta[\mathbf{z}]^{-\alpha} + \|\mathbf{z}\|_{L^\infty} \big).
\end{align*}
For the term involving $N_6$,
using the fact that
\begin{align*}
  R\big(|Z_{k,j}(\zeta,\mu)|\big) - R& \big(|Z_{k,j}(\zeta-\eta,\mu)|\big)
  = - \eta \int_0^1 \frac{\partial }{\partial \zeta} R\big(|Z_{k,j}(\zeta-s\eta,\mu)|\big) \dd s \\
  & = - \eta \int_0^1 R'\big(|Z_{k,j}(\zeta-s\eta,\mu)|\big) \frac{Z_{k,j}(\zeta-s\eta,\mu)
  \cdot \partial_\zeta Z_{k,j}(\zeta-s\eta,\mu)}{|Z_{k,j}(\zeta-s\eta,\mu)|} \dd s ,
\end{align*}
together with \eqref{eq:Zkk-fact}, \eqref{eq:diff-pointwise-es}, we deduce that
\begin{align*}
  |N_6| & \leq C |\eta|
  \sum_{j=1}^N \int_0^1 \dd s \int_{\TT}|\partial_\zeta Z_{k,k}(\zeta-\eta,\mu)|
  \, \big|R'(|Z_{k,j}(\zeta -s\eta,\mu)|) \big|  |\partial_\zeta Z_{k,j}(\zeta -s\eta,\mu)|\dd \mu \\
  & \leq C |\eta| \|\mathbf{z}\|_{C^1}
  \int_0^1 \dd s \int_{\TT}
  \, \Big(\frac{1}{|Z_{k,k}(\zeta-s\eta,\mu)|^{1+\alpha}} + \frac{1}{|Z_{k,k}(\zeta-s\eta,\mu)|} \Big)
  |\partial_\zeta Z_{k,k}(\zeta-s \eta,\mu)|  \dd \mu \\
  & \quad + C |\eta| \|\mathbf{z}\|_{C^1}^2 \sum_{j\neq k} \int_0^1 \dd s \int_{\TT}
  \, \Big(\frac{1}{|Z_{k,j}(\zeta-s\eta,\mu)|^{1+\alpha}} + \frac{1}{|Z_{k,j}(\zeta-s\eta,\mu)|} \Big) \dd \mu \\
  &\leq C |\eta| \lVert \mathbf{z}\rVert_{C^1}  \big(\lVert z_k\rVert_{H^2} + 1 \big)
  \Big(\lVert F[z_k]\rVert_{L^\infty}^{1+\alpha}
  + \delta[\mathbf{z}]^{-1-\alpha} + 1 \Big).
\end{align*}
Hence, we have
\begin{align*}
  \Big|\frac{\eta \big(z_k(\zeta) - z_k(\zeta-\eta)\big)\cdot N_6}{|z_k(\zeta)-z_k(\zeta-\eta)|^3} \Big|
  \leq C \lVert \mathbf{z} \rVert_{H^2} \big(\lVert z_k\rVert_{H^2}
  + 1 \big)\lVert F[z_k]\rVert^2_{L^\infty}
  \Big(\lVert F[z_k]\rVert_{L^\infty}^{1+\alpha}
  + \delta[\mathbf{z}]^{-1-\alpha} + 1 \Big).
\end{align*}

Finally, we deal with the term involving $N_7$, which is decomposed as
\begin{align*}
  & \frac{|\eta|\big(z_k(\zeta)-z_k(\zeta-\eta)\big)\cdot N_7}{|z_k(\zeta)-z_k(\zeta-\eta)|^3}  \\
  & = \frac{|\eta|\big(z_k(\zeta)-z_k(\zeta-\eta)\big)}{|z_k(\zeta)-z_k(\zeta-\eta)|^3}
  \cdot \Big(\lambda_k(\zeta)\big(\partial_\zeta z_k(\zeta) - \partial_\zeta z_k(\zeta-\eta)\big)
  + \big(\lambda_k(\zeta)-\lambda_k(\zeta - \eta)\big) \partial_\zeta z_k(\zeta) \Big) \\
  & = \frac{|\eta|\big(z_k(\zeta) - z_k(\zeta-\eta) - \eta \partial_\zeta z_k(\zeta)\big)}{|z_k(\zeta) - z_k(\zeta - \eta)|^3}
  \cdot \big(\partial_\zeta z_k(\zeta) - \partial_\zeta z_k(\zeta - \eta)\big) \lambda_k(\zeta) \\
  & \quad +  \frac{|\eta|\eta\,}{|z_k(\zeta) - z_k(\zeta-\eta)|^3}
  \partial_\zeta z_k(\zeta) \cdot \big(\partial_\zeta z_k(\zeta) - \partial_\zeta z_k(\zeta - \eta)\big) \lambda_k(\zeta) \\
  & \quad + \frac{|\eta|\big(z_k(\zeta)-z_k(\zeta-\eta)\big)}{|z_k(\zeta)-z_k(\zeta-\eta)|^3}\cdot
  \partial_\zeta z_k(\zeta) \big(\lambda_k(\zeta)-\lambda_k(\zeta - \eta)\big) \\
  & \triangleq N_{71} + N_{72} + N_{73}.
\end{align*}

For $N_{71}$, it is obvious to see that
\begin{align*}
  |N_{71}| \leq C \lVert F[z_k]\rVert^3_{L^\infty} \lVert \partial_\zeta z_k\rVert^2_{C^{\frac{1}{2}}} |\lambda_k(\zeta)|
  \leq C \lVert F[z_k]\rVert^3_{L^\infty} \lVert z_k\rVert^2_{H^2} \|\lambda_k\|_{L^\infty}.
\end{align*}
For $N_{72}$, in view of \eqref{eq:id-diff-1}, it yields
\begin{align*}
  |N_{72}|\leq C \lVert F[z_k]\rVert^3_{L^\infty} \lVert \partial_\zeta z_k\rVert^2_{C^{\frac{1}{2}}} \|\lambda_k\|_{L^\infty}.
\end{align*}
In addition, from \eqref{eq:NL-k-GP}-\eqref{eq:lambda-def} we have the following bound:
\begin{align*}
  \|\lambda_k\|_{L^\infty} \leq \frac{C}{A(t)^{1/2}} \|z_k\|_{C^1}
  \int_{-\pi}^\pi |\partial_\zeta \mathrm{NL}_k(\zeta)| \dd \zeta.
\end{align*}
For the estimation of $\mathrm{NL}_k = \mathrm{NL}_{j=k} + \mathrm{NL}_{j\neq k}$,
we only treat the conjugate terms since the other terms can be handled similarly.
For $\mathrm{NL}_{j=k}$, using \eqref{eq:Zkk-fact} and \eqref{eq:diff-pointwise-es}, we have
\begin{align*}
  \int_{\TT}|\partial_\zeta \mathrm{NL}_{j=k}| \dd \zeta
  \leq & C \int_{\TT} \int_{\TT} |\partial^2_\zeta Z_{k,k}(\zeta,\mu)| |R(|Z_{k,k}(\zeta,\mu)|) |\dd \mu \dd \zeta \\
  & + C \int_{\TT} \int_{\TT} |\partial_\zeta Z_{k,k}(\zeta,\mu)|^2 |R'(|Z_{k,k}(\zeta,\mu)|)|\dd \mu \dd \zeta \\
  \leq & C \big(\lVert F[z_k]\rVert^\alpha_{L^\infty} + \|z_k\|_{L^\infty} \big) \lVert z_k\rVert_{H^2}
  + C \big(\lVert z_k\rVert_{H^2}  + \lVert z_k\rVert^{\frac{2}{3}}_{H^2}\big)^2
  \big(\lVert F[z_k]\rVert^{1 +\alpha}_{L^\infty} + 1 \big) .
\end{align*}
For $\mathrm{NL}_{j\ne k}$, we deduce that
\begin{align*}
  \int_{\TT}|\partial_\zeta \mathrm{NL}_{j\ne k}|\dd \zeta
  \leq & C \sum_{j\neq k} \int_{\TT} \int_{\TT}|\partial^2_\zeta Z_{k,j}(\zeta,\mu)|R(|Z_{k,j}(\zeta,\mu)|)|\dd \mu \dd \zeta \\
  & + C \sum_{j\neq k} \int_{\TT} \int_{\TT} |\partial_\zeta Z_{k,j}(\zeta,\mu)|^2|R'(|Z_{k,j}(\zeta,\mu)|)|\dd \mu \dd \zeta \\
  \leq &  C \lVert \mathbf{z}\rVert_{H^2} \big(\delta[\mathbf{z}]^{-\alpha} + \|\mathbf{z}\|_{L^\infty} \big)
  + C \lVert \mathbf{z}\rVert_{C^1}^2 \big(\delta[\mathbf{z}]^{-1-\alpha} + 1 \big).
\end{align*}
Hence, we obtain that
\begin{equation}\label{eq:lambda_k-infty-bound}
\begin{aligned}
  \|\lambda_k\|_{L^\infty} \leq & C \|F[z_k]\|_{L^\infty}  \|\mathbf{z}\|_{H^2}
  \big(\|\mathbf{z}\|_{H^2}^2 + 1  \big)\big( \lVert F[z_k]\rVert^{1+\alpha}_{L^\infty}
  +\delta[\mathbf{z}]^{-1-\alpha} + \|\mathbf{z}\|_{L^\infty} +  1\big),
\end{aligned}
\end{equation}
and thus the terms $N_{71}$ and $N_{72}$ are under control.
For the remaining term $N_{73}$, it follows from the Newton-Leibniz's formula
$\lambda_k(\zeta) -\lambda_k(\zeta-\eta) = \eta \int_0^1 \partial_\zeta \lambda_k(\zeta -s \eta) \dd s$ that
\begin{align*}
  \frac{|\eta|(z_k(\zeta)-z_k(\zeta-\eta))\cdot N_{73}}{|z_k(\zeta)-z_k(\zeta-\eta)|^3}
  \leq C \lVert F[z_k]\rVert^2_{L^\infty}
  \lVert z_k\rVert_{C^1}\lVert \partial_\zeta \lambda_k\rVert_{L^\infty}.
\end{align*}
In addition, $\lVert \partial_\zeta\lambda_k\rVert_{L^\infty}$ has already been estimated in the section \ref{subsec:es-tang}.

To conclude, gathering the estimates in the above subsections, we obtain the desired estimate \eqref{ineq:main-estimate}.

\subsection{Proof of Proposition \ref{prop:reg-param}}\label{subsec:reg-param}

From \eqref{eq:zkyk-exp} we get
\begin{align*}
  \partial_t z_k(\zeta,t)  = \partial_t y_k\big(\phi_k(\zeta,t),t\big) + \partial_\mu y_k(\phi_k(\zeta,t),t) \cdot \partial_t \phi_k(\zeta,t).
\end{align*}
By using the contour equations of $y_k(\mu,t)$ and $z_k(\zeta,t)$ as in \eqref{eq:contour}-\eqref{eq:main-eq-GP}
(exactly arguing as the computation in \cite[p.28]{GanP21}), we have
\begin{equation}\label{eq:phi-k-def}
\begin{aligned}	
  \partial_t \phi_k(\zeta,t)
  =& \sum_{j=1}^N a_j \int_{\TT} \big(\partial_\zeta \phi_k(\zeta,t) - \partial_\zeta \phi_j(\zeta-\eta,t)\big)
  R(|z_k(\zeta,t) - z_j(\zeta - \eta,t)|)\dd \eta \\	
  & + \sum_{j=1}^N a_j \int_{\TT} \big(\partial_\zeta \phi_k(\zeta,t) - \partial_\zeta \phi_j(\zeta -\eta,t)\big)
  R(|z_k(\zeta,t)-\overline{z}_j(\zeta-\eta,t)|)\dd \eta\\
  & + \lambda_k(\zeta)\, \partial_\zeta \phi_k(\zeta,t) \\
  \triangleq & \, H_1 + H_2 + H_3.
\end{aligned}
\end{equation}

From the equation \eqref{eq:phi-k-def} we have
\begin{align*}
  \frac{1}{2} \frac{\dd }{\dd t} \big(\lVert \phi_k-\zeta\rVert_{H^2}^2 + \|\phi_k\|_{H^2}^2 \big)
  = \int_{\TT} (\phi_k(\zeta)-\zeta) \partial_t\phi_k\dd \zeta
  + \int_{\TT} \phi_k(\zeta) \partial_t\phi_k\dd \zeta
  + 2 \int_{\TT} \partial^2_\zeta \phi_k(\zeta) \, \partial_t\partial^2_\zeta \phi_k(\zeta) \dd \zeta.
\end{align*}
It follows that
\begin{align*}
  \Big|\int_{\TT}(\phi_k(\zeta)-\zeta) \partial_t\phi_k\dd \zeta \Big|
  + \Big|\int_{\TT} \phi_k(\zeta)\,\partial_t\phi_k(\zeta) \dd \zeta \Big|
  \leq C \big(\lVert \phi_k(\zeta)-\zeta\rVert_{L^2} + \|\phi_k\|_{L^2} \big)
  \lVert \partial_t\phi_k\rVert_{L^\infty}.
\end{align*}
In addition, in view of \eqref{eq:lambda_k-infty-bound}, it is easy to see that
\begin{align*}
  \lVert \partial_t \phi_k \rVert_{L^\infty} & \leq C \big(\lVert F[z_k]\rVert^\alpha_{L^\infty}
  + \delta[\mathbf{z}]^{-\alpha} + \|\mathbf{z}\|_{L^\infty} + 1\big) \lVert \phi_k\rVert_{C^1}
  + C \lVert \lambda_k\rVert_{L^\infty} \lVert \phi_k\rVert_{C^1} \\
  & \leq C \|\phi_k\|_{C^1}
  \big(\|\mathbf{z}\|_{H^2}^3 + 1  \big)\big( \lVert F[z_k]\rVert^{1+\alpha}_{L^\infty}
  +\delta[\mathbf{z}]^{-1-\alpha} + \|\mathbf{z}\|_{L^\infty} +  1\big).
\end{align*}

Since the estimation of $H_1$ in \eqref{eq:phi-k-def} is much easier than that of $H_2$ in
\eqref{eq:phi-k-def}, we shall only deal with the last two terms in \eqref{eq:phi-k-def}.
Differentiating two times on the term $H_2$, we find
\begin{align*}
  \partial^2_\zeta H_2 = H_{21} + H_{22} + H_{23} + H_{24} + H_{25},
\end{align*}
where (recalling that $Z_{k,j} = Z_{k,j}(\zeta,\eta,t) = z_k(\zeta,t)- \overline{z}_j(\zeta - \eta,t)$)
\begin{align*}
  H_{21} &\triangleq \sum_{j=1}^N a_j \int_{\TT}\big(\partial^3_\zeta \phi_k(\zeta,t)
  - \partial^3_\zeta \phi_j(\zeta -\eta,t)\big) R(|Z_{k,j}|) \dd \eta, \\	
  H_{22} & \triangleq 2 \sum_{j=1}^N a_j \int_{\TT} \big(\partial^2_\zeta \phi_k(\zeta,t)
  - \partial^2_\zeta \phi_j(\zeta - \eta,t)\big)
  R'(|Z_{k,j}|) \frac{\partial_\zeta Z_{k,j}\cdot Z_{k,j}}{|Z_{k,j}|} \dd \eta, \\		
  H_{23} & \triangleq \sum_{j=1}^N a_j \int_{\TT}\big(\partial_\zeta \phi_k(\zeta,t)
  - \partial_\zeta \phi_j(\zeta - \eta,t)\big)
  \Big(R''(|Z_{k,j}|) - \frac{R'(|Z_{k,j}|)}{|Z_{k,j}|} \Big)
  \Big(\frac{\partial_\zeta Z_{k,j} \cdot Z_{k,j}}{|Z_{k,j}|}\Big)^2 \dd \eta, \\
  H_{24} & \triangleq \sum_{j=1}^N a_j \int_{\TT} \big(\partial_\zeta \phi_k(\zeta,t) - \partial_\zeta
  \phi_j(\zeta- \eta,t)\big) R'(|Z_{k,j}|) \frac{\partial^2_\zeta Z_{k,j} \cdot Z_{k,j}}{|Z_{k,j}|} \dd \eta, \\
  H_{25} & \triangleq \sum_{j=1}^N a_j \int_{\TT} \big(\partial_\zeta \phi_k(\zeta,t) - \partial_\zeta \phi_j(\zeta-\eta,t)\big)
  R'(|Z_{k,j}|)\frac{|\partial_\zeta Z_{k,j}|^2}{|Z_{k,j}|}\dd \eta.
\end{align*}
For $H_{21}$, we separately consider $H_{21}|_{j=k}$ (containing only $j=k$ term)
and $H_{21}|_{j\neq k}$ (the summation over all $j\neq k \in \{1,\cdots,N\}$).
By a symmetrization argument, we see that
\begin{align*}
  \int_{\TT}\partial^2_\zeta \phi_k(\zeta) \cdot H_{21}|_{j=k} \dd\zeta
  & =\frac{a_k}{4} \int_{\TT} \int_{\TT} \partial_\zeta
  \big(|\partial^2_\zeta \phi_k(\zeta) - \phi_k(\zeta-\eta)|^2\big) R(|Z_{k,k}|) \dd \eta \dd \zeta \\
  & = -\frac{a_k}{4} \int_{\TT} \int_{\TT} |\partial^2_\zeta \phi_k(\zeta) - \phi_k(\zeta-\eta)|^2
  R'(|Z_{k,k}|) \frac{Z_{k,k} \cdot \partial_\zeta Z_{k,k}}{|Z_{k,k}|} \dd \eta \dd \zeta,
\end{align*}
which gives
\begin{align*}
  \Big|\int_{\TT} \partial^2_\zeta \phi_k(\zeta) \cdot H_{21}|_{j=k} \dd\zeta \Big|
  \leq C \big( \|\mathbf{z}\|_{H^2} + 1\big)  \big(\lVert F[z_k] \rVert^{1+\alpha}_{L^\infty} + 1\big) \lVert \phi_k \rVert^2_{H^2}.
\end{align*}
On the other hand, for $H_{21}|_{j\neq k}$, we observe that
\begin{align*}
  \int_{\TT}\partial^2_\zeta \phi_k(\zeta)\cdot H_{21}|_{j\neq k} \dd \zeta
  = & \sum_{j\neq k} a_j \int_{\TT} \int_{\TT} \partial^2_\zeta \phi_k(\zeta)
  \cdot \partial^3_\zeta \phi_k(\zeta) R(|Z_{k,j}|)\dd \eta\dd \zeta \\
  & - \sum_{j\neq k} a_j \int_{\TT} \int_{\TT} \partial^2_\zeta \phi_k(\zeta) \cdot \partial^3_\zeta \phi_k(\zeta - \eta)
  R(|Z_{k,j}|) \dd \eta\dd \zeta \\
  = & - \frac{1}{2} \sum_{j\neq k} a_j \int_{\TT} \int_{\TT} | \partial^2_\zeta \phi_k(\zeta)|^2
  R'(|Z_{k,j}|) \frac{Z_{k,j}\cdot \partial_\zeta Z_{k,j}}{|Z_{k,j}|} \dd \eta\dd \zeta \\
  & - \sum_{j\neq k} a_j \int_{\TT} \int_{\TT} \partial^2_\zeta \phi_k(\zeta) \cdot
  \partial^2_\zeta \phi_k(\zeta - \eta) R'(|Z_{k,j}|) \frac{Z_{k,j}\cdot \partial_\eta Z_{k,j}}{|Z_{k,j}|}  \dd \eta \dd \zeta,
\end{align*}
which yields
\begin{align*}
  \Big|\int_{\TT}\partial^2_\zeta \phi_k(\zeta)\cdot H_{21}|_{j\neq k} \dd \zeta \Big|
  \leq \, C \big( \delta[\mathbf{z}]^{-1-\alpha} + 1 \big) \|\mathbf{z}\|_{H^2} \lVert \phi_k\rVert^2_{H^2} .
\end{align*}
The term $H_{22}$ can be bounded in a similar way, namely,
\begin{align*}
  \Big|\int_{\TT} \partial^2_\zeta \phi_k(\zeta) \cdot H_{22} \dd\zeta\Big|
  \leq C \big(\lVert F[z_k] \rVert^{1+\alpha}_{L^\infty} + \delta[\mathbf{z}]^{-1-\alpha}  + 1\big) \big(\|\mathbf{z}\|_{H^2} + 1 \big)
  \lVert \phi_k \rVert^2_{H^2}.
\end{align*}
For $H_{23}$ and $H_{25}$, by applying Lemma \ref{lem:G-lemma} and \eqref{eq:Zkk-fact}, \eqref{eq:diff-pointwise-es}, we have
\begin{align*}
  & \Big|\int_{\TT} \partial^2_\zeta \phi_k(\zeta) \cdot \big(H_{23} + H_{25} \big) \dd \zeta\Big| \\
  & \leq C \sum_{j=1}^N \int_0^1 \int_{\TT} \int_{\TT} |\partial^2_\zeta \phi_k(\zeta)| |\partial^2_\zeta
  \phi_k(\zeta -s\eta)| |\eta| |\partial_\zeta Z_{k,j}|^2 \Big( |R''(|Z_{k,j}|)|
  + \frac{|R'(|Z_{k,j}|)|}{|Z_{k,j}|} \Big) \dd \eta\dd \zeta \dd s \\
  & \leq C \big(\|\mathbf{z}\|_{H^2}^2 + 1 \big)
  \big(\lVert F[z_k]\rVert^{\alpha+2}_{L^{\infty}} + \delta[\mathbf{z}]^{-2-\alpha} + 1  \big) \lVert \phi_k\rVert^2_{H^2}.
\end{align*}
For $H_{24}$, using the fact that $H^2\hookrightarrow C^{1,\frac{1}{2}}$,
we obtain that
\begin{align*}
  \Big| \int_{\TT} \partial^2_\zeta \phi_k(\zeta) \cdot H_{24} \dd \zeta \Big|
  \leq &  C \lVert \partial_\zeta \phi_k \rVert_{C^{\frac{1}{2}}}
  \sum_{j=1}^N \int_{\TT} \int_{\TT} |\partial^2_\zeta \phi_k(\zeta)| \big|\partial^2_\zeta Z_{k,j}(\zeta,\eta)\big| |\eta|^{\frac{1}{2}} \,|R'(|Z_{k,j}|)|
  \dd \eta \dd \zeta \\
  \leq & C \lVert \phi_k\rVert^2_{H^2} \|\mathbf{z}\|_{H^2}
  \big( \lVert F(z)\rVert^{\alpha+1}_{L^\infty} + \delta[\mathbf{z}]^{-1-\alpha} + 1 \big).
\end{align*}

Next, we move to the estimation of the last term $H_3 = \lambda_k(\zeta) \partial_\zeta \phi_k(\zeta)$.
By Leibniz's rule, we have
\begin{align*}
  \int_{\TT}\partial^2_\zeta \phi_k(\zeta)\cdot & \, \partial^2_\zeta \big(\lambda_k(\zeta) \partial_\zeta \phi_k(\zeta)\big) \dd \zeta  =
  \int_{\TT}\partial^2_\zeta \phi_k(\zeta)\cdot \partial_\zeta \phi_k(\zeta) \, \partial^2_\zeta \lambda_k(\zeta)\dd \zeta \\
  & \quad + 2 \int_{\TT} \partial^2_\zeta \phi_k(\zeta)\cdot \partial^2_\zeta \phi_k(\zeta)\,\partial_\zeta \lambda_k(\zeta)\dd \zeta 
  + \int_{\TT} \partial^2_\zeta \phi_k(\zeta) \cdot \partial^3_\zeta \phi_k(\zeta)\, \lambda_k(\zeta) \dd \zeta \\
  & \triangleq H_{31} + H_{32} + H_{33}.
\end{align*}
Noticing that
\begin{align*}
  H_{33} = \frac{1}{2}\int_{\TT} \partial_\zeta \big(|\partial^2_\zeta \phi_k(\zeta)|^2\big)  \lambda_k(\zeta) \dd \zeta
  = - \frac{1}{2} \int_{\TT} |\partial^2_\zeta \phi_k(\zeta)|^2 \partial_\zeta \lambda_k(\zeta)\dd \zeta,
\end{align*}
and recalling the estimate of $\lVert \partial_\zeta \lambda_k\rVert_{L^\infty}$ in \eqref{eq:par-lambda}, we have
\begin{align*}
  |H_{32}|+|H_{33}| \leq C \lVert \phi_k\rVert^2_{H^2}
  \|F[z_k]\|_{L^\infty}
  \|\mathbf{z}\|_{H^2} \big(\|\mathbf{z}\|_{H^2} +1 \big)
  \big(\|F[z_k]\|_{L^\infty}^{1+\alpha} + \delta[\mathbf{z}]^{-1-\alpha} + \|\mathbf{z}\|_{L^\infty}
  + 1 \big).
\end{align*}
In the sequel, we focus on the estimation of $H_{31}$. By differentiating $\lambda_k(\zeta)$ given by \eqref{eq:lambda-def}, we have
\begin{align*}
  \partial^2_\zeta \lambda_k(\zeta) =  - \partial_\zeta \Big(\frac{\partial_\zeta z_k(\zeta)
  \cdot \partial_\zeta \mathrm{NL}_k (\zeta)}{A_k(t)} \Big)
  = \partial_\zeta \Big(  C_1 + C_2 + C_3 + C_4 \Big),
\end{align*}
where
\begin{align*}
  C_1 \triangleq & - \sum_{j=1}^N a_j \frac{\partial_\zeta z_k(\zeta)}{A_k(t)} \cdot
  \int_{\TT}\partial^2_\zeta Z_{k,j}(\zeta,\eta,t) R(|Z_{k,j}(\zeta,\eta,t)|) \dd \eta,
\end{align*}
\begin{align*}
  C_2 \triangleq & - \sum_{j=1}^N a_j \frac{\partial_\zeta z_k(\zeta)}{A_k(t)} \cdot
  \int_{\TT}\partial_\zeta Z_{k,j}(\zeta,\eta,t) R'(|Z_{k,j}(\zeta,\eta,t)|)
  \frac{Z_{k,j}\cdot \partial_\zeta Z_{k,j}}{|Z_{k,j}|}\dd \eta,
\end{align*}
and $C_3$ and $C_4$ are those non-conjugate terms (with $Z_{k,j}$ replaced by $\mathbf{Z}_{k,j}$ in definitions of $C_1$ and $C_2$).
We only deal with those conjugate terms $C_1$ and $C_2$ since the other terms are much easier.
For  $C_1$, using the fact that $\partial_\zeta z_k(\zeta) \cdot \partial^2_\zeta z_k(\zeta) = 0$, we have
\begin{align*}
  \partial_\zeta C_1 = \partial_\zeta \bigg(\sum_{j=1}^N a_j \frac{\partial_\zeta z_k(\zeta)}{A_k(t)} \cdot
  \int_{\TT}\partial^2_\zeta \overline{z}_j(\zeta - \eta,t) R(|Z_{k,j}(\zeta,\eta,t)|) \dd \eta \bigg)
  = C_{11} + C_{12} + C_{13},
\end{align*}
where
\begin{align*}
  C_{11} & \triangleq \sum_{j=1}^N a_j \frac{\partial^2_\zeta z_k(\zeta)}{A_k(t)}\cdot
  \int_{\TT}\partial^2_\zeta \overline{z}_j(\zeta - \eta,t) R(|Z_{k,j}(\zeta,\eta,t)|) \dd \eta,\\
  C_{12} & \triangleq  \sum_{j=1}^N a_j \frac{\partial_\zeta z_k(\zeta)}{A_k(t)}\cdot
  \int_{\TT}\partial^2_\zeta \overline{z}_j (\zeta-\eta,t) R'(|Z_{k,j}(\zeta,\eta,t)|)
  \frac{Z_{k,j} \cdot \partial_\zeta Z_{k,j}}{|Z_{k,j}|}\dd \eta, \\
  C_{13} & \triangleq \sum_{j=1}^N a_j \frac{\partial_\zeta z_k(\zeta)}{A_k(t)}\cdot
  \int_{\TT}\partial^3_\zeta \overline{z}_j(\zeta - \eta,t) R(|Z_{k,j}(\zeta,\eta,t)|)\dd \eta.
\end{align*}
The term $C_{11}$ can be treated as
\begin{align*}
  \lVert C_{11}\rVert_{L^2} & \leq C \lVert z_k\rVert_{H^2}\|F[z_k]\|_{L^\infty}^2 \sum_{j=1}^N
  \Big\lVert \int_{\TT} \partial^2_\zeta \overline{z}_j(\zeta - \eta,t) R(|Z_{k,j}(\zeta,\eta,t)|) \dd \eta
  \Big\rVert_{L^\infty} \\
  & \leq C \lVert z_k\rVert^2_{H^2} \lVert F[z_k]\rVert^2_{L^\infty}
  \big( \|F[z_k]\|_{L^\infty}^\alpha + \delta[\mathbf{z}]^{-\alpha} + \|\mathbf{z}\|_{L^\infty} \big).
\end{align*}
For $C_{12}$, using \eqref{eq:Zkk-fact} and \eqref{eq:diff-pointwise-es}, we infer that
\begin{align*}
  \lVert C_{12}\rVert_{L^2} \leq C  \|F[z_k]\|_{L^\infty}  \lVert \mathbf{z} \rVert_{H^2} \big(\|\mathbf{z}\|_{H^2} + 1 \big)
  \big(\|F[z_k]\|_{L^\infty}^{1+\alpha} + \delta[\mathbf{z}]^{-1-\alpha}  + 1 \big).
\end{align*}
For $C_{13}$, we integrate by parts to obtain that
\begin{align*}
  \lVert C_{13} \rVert_{L^{2}} & \leq C \sum_{j=1}^N  \Big \lVert\frac{\partial_\zeta z_k(\zeta)}{A_k(t)}\cdot
  \int_{\TT}\partial_\eta \partial^2_\zeta \overline{z}_j (\zeta-\eta,t)
  R(|Z_{k,j}(\zeta,\eta,t)|)\dd \eta \Big\rVert_{L^2} \\
  & = C \sum_{j=1}^N \Big \lVert\frac{\partial_\zeta z_k(\zeta)}{A_k(t)}\cdot
  \int_{\TT}\partial^2_\zeta \overline{z}_j(\zeta-\eta,t) R'(|Z_{k,j}(\zeta,\eta,t)|)
  \frac{Z_{k,j} \cdot \partial_\zeta \overline{z}_j(\zeta-\eta)}{|Z_{k,j}|} \dd \eta \Big\rVert_{L^2} \\
  & = C \Big \lVert\int_{\TT}\frac{\big(\partial_\zeta z_k(\zeta) - \partial_\zeta \overline{z}_k(\zeta-\eta)\big)\cdot
  \partial^2_\zeta \overline{z}_k(\zeta-\eta,t)}{A_k(t)} R'(|Z_{k,k}|)
  \frac{Z_{k,k}\cdot \partial_\zeta \overline{z}_k(\zeta-\eta)}{|Z_{k,k}|}\dd \eta \Big\rVert_{L^2} \\
  & \quad + C \sum_{j\neq k} \Big \lVert\frac{\partial_\zeta z_k(\zeta)}{A_k(t)}\cdot
  \int_{\TT}\partial^2_\zeta \overline{z}_j(\zeta-\eta,t) R'(|Z_{k,j}|)
  \frac{Z_{k,j} \cdot \partial_\zeta \overline{z}_j(\zeta-\eta)}{|Z_{k,j}|} \dd \eta \Big\rVert_{L^2}  \\
  & \leq C \lVert \mathbf{z}\rVert_{H^2}^2 \big(\|\mathbf{z}\|_{H^2} + 1 \big) \big(\|F[z_k]\|_{L^\infty}^2 + \|F[z_k]\|_{L^\infty} \big)
  \big(\|F[z_k]\|_{L^\infty}^{1+\alpha} + \delta[\mathbf{z}]^{-1-\alpha} + 1 \big),
\end{align*}
where in the third line we have used the fact that  $\partial_\zeta \overline{z}_k(\zeta - \eta)\cdot \partial^2_\zeta \overline{z}_k(\zeta-\eta)=0$.
	
Next, for $\partial_\zeta C_2$, we have
\begin{align*}
  \partial_\zeta C_2 & = - \sum_{j=1}^N a_j \frac{\partial_\zeta^2 z_k(\zeta)}{A_k(t)} \cdot
  \int_{\TT}\partial_\zeta Z_{k,j}(\zeta,\eta,t) R'(|Z_{k,j}|)
  \frac{Z_{k,j}\cdot \partial_\zeta Z_{k,j}}{|Z_{k,j}|}\dd \eta \\
  & \quad - \sum_{j=1}^N a_j \frac{\partial_\zeta z_k(\zeta)}{A_k(t)} \cdot
  \int_{\TT}\partial_\zeta^2 Z_{k,j}(\zeta,\eta,t) R'(|Z_{k,j}|)
  \frac{Z_{k,j}\cdot \partial_\zeta Z_{k,j}}{|Z_{k,j}|}\dd \eta \\
  & \quad - \sum_{j=1}^N a_j \frac{\partial_\zeta z_k(\zeta)}{A_k(t)} \cdot
  \int_{\TT}\partial_\zeta Z_{k,j}(\zeta,\eta,t) R''(|Z_{k,j}|)
  \Big(\frac{Z_{k,j}\cdot \partial_\zeta Z_{k,j}}{|Z_{k,j}|}\Big)^2 \dd \eta \\
  & \quad - \sum_{j=1}^N a_j \frac{\partial_\zeta z_k(\zeta)}{A_k(t)} \cdot
  \int_{\TT}\partial_\zeta Z_{k,j}(\zeta,\eta,t) R'(|Z_{k,j}|)
  \bigg( \partial_\zeta \Big(\frac{Z_{k,j}}{|Z_{k,j}|}\Big)\cdot \partial_\zeta Z_{k,j} 
  + \frac{Z_{k,j}\cdot \partial_\zeta^2 Z_{k,j}}{|Z_{k,j}|} \bigg) \dd \eta \\
  & \triangleq C_{21} + C_{22} + C_{23} + C_{24} .
\end{align*}
For the term $C_{23}$,  by using the fact that (similarly as deriving \eqref{eq:id-diff-1})
\begin{align*}
  \partial_\zeta z_k(\zeta) \cdot \partial_\zeta Z_{k,k}(\zeta,\eta)
  = \partial_\zeta z_k(\zeta) \cdot (\partial_\zeta z_k(\zeta) - \partial_\zeta \overline{z}_k(\zeta - \eta))
  = \frac{1}{2}| \partial_\zeta z_k(\zeta) - \partial_\zeta z_k(\zeta - \eta)|^2,
\end{align*}
we deduce that
\begin{align*}
  \|C_{23}\|_{L^2} & \leq C \|F[z_k]\|_{L^\infty}^2
  \Big\|\int_{\TT} |\partial_\zeta z_k(\zeta) - \partial_\zeta z_k(\zeta-\eta)|^2 \,
  |R''(|Z_{k,k}|)|\, |\partial_\zeta Z_{k,k}|^2 \dd \eta \Big\|_{L^2} \\
  &  \quad + C \|F[z_k]\|_{L^\infty} \|\mathbf{z}\|_{H^2}^3 \big(\delta[\mathbf{z}]^{-2-\alpha} + 1 \big) \\
  & \leq C \big(\|F[z_k]\|_{L^\infty}^2 + \|F[z_k]\|_{L^\infty}\big) \|\mathbf{z}\|_{H^2}^2 \big(\|\mathbf{z}\|_{H^2}^2 +1 \big)
   \big(\|F[z_k]\|_{L^\infty}^{2+\alpha} + \delta[\mathbf{z}]^{-2-\alpha} + 1 \big).
\end{align*}
The remaining terms $C_{21}$, $C_{22}$, and $C_{24}$ can be estimated as above (in a much easier way), thus we omit the details.
Then we conclude that
\begin{align*}
  \lVert \partial_\zeta C_2 \rVert_{L^2} + \|\partial_\zeta^2 \lambda_k\|_{L^\infty} \leq  C ,
\end{align*}
with $C>0$ depending on $\|\mathbf{z}\|_{H^2}$, $\delta[\mathbf{z}]^{-1}$ and $\|F[z_k]\|_{L^\infty}$.

Therefore, collecting the above estimates, we obtain the desired result
\begin{align*}
  \frac{\dd }{\dd t}\Big(\sum_{k=1}^N \big(\lVert \phi_k-\zeta\rVert^2_{H^2} + \lVert \phi_k\rVert^2_{H^2} \big) \Big)
  \leq C  \sum_{k=1}^N \big(\lVert \phi_k -\zeta \rVert^2_{H^2} + \lVert \phi_k\rVert^2_{H^2}\big).
\end{align*}

\bigskip
\subsection*{Acknowledgments}
The authors were partially supported by National Key Research and Development Program of China No. 2020YFA0712900 (LX and ZX), 
National Natural Science Foundation of China Nos. 12371199 (QM), 12271045 (LX), and  National Science Foundation Grants DMS-2108264, 
DMS-2238219 (CT).
ZX was supported by a scholarship (No. 202306040124) of China Scholarship Council. Part of the work was done when ZX was a visiting Ph.D. student from Oct. 2023 to Oct. 2024 in University of Seville, and he would like to express his deep gratitude for the hospitality by University of Seville.

\end{document}